\documentclass[11pt,a4paper]{article} 

\usepackage[english]{babel}
\usepackage[utf8]{inputenc}
\usepackage{lmodern}
\usepackage[T1]{fontenc}

\usepackage[a4paper,top=3cm,bottom=2cm,left=2.5cm,right=2.5cm,marginparwidth=1.75cm]{geometry}

\usepackage{amsmath}
\usepackage{graphicx}
\usepackage{xcolor}
\usepackage[hidelinks]{hyperref} 
\usepackage{amscd}     
\usepackage{amsfonts}
\usepackage{dsfont}
\usepackage{latexsym}
\usepackage{amscd}
\usepackage{verbatim}
\usepackage{graphicx}
\usepackage{mathtools} 
\usepackage{tikz} 
\usepackage{tikz-cd}
\usepackage{algorithm}
\usepackage{algpseudocode}
\usepackage{blkarray}
\usepackage[toc,page]{appendix}
\usepackage{booktabs}

\newcommand{\tikzmark}[2]{\tikz[overlay,remember picture] \node[anchor=base] (#1) {#2};}

\usepackage{amsmath}
	\numberwithin{equation}{section}
\usepackage{amssymb}
\usepackage{amsthm}
	\theoremstyle{plain}
		\newtheorem{thm}{Theorem}[section]
		\newtheorem{coro}[thm]{Corollary}
		\newtheorem{lem}[thm]{Lemma}
		\newtheorem{prop}[thm]{Proposition}
	\theoremstyle{definition}
		\newtheorem{defn}[thm]{Definition}
		\newtheorem{ex}[thm]{Example}
        \newtheorem{convention}[thm]{Convention}

	\theoremstyle{remark}
		\newtheorem{rmk}[thm]{Remark}

\usepackage{mparhack}

\usepackage{caption}
\usepackage{subcaption}

\newcommand{\ns}{\mathcal{S}} 

\renewcommand{\=}{\coloneqq}	 

\DeclareMathOperator{\image}{im}
\DeclareMathOperator{\id}{id}
\DeclareMathOperator{\coker}{coker}	
\DeclareMathOperator{\Tame}{\mathsf{Tame}}
\DeclareMathOperator{\Vect}{vect}
\DeclareMathOperator{\rank}{rank}

\DeclareMathOperator{\Dgm}{Dgm}

\def\Xgenlabel{\mathcal{G}_X}
\def\Zgenlabel{\mathcal{G}_Z}
\def\Xrellabel{\mathcal{R}_X}

\title{Algebraic Wasserstein distances and stable homological invariants of data}

\author{Jens Agerberg \and Andrea Guidolin \and Isaac Ren \and Martina Scolamiero}

\date{}

\makeindex

\begin{document}

\maketitle

\begin{abstract}
Distances have a ubiquitous role in persistent homology, from the direct comparison of homological representations of data to the definition and optimization of invariants. In this article we introduce a family of parametrized pseudometrics between persistence modules based on the algebraic Wasserstein distance defined by Skraba and Turner,
and phrase them in the formalism of noise systems. This is achieved by comparing $p$-norms of cokernels (resp.\ kernels) of monomorphisms (resp.\ epimorphisms) between persistence modules and corresponding bar-to-bar morphisms, a novel notion that allows us to bridge between algebraic and combinatorial aspects of persistence modules. We use algebraic Wasserstein distances to define invariants, called Wasserstein stable ranks, which are 1-Lipschitz stable with respect to such pseudometrics. We prove a low-rank approximation result for persistence modules which allows us to efficiently compute Wasserstein stable ranks, and we propose an efficient algorithm to compute the interleaving distance between them. Importantly, Wasserstein stable ranks depend on interpretable parameters which can be learnt in a machine learning context. Experimental results illustrate the use of Wasserstein stable ranks on real and artificial data and highlight how such pseudometrics could be useful in data analysis tasks.\\


\end{abstract}

\tableofcontents

\subsection*{Acknowledgements}
This work was partially supported by the Swedish Research Council (Vetenskapsrådet), the Wallenberg AI, Autonomous Systems and Software Program (WASP) funded by the Knut and Alice Wallenberg Foundation, the dBrain collaborative project at Digital Futures at KTH, the Strategic Support grant of the Digitalisation Platform at KTH, and by the Data Driven Life Science (DDLS) program funded by the Knut and Alice Wallenberg Foundation.


\section{Introduction}
While Topological Data Analysis (TDA) has historically focused on studying the global shape of data, persistent homology has since grown to provide popular techniques for incorporating both global topological features and local geometry into data analysis pipelines \cite{adams2021topology}. Through the lens of persistent homology, global topological features can be encoded by long bars in a barcode decomposition of the persistence module, while local geometric features are characterized by short bars in the barcode. Indeed, both the information of long bars and short bars in the barcode \cite{bendich2016persistent,hiraoka2016hierarchical}, as well as their location along the filtration scale \cite{stolz2017persistent,chacholski2020metrics, agerberg2021supervised}, turn out to be relevant in data analysis tasks.
Introduced to persistent homology in \cite{cohen2010lipschitz}, Wasserstein distances offer a way to determine a trade-off between global and local features in persistence. Such distances have been widely used in applications and have been studied both from a combinatorial perspective and more recently with an algebraic approach \cite{bubenik2022exact, skraba2020wasserstein}.
Wasserstein distances are parametrized by two parameters in $[1,\infty]$, commonly fixed to the values of $1$, $2$, and $\infty$. 
One of the aims of this article is to define a richer family of parametrized Wasserstein distances where, in addition to standard parameters  determining sensitivity to short bars globally in the parameter space, a \emph{contour} is introduced to locally weight different parts of the parameter space. 
We propose that the optimal parameter values for a particular task should be learned in a machine learning context. Our contribution is part of more general efforts of identifying parametrized families of metrics and invariants for persistence \cite{bubenik2015metrics,scolamiero2017multidimensional,hofer2017deep,zhao2019learning,carriere2020perslay}. 

\paragraph{Algebraic Wasserstein distances.}
The study of algebraic distances between persistence modules is an active research direction in TDA, as demonstrated by the recent works on amplitudes  \cite{giunti2021amplitudes} and exact weights \cite{bubenik2022exact}. 
In this article we provide a new proof that the $p$-norm of a persistence module, introduced in \cite{skraba2020wasserstein}, defines a pseudometric for all $p\in [1,\infty]$. While \cite{skraba2020wasserstein} constructs a correspondence between the pseudometric induced by the $p$-norm and the Wasserstein distance between persistence diagrams, 
our proof shows that the $p$-norm determines a noise system \cite{scolamiero2017multidimensional} and therefore an induced pseudometric.
Our approach easily generalizes to define new pseudometrics on persistence modules, as for example the pseudometrics $d^q_{\ns^{p,C}}$ that combine $p$-norms with contours, effectively used as a reparametrization of the parameter space $[0,\infty)$. From a technical perspective, our framework requires to prove the axioms of noise systems without assuming that the $p$-norm induces a pseudometric (including the triangular inequality property) but rather studying how the $p$-norm interacts with monomorphisms, epimorphism, and short exact sequences. Among the axioms of noise systems, the one on short exact sequences (Lemma \ref{lem:ns_ses2}) is difficult to prove with our assumptions and to this purpose we introduce bar-to-bar morphisms, explained below.

It is  interesting to see that  Wasserstein distances fit in the noise system framework, as they are fundamentally different from noise systems that have been studied from a computational perspective  so far.
In fact, algorithms for the computation of stable ranks (that can be seen as vectorizations of persistence modules depending on the noise system) were only developed for so called \emph{simple noise systems}  \cite{gafvert2017stable,chacholski2020metrics}. 
These noise systems have the extra property of being closed under direct sums, and can intuitively be thought of as being sensitive only to the longest bars, which leads to $L^{\infty}$-type distances.
The noise systems associated with algebraic Wasserstein distances for $p<\infty$ are of a different nature, and in particular they are not closed under direct sums.

From a practical and computational perspective, combinatorial distances between persistence diagrams are more straightforward to compute than algebraic distances between persistence modules.
The combinatorial $(p, C)$-Wasserstein 
distances associated to  $d^p_{\ns^{p,C}}$  in Section \ref{subsec:alg_and_comb_WpC} offer a convenient way to compute contour distances and the combination of contour and Wasserstein distances between persistence modules, relying on the already developed computational machinery for Bottleneck $(p=\infty)$ and Wasserstein distances between persistence diagrams. 
In this article, however, our focus is not on the computation of the Wasserstein distance between two given persistence modules, but on invariants called Wasserstein stable ranks defined and computed using the distances.

\paragraph{Bar-to-bar morphisms.} The approach carried out in this article for proving that $p$-norms of persistence modules satisfy the axioms of noise systems relies on comparing monomorphisms (resp.\ epimorphisms) between persistence modules and so-called \emph{bar-to-bar} monomorphisms (resp.\ epimorphisms) between the same persistence modules.
Intuitively, in a bar-to-bar morphism (see Definition \ref{def:bartobar-morph}) every bar in the barcode decomposition of the domain maps non-trivially to at most one bar in the barcode decomposition of the codomain. 
Bar-to-bar morphisms are thus much simpler than general morphisms of persistence modules, and we show that they can be used to effectively reduce algebraic problems to easier problems of combinatorial nature. In particular, an important problem related to the definition and construction of algebraic distances is the minimization of kernels and cokernels of morphisms (see e.g.\ Definition \ref{def:dpS}) with respect to a chosen notion of ``size'', which in this article is the $p$-norm of persistence modules (Section \ref{subsec:pnorms}) or a more general notion of norm combining $p$-norms and contours (Definition \ref{def:pCnorm}).
Our main theoretical results Theorem~\ref{thm:mono-bar} and Theorem~\ref{thm:epi-bar} state that for any monomorphism (resp.\ epimorphism) between two persistence  modules  there exists a bar-to-bar monomorphism (resp.\ epimorphism)
between the same persistence modules whose cokernel (resp.\ kernel) has smaller or equal norm.

Various types of bar-to-bar morphisms can be constructed. For example, as we observe in Section \ref{subsec:canonical}, there are bar-to-bar monomorphisms and epimorphisms associated with the induced matchings of \cite{bauer2015induced}. Bar-to-bar morphisms are however more general then the induced matchings, and are also fundamentally different from other notions of matchings, such as the sub-barcode matchings of \cite{chubet2022theory}.
However, our bar-to-bar morphisms can be used as a tool to prove that the monomorphism (resp.\ epimorphism) associated to the induced matching has the cokernel (resp.\ kernel) with minimal $p$-norm among all monomorphisms (resp.\ epimorphisms) with the same domain and codomain.
(Corollary \ref{coro:mono-canonical} and Corollary \ref{coro:epi-canonical}). 

Several other key results of this article are proven leveraging Theorem~\ref{thm:mono-bar} and Theorem~\ref{thm:epi-bar}. For example, we use these theorems to show that $p$-norms of persistence modules satisfy the  axiom of noise systems on short exact sequences (Lemma \ref{lem:ns_ses2}). 
We also use our main results of Section \ref{sec:mono_pnorm} to prove a low-rank approximation result for persistence modules, analogous for example to the Eckart-Young-Mirsky theorem in the context of matrices (compare e.g.\ with \cite[Sect.\ 1]{golub1987generalization} and references therein), where the notion of rank we use for persistence modules is the number of bars in the barcode decomposition. Given a persistence module $X$ of rank $k$, for every $r\le k$ we identify a persistence module of rank $r$ that is closest to $X$ in algebraic Wasserstein distance, and we express its distance from $X$ (Proposition \ref{prop:d_X_deleting_bars} and Proposition \ref{prop:closest_rank_C}).
These results and their generalization (Proposition \ref{prop:closest_rank_C}) allow us to compute Wasserstein stable ranks (see Proposition \ref{prop:computing_sr}), the invariants that we introduce in this article.

\paragraph{Wasserstein stable ranks, a class of learnable vectorizations.} The computation of Wasserstein distances between persistence modules remains  expensive despite recent progress \cite{kerber2017geometry}, and the space of persistent modules is not directly amenable to statistical methods and machine learning. For these reasons, feature maps from persistence modules or diagrams have become an important component of the TDA machine learning pipeline. These techniques introduce a map between the space of persistence modules and a vector space where statistical and machine learning methods are well-developed. 
We propose a new class of feature maps, directly related to the Wasserstein distances $d^q_{\ns^{p,C}}$ between persistence modules, and with interpretable, learnable parameters.
Having fixed a pseudometric in the family of Wasserstein distances $d^q_{\ns^{p,C}}$, the Wasserstein stable rank of a persistence module with respect to the chosen pseudometric can be explicitly computed with a formula (Proposition \ref{prop:computing_sr}) derived from our results on monomorphisms and epimorphisms. The computational complexity of determining the Wasserstein stable rank is $O(n \log n)$ in the number $n$ of bars of a persistence module.

A parametrized family of stable ranks can be obtained by varying the Wasserstein distances, opening up for the possibility to tune parameters for a particular task, resulting in feature maps that focus on the discriminative aspects of the persistence modules in a dataset. Previous learnable feature maps \cite{hofer2017deep, carriere2020perslay, reinauer2021persformer} make the choice of \textit{expressiveness} (being able to learn any arbitrary function on the space of persistence modules) over \textit{stability} (learning a function under the constraint that it is robust to perturbations of the input). Moreover, since the methods are often parametrized by complex neural networks, it is difficult to compare and interpret parametrizations learned for different tasks. Our Wasserstein stable ranks are stable by construction. More precisely, the interleaving distance between Wasserstein stable ranks is  $1$-Lipschitz with respect to the corresponding Wasserstein distance used in its construction. Similarly to Wasserstein stable ranks, we also provide a simple formula for computing the interleaving distance between them at the cost of $O(n \log n)$ in the maximum number  of bars in the two persistence modules we are comparing.

We use a metric learning framework to learn an optimal parametrization for a problem at hand, observe that a better model can be obtained by jointly optimizing the parameters $p$ and the ones related to the contour $C$ and illustrate that the output can be readily interpreted in terms of the learned parametrization focusing on e.g.\ global/local features or various parts of the filtration scale. The methods are demonstrated on a synthetic and a real-world datasets.

\paragraph{Outline of the paper.}  
Section \ref{sec:prelim} contains background material.
In Section \ref{sec:mono_pnorm} we prove results on the $p$-norm of the cokernel of a monomorphism and, dually, of the kernel of an epimorphism of persistence modules. 
Section \ref{sec:Wass_noise} is a study of Wasserstein distances and their generalizations involving contours in the framework of noise systems.
In Section \ref{sec:stable_rank} we compute Wasserstein stable ranks and interleaving distances between them, which we use to formulate a metric learning problem.
In Section \ref{sec:examples_analyses} we illustrate the use of Wasserstein stable ranks on synthetic and real-world data, learning optimal parameters of algebraic Wasserstein distances.


\section{Preliminaries}
\label{sec:prelim}

\subsection{Persistence modules and persistent homology}
\label{subsec:persmod}
Let $[0,\infty )$ denote the totally ordered set of nonnegative real numbers, regarded as the category induced by the order structure. We consider an arbitrary fixed field $K$ and denote by $\Vect_\mathit{K}$ the category of finite dimensional vector spaces over $K$. A \textbf{persistence module over $K$} is a functor $X\colon [0,\infty ) \to \Vect_\mathit{K}$. Explicitly, $X$ consists of a collection of finite dimensional vector spaces $X_t$ for all $t$ in $[0,\infty)$, together with a collection of linear functions $X_{s\le t}\colon X_s \to X_t$, called \textbf{transition functions}, for all $s\le t$ in $[0,\infty)$, such that  $X_{s\le t} X_{r\le s}=X_{r\le t}$ for all $r\le s\le t$, and $X_{t\le t}$ is the identity function on $X_t$ for all $t$ in $[0,\infty)$.
A \textbf{morphism} or natural transformation $f\colon X\to Y$ between two persistence modules $X$ and $Y$ is a collection of linear functions $f_t\colon X_t \to Y_t$, for all $t$ in $[0,\infty )$, such that $f_t X_{s\le t}= Y_{s\le t} f_s$ for all $s\le t$ in $[0,\infty )$.

A persistence module $X$ is \textbf{tame} if there exist real numbers $0=t_0<t_1<\cdots <t_k$ such that the transition function $X_{s\le t}$ is a non-isomorphism only if $s<t_i \le t$ for some $i\in \{ 1,\ldots ,k\}$.  
We denote by $\Tame$ the category of tame persistence modules and morphisms between them. The class of objects of this category will be denoted by $\Tame$ as well.

\begin{convention}
In this article we always work in the category of tame persistence modules over a fixed field $K$. For brevity the term persistence module will be used to refer to tame persistence modules over $K$.
\end{convention}

A morphism $f\colon X\to Y$ in $\Tame$ is a monomorphism (respectively, an epimorphism or isomorphism) if the linear functions $f_t\colon X_t \to Y_t$ are monomorphisms (respectively,  epimorphisms or isomorphisms) of vector spaces, for all $t$ in $[0,\infty )$. 
Kernels, cokernels and direct sums in $\Tame$ are defined componentwise. For example, for any persistence modules $X$ and $Y$, the direct sum $X\oplus Y$ is the persistence module defined by $(X\oplus Y)_t = X_t\oplus Y_t$ and  $(X\oplus Y)_{s\le t} = X_{s\le t}\oplus Y_{s\le t}$, for all $s\le t$ in $[0,\infty )$. The zero persistence module or \textbf{zero module}, i.e., the functor identically equal to the zero vector space on objects,  will be denoted by $0$.

Let $a<b$ in $[0,\infty]$. We denote by $K(a,b)$ the persistence module defined as follows: for any $t$ in $[0,\infty )$,
\begin{equation*}
K(a,b)_t \coloneqq   \begin{cases}
K &\text{if $a \le t < b$}\\
0 &\text{otherwise} ,
\end{cases}
\end{equation*}
and for any $s\le t$ in $[0,\infty )$,
\begin{equation*}
K(a,b)_{s\le t} \coloneqq   \begin{cases}
\id_K &\text{if $K(a,b)_s= K = K(a,b)_t$}\\
0 &\text{otherwise} .
\end{cases}
\end{equation*}
We call $K(a,b)$ the \textbf{bar} (or interval module) with \textbf{start-point} $a$ and \textbf{end-point} $b$. We say that the bar $K(a,b)$ is \textbf{infinite} if $b=\infty$ and \textbf{finite} otherwise. We say that the left-closed, right-open interval $[a,b)$ in $[0,\infty)$ is the \textbf{support} of the bar $K(a,b)$.
As an easy consequence of naturality, a morphism $f\colon  K(a_1,b_1)\to K(a_2,b_2)$ between bars can be nonzero (i.e.\ have some component $f_a$ different from the zero map) only if $a_2\le a_1 < b_2 \le b_1$. In this case, $\ker f$ is isomorphic to $K(b_2,b_1)$ if $b_2< b_1$, and is zero otherwise, and $\coker f$ is isomorphic to $K(a_2,a_1)$ if $a_2< a_1$, and is zero otherwise.

A persistence module is \textbf{indecomposable} if, whenever it is isomorphic to a direct sum $Y\oplus Z$ with $Y$ and $Z$ in $\Tame$, either $Y=0$ or $Z=0$. Bars are indecomposable and, as the following fundamental result implies, any indecomposable in $\Tame$ is isomorphic to a bar. We refer the reader to \cite{chazal2016structure} for more details on the algebraic structure of persistence modules. 

\begin{thm}[Structure of persistence modules]
\label{thm:barcode_decomp}
 Any (tame) persistence module $X$ is isomorphic to a finite direct sum of bars of the form $\bigoplus_{i=1}^k K(a_i,b_i)$, with $a_i<b_i$ in $[0,\infty]$ for every $i\in \{1,\ldots ,k\}$. This decomposition is unique up to permutation: if $X\cong \bigoplus_{i=1}^k K(a_i,b_i) \cong \bigoplus_{j=1}^\ell K(c_j,d_j)$, then $k=\ell$ and there exists a permutation $\sigma$ on $\{1,\ldots ,k\}$ such that $a_i = c_{\sigma (i)}$ and $b_i = d_{\sigma (i)}$, for every $i\in \{1,\ldots ,k\}$. 
\end{thm}

A decomposition of a persistence module $X$ as a direct sum of bars as in Theorem \ref{thm:barcode_decomp} is called a \textbf{barcode decomposition} of $X$. In this article, we will occasionally denote a barcode decomposition of $X$ by $\bigoplus_{i=1}^k X_i$ when we do not need an explicit notation for the bars' start- and end-points. The number $k$ of bars in any barcode decomposition of $X$ is called the \textbf{rank} of $X$, denoted by $\rank (X)$. 

Given a persistence module $X$, consider an element $x\in X_a$ for some $a$ in $[0,\infty )$, and let $b\coloneqq \sup \{t \in [a,\infty) \mid X_{a\le t}(x) \ne 0 \}$ in $[a,\infty]$. The element $x$ is called a \textbf{generator} of $X$ if the morphism $g\colon  K(a,b)\to X$ defined by $g_a (1)=x$ is such that the composition $r g$ with some morphism $r\colon  X\to K(a,b)$ is the identity on $K(a,b)$. We call $K(a,b)$ the \textbf{bar generated by} $x$, and we observe that it is a direct summand of $X$. We call a collection of elements $\{ x_i \in X_{a_i}\}_{i=1}^k$ a \textbf{set of generators} of $X$ if each $x_i$ generates a bar $K(a_i,b_i)$ and the morphisms $g_i \colon K(a_i,b_i)\to X$ defined by $x_i$ induce an isomorphism $\bigoplus_{i=1}^k K(a_i,b_i)\to X$.

As we will use basic homological algebra methods in $\Tame$, we remark that infinite bars $K(a,\infty)$, for all $a$ in $[0,\infty )$, are free in $\Tame$, and that the notions of free and projective coincide in $\Tame$ (see \cite{bubenik2021homological} for details). 
Any bar $K(a,b)$ with $b<\infty$ admits a minimal free resolution of the form $0\to K(b,\infty)\to K(a,\infty)\to K(a,b) \to 0$.

\begin{rmk}
We note that $\rank (X)$ can be viewed as a classical homological invariant corresponding to the number of generators in a minimal free resolution of $X$, which yields an alternative definition of the rank that is applicable to multiparameter persistence modules \cite{scolamiero2017multidimensional}. 
\end{rmk}

Lastly, let us briefly comment on a set theoretical detail regarding the category $\Tame$.
In $\Tame$, the class of isomorphism classes of objects is a set, as a consequence of Theorem \ref{thm:barcode_decomp}.  
In this article, we consider some class functions defined on $\Tame$, and we occasionally refer to them simply as functions for brevity. Since all class functions on $\Tame$ we consider are constant on isomorphism classes of objects, they can be regarded as proper functions defined on the set of isomorphism classes of persistence modules.

\subsection{Contours}
\label{subsec:contours}
Contours can be thought of as describing coherent ways to ``flow'' across the parameter space $[0,\infty )$ of persistence modules. 
In this article, we call \textbf{contour} a function $C\colon [0,\infty)\times [0,\infty) \to [0,\infty)$ such that, for all $a,b, \varepsilon , \tau$ in $[0,\infty)$, the following inequalities hold:
\begin{enumerate}
    \item if $a\le b$ and $\varepsilon \le \tau$, then $C(a,\varepsilon)\le C(b,\tau)$;
    \item $a \le C(a,0)$;
    \item $C(C(a,\varepsilon), \tau)\le C(a,\varepsilon +\tau)$.
\end{enumerate}

In \cite{gafvert2017stable} contours are defined in the case of $n$-parameter persistence modules. Contours are further studied for $1$-parameter persistence in \cite{chacholski2020metrics}, where several concrete examples are given. In \cite{chacholski2020metrics}, the definition of contour is slightly more general than ours; for example, $C(a,\varepsilon)$ can take the value $\infty$.
Similar notions to contours appear in the literature by the name of \emph{superlinear families of tranlations} \cite{bubenik2015metrics} and \emph{flows on posets} \cite{deSilva2018theory}.

A contour $C$ is called an \textbf{action} if the inequalities of (2.) and (3.) are equalities, that is, if $a = C(a,0)$ and $C(C(a,\varepsilon), \tau) = C(a,\varepsilon +\tau)$, for all $a,\varepsilon,\tau$.
A contour $C$ is \textbf{regular} \cite{chacholski2020metrics} if the following conditions hold: 
\begin{itemize}
    \item $C(-,\varepsilon)\colon  [0,\infty ) \to [0,\infty )$ is a monomorphism for all $\varepsilon \in [0,\infty)$;
    \item $C(a,-)\colon  [0,\infty) \to [0,\infty)$ is a monomorphism whose image is $[a,\infty)$, for all $a \in [0,\infty)$.
\end{itemize}
The second condition of regular contours ensures that $C(a,0)=a$, for any $a$ in $[0,\infty )$, and that $C$ is strictly increasing in the second variable:  $C(a,\varepsilon)< C(a,\tau)$ whenever $\varepsilon < \tau$, for any $a$ in $[0,\infty )$. For brevity, we call a contour $C$ a \textbf{regular action} if it is both regular and an action.

Let $C$ be a regular contour. For all $a \in [0,\infty)$, we define the function $\ell(a,-)$ to be the inverse of the function $C(a,-)\colon [0,\infty )\to [a,\infty )$, that is, $\ell (a,b) = C(a, -)^{-1}(b)$ for any $b\in [a,\infty)$, and we set $\ell(a,\infty)=\infty$. We call $\ell$ the \textbf{lifetime function} associated with $C$.
We observe that, since regular contours are injective functions in the second variable, $\ell(a,b)$ is well-defined for every pair $a \le b$. Throughout the article, the \textbf{lifetime of a bar $K(a,b)$ with respect to a contour $C$} is the value $\ell(a,b)$ of the lifetime function associated with $C$.

As a first example of contour we consider the \textbf{standard contour}, i.e.\ the function $D$ defined by $D(a,\varepsilon)=a+\varepsilon$, for every $a,\varepsilon \in [0,\infty)$. Informally, the standard contour describes the most uniform way to flow in the parameter space $[0,\infty )$ of a persistence module, linearly with unitary speed. 
We now introduce a large family of contours, called \textbf{integral contours of distance type} \cite{chacholski2020metrics,agerberg2021supervised}, parametrized by certain real-valued functions. Let $f\colon [0,\infty ) \to (0,\infty )$ be a Lebesgue measurable function, called here a \textbf{density}. For every $a,\varepsilon \in [0,\infty)$, let $D_f (a,\varepsilon )$ be the real number in $[a,\infty )$ such that
\begin{equation*}
\varepsilon = \int_{a}^{D_f (a,\varepsilon )} f(x) \,dx ,
\end{equation*}
which is uniquely defined since $f$ takes strictly positive values.  
The function $D_f\colon [0,\infty)\times [0,\infty) \to [0,\infty )$ is a contour; moreover, it is regular and an action. We observe that, if the density $f$ is the constant function with value $1$, the distance type contour $D_1$ coincides with the standard contour.

\subsection{Noise systems}
\label{subsec:noise_systems_distances}
Noise systems provide a way to quantify the size of persistence modules and to produce pseudometrics on $\Tame$ by comparing their sizes  \cite{scolamiero2017multidimensional}.
A \textbf{noise system} on $\Tame$ is a sequence $\ns = \{ \ns_{\varepsilon} \}_{\varepsilon \in [0,\infty)}$ of subclasses of $\Tame$ such that:
\begin{itemize}
\item $0 \in \ns_{\varepsilon}$, for all $\varepsilon$,
\item $\ns_{\tau} \subseteq \ns_{\varepsilon}$ whenever $\tau \le \varepsilon$,
\item if $0\to X_0 \to X_1 \to X_2 \to 0$ is a short exact sequence in $\Tame$, then:
\begin{itemize}
    \item if $X_1 \in \ns_{\varepsilon}$, then $X_0, X_2 \in \ns_{\varepsilon}$, 
    \item if $X_0 \in \ns_{\varepsilon}$ and $X_2 \in \ns_{\tau}$, then $X_1 \in \ns_{\varepsilon + \tau}$.
\end{itemize}
\end{itemize}

Given a noise system $\ns = \{ \ns_{\varepsilon} \}_{\varepsilon \in [0,\infty)}$ it is natural to associate to each persistence module $X$ the smallest $\varepsilon$ such that $X\in \ns_{\varepsilon}$. This defines a class function $\alpha_{\ns}\colon  \Tame \rightarrow [0,\infty ]$ called in \cite{giunti2021amplitudes} the \textbf{amplitude} associated to $\ns$.  

A noise system $\ns = \{ \ns_{\varepsilon} \}_{\varepsilon \in [0,\infty)}$ is \textbf{closed under direct sums} if $X\oplus Y \in \ns_{\varepsilon}$ whenever $X, Y \in \ns_{\varepsilon}$, for every $\varepsilon \in [0,\infty)$. Contours (Section \ref{subsec:contours}) provide examples of noise systems satisfying this property. Given a contour $C$ and any $\varepsilon \in [0,\infty )$, let 
\begin{equation*}
\ns_{\varepsilon} \coloneqq  \{ X \in \Tame \mid X_{a \le C(a,\varepsilon)}=0 \text{ for all } a\in [0,\infty )  \}. 
\end{equation*}
It is proved in \cite[Prop.~9.4]{gafvert2017stable} that the sequence $\{ \ns_{\varepsilon} \}_{\varepsilon \in [0,\infty)}$ defined in this way is a noise system closed under direct sums. In particular, the noise system induced by the standard contour has components
\begin{equation*}
\ns_{\varepsilon} \coloneqq  \{ X \in \Tame \mid X_{a \le a+\varepsilon }=0 \text{ for all } a\in [0,\infty)\},
\end{equation*}
and coincides with the \textbf{standard noise system} introduced in \cite{scolamiero2017multidimensional}.

\subsection{Pseudometrics between persistence modules}
\label{subsec:noise_induced_distances}
In this article, we call (extended) \textbf{pseudometric} on $\Tame$ a class function $d$ assigning to any pair of persistence modules $X,Y$ in $\Tame$ an element $d(X,Y)\in [0,\infty]$ such that the following conditions hold for any $X,Y,Z$: 
\begin{itemize}
\item $d(X,Y) =d(Y,X)$,
\item $d(X,Y) =0$ whenever $X$ is isomorphic to $Y$,
\item $d(X,Z)\le d(X,Y)+d(Y,Z)$. 
\end{itemize}
The third condition, known as the triangle inequality, combined with the second one yields $d(X,Y)=d(X',Y')$ whenever $X\cong X'$ and $Y\cong Y'$. This definition of pseudometric coincides with Definition $3.3$ in \cite{bubenik2022exact} when considering the category $\Tame$.

We now briefly explain how noise systems yield pseudometrics on $\Tame$.
Let $\ns$ be a noise system on $\Tame$. For any $\varepsilon \in [0,\infty)$, we say that two persistence modules $X$ and $Y$ are $\varepsilon$-\textbf{close} if there exists a persistence module $Z$ and a pair of morphisms $X \xleftarrow{f} Z \xrightarrow{g} Y$ such that
\begin{equation*}
\ker f \in \ns_{\varepsilon_1}, \quad \coker f \in \ns_{\varepsilon_2}, \quad
\ker g \in \ns_{\varepsilon_3}, \quad \coker g \in \ns_{\varepsilon_4}, 
\end{equation*}
for some $\varepsilon_1, \varepsilon_2, \varepsilon_3, \varepsilon_4 \in [0, \infty)$ 
such that $\varepsilon_1 + \varepsilon_2 + \varepsilon_3 + \varepsilon_4 \le \varepsilon$. Define
\begin{equation*}
d_{\ns}(X,Y) \= \inf \left\{ \varepsilon \in [0,\infty) \mid \text{$X$ and $Y$ are $\varepsilon$-close} \right\} ,
\end{equation*}
adopting the convention $\inf \varnothing = \infty$. As shown in  \cite[Prop.~8.7]{scolamiero2017multidimensional}, $d_{\ns}$ is a pseudometric on $\Tame$. 

We remark that the pseudometric $d_{\ns}$ associated with the standard noise system is equivalent to the interleaving distance \cite{lesnick2015theory}, as proved by \cite[Prop. 12.2]{gafvert2017stable}.

\subsection{Hierarchical stabilization and stable rank}
\label{subsec:stabilization}
In the context of TDA, \textbf{hierarchical stabilization} is a method to convert a discrete invariant of persistence modules into a stable invariant suitable for data analysis. This technique has been studied in  \cite{scolamiero2017multidimensional,gafvert2017stable} in the case of multiparameter persistence modules, and has been further investigated in \cite{chacholski2020metrics} in the case of one-parameter persistence. 
Hierarchical stabilization has a very general formulation, which allows for several choices of discrete invariants, and in principle is not restricted to categories of persistence modules. For the hierarchical stabilization of the rank, also called stable rank, some computational methods have been developed \cite{gafvert2017stable,chacholski2020metrics}. In this article we will restrict our attention to the stable rank and further develop its computation.

Besides choosing a discrete invariant, hierarchical stabilization requires the choice of a pseudometric between persistence modules, which plays an active role in calculating the corresponding stable invariant. 
Consider the rank of a persistence module (Section \ref{subsec:persmod}) as a class function $\rank\colon  \Tame \to \mathbb{N}$ mapping any persistence module $X$ to the natural number $\rank (X)$.

\begin{defn}
\label{def:stable_rank}    
Given a pseudometric $d$ on $\Tame$ (Section \ref{subsec:noise_induced_distances}), the \textbf{stable rank} of a persistence module $X$ with respect to the pseudometric $d$ is the function $\widehat{\rank}_d (X)\colon  [0,\infty )\to [0,\infty)$ defined, for all $t\in [0,\infty )$, by 
\begin{equation*}
\widehat{\rank}_d (X)(t) \coloneqq  \min \{ \rank (Y) \mid Y\in \Tame \text{ and } d(X,Y)\le t \} .
\end{equation*}
\end{defn}
We observe that the function $\widehat{\rank}_d (X)$ is non-increasing and takes values in $\mathbb{N}$, so it belongs to the set $\mathcal{M}$ of Lebesgue measurable functions $[0,\infty) \to [0,\infty)$. 

To illustrate the stability of the invariant $\widehat{\rank}_d$, we consider a pseudometric $d_{\bowtie}$ on $\mathcal{M}$, called the  \textbf{interleaving distance}, defined for all $f,g\in \mathcal{M}$ by
\begin{equation*}
d_{\bowtie} (f,g) \coloneqq  \inf \{ \varepsilon \in [0,\infty ) \mid  f(t)\ge g(t+\varepsilon) \text{ and } g(t)\ge f(t+\varepsilon), \text{ for all } t\in [0,\infty ) \} ,
\end{equation*}
setting by convention $\inf \varnothing = \infty$. 
The stable rank then satisfies the following Lipschitz condition.

\begin{prop}[\cite{scolamiero2017multidimensional}]
\label{prop:stability_sr}  
Let $d$ be a pseudometric on $\Tame$, and let $X,Y$ be persistence modules. Then
$d(X,Y) \ge d_{\bowtie} (\widehat{\rank}_d (X),\widehat{\rank}_d (Y))$.
\end{prop}

\subsection{\texorpdfstring{$p$}{p}-norms}
\label{subsec:pnorms}
In this subsection, we briefly review properties of $p$-norms that are useful for our work. 
For $p\in [1,\infty ]$, the $p$-norm (also called $L^p$-norm) on $\mathbb{R}^n$ is the function $\left\| \cdot \right\|_p \colon  \mathbb{R}^n \to [0,\infty)$ defined, for each $x=(x_1, x_2, \ldots , x_n)\in \mathbb{R}^n$, by 
\begin{equation*}
\left\| x \right\|_p \=  \begin{cases}
\left( \sum_{i=1}^n \lvert x_i \rvert^{p} \right)^{\frac{1}{p}} &\text{for $p\in [1,\infty)$}\\
\max \{ \lvert x_i \rvert \}_{i\in \{1,\ldots ,n\} } &\text{for $p=\infty$} .
\end{cases}
\end{equation*}
We note that $\displaystyle \|x\|_\infty=\lim_{p\to\infty}\|x\|_p$, for all $x\in \mathbb{R}^n$.
The triangle inequality (or subadditivity condition)
$\left\| x+y \right\|_p \le \left\| x \right\|_p + \left\| y \right\|_p$, for all  $x,y\in \mathbb{R}^n$, is also referred to as Minkowski inequality.

A fundamental property of $p$-norms on $\mathbb{R}^n$ is the following: for $x\in \mathbb{R}^n$ and for $1\le p\le q\le \infty$, the inequalities
\begin{equation}
\label{eq:monotonicity_p}
\left\| x \right\|_{q} \le \left\| x \right\|_{p} \le n^{\left( \frac{1}{p} - \frac{1}{q}\right)}\left\| x \right\|_{q}
\end{equation}
hold and are sharp, where by convention we set $\frac{1}{\infty}=0$. We refer to the first inequality as the monotonicity property of $p$-norms.

The following elementary property of $p$-norms is useful in this work: for $p\in [1,\infty]$, if $x=(x_1, \ldots , x_n)\in \mathbb{R}^n$, $y=(y_1, \ldots , y_m)\in \mathbb{R}^m$ and $z=(x_1, \ldots , x_n, y_1, \ldots , y_m)\in \mathbb{R}^{n+m}$, then
\begin{equation}
\label{eq:concatenate_p}
\left\| \left( \left\| x \right\|_{p}, \left\| y \right\|_{p} \right) \right\|_{p} = \left\| z \right\|_{p} .
\end{equation}
Finally, let us also observe that $p$-norms are permutation invariant, and that they preserve the order on $[0,\infty)^n$, meaning that if $x\le y$ in $[0,\infty)^n$ according to the coordinate-wise order, then $\|x\|_p \le \|y\|_p$. 

In this article, we generally consider $p$-norms as functions from $[0,\infty]^n$ to $[0,\infty]$, extending the usual definition by setting $\| x\|_p =\infty$ whenever $x$ has some coordinate $x_i=\infty$. All properties stated above still hold with this definition.

Following \cite{skraba2020wasserstein}, we will consider  $p$-norms of persistence modules, whose definition relies on the barcode decomposition (Section \ref{subsec:persmod}). 
For $p\in [1,\infty]$, the $p$-norm of a persistence module $X$ having barcode decomposition $X\cong \bigoplus_{i=1}^k K(a_i, b_i)$ is defined by
\begin{equation*}
\left\| X\right\|_p \=  \begin{cases}
\left( \sum_{i=1}^k \lvert b_i - a_i \rvert^{p} \right)^{\frac{1}{p}} &\text{for $p\in [1,\infty)$}\\
\max \{ \lvert b_i - a_i \rvert \}_{i\in \{1,\ldots ,k\} } &\text{for $p=\infty$} .
\end{cases}
\end{equation*}

For $p\in [1,\infty]$ and $\varepsilon \in [0,\infty)$, the class of tame persistence modules with $p$-norm  smaller or equal to $\varepsilon$ is denoted by:
\begin{equation*}
\ns^{p}_{\varepsilon} \= \{ X \in \Tame \mid \left\| X \right\|_{p} \le \varepsilon \} ,
\end{equation*}
and we set $\ns^{p} \= \{ \ns^{p}_{\varepsilon} \}_{\varepsilon \in [0,\infty)}$.


\section{Monomorphisms, epimorphisms, and their \texorpdfstring{$p$}{p}-norms}
\label{sec:mono_pnorm}
In this section we introduce bar-to-bar morphisms between persistence modules (Definition~\ref{def:bartobar-morph}), which can  informally be described as morphisms such that every bar in the barcode decomposition of the domain maps non-trivially to at most one bar in the barcode decomposition of the codomain. Our aim is proving results (Theorem~\ref{thm:mono-bar} and Theorem~\ref{thm:epi-bar}) which compare monomorphisms and epimorphisms between two persistence modules to bar-to-bar  
monomorphisms and epimorphisms
between the same persistence modules. These results allow us to reduce algebraic problems to much simpler combinatorial problems, as shown for example in Corollary \ref{coro:mono-canonical} and Corollary \ref{coro:epi-canonical}.

\subsection{Free presentations of monomorphisms}
\label{sec:free_pres_mono}

Given a monomorphism $f\colon Z\hookrightarrow X$ between persistence modules, we want to determine the barcode decomposition of $\coker f$. We briefly describe a method that uses free resolutions of the persistence modules $Z$ and $X$. 

Consider the diagram 
\begin{equation*}
\begin{tikzcd}
0 & {R_Z} & {G_Z} & Z & 0 \\
0 & {R_X} & {G_X} & X & 0 \\
{} & {} & {} & {\coker f} & {}
\arrow["f", hook', from=1-4, to=2-4]
\arrow["{f_G}", from=1-3, to=2-3]
\arrow["{f_R}", from=1-2, to=2-2]
\arrow[from=1-1, to=1-2]
\arrow["i_Z", from=1-2, to=1-3]
\arrow["p_Z", from=1-3, to=1-4]
\arrow[from=1-4, to=1-5]
\arrow[from=2-1, to=2-2]
\arrow["i_X", from=2-2, to=2-3]
\arrow["p_X", from=2-3, to=2-4]
\arrow[from=2-4, to=2-5]
\arrow["q", two heads, from=2-4, to=3-4]
\end{tikzcd}
\end{equation*}
where the rows are (minimal) free resolutions of the persistence modules $Z$ and $X$ respectively, and $q$ denotes the canonical epimorphism. The given morphism $f$ induces a morphism $f_G\colon G_Z \to G_X$ between the modules of generators and a morphism $f_R\colon R_Z \to R_X$ between the modules of relations that make the diagram commutative (see e.g.\ \cite[Thm.~6.16]{rotman2009homological}). We have $\coker f \cong \coker([f_G \; i_X]\colon  G_Z \oplus R_X \xrightarrow{} G_X)$, where the morphism $[f_G \; i_X]$ sends $(z,r)\in G_Z \oplus R_X$ to $f_G (z) + i_X (r)$. The isomorphism of cokernels is easy to prove, for example observing that the image of the composition $q p_X$ is $\coker f$ given that both $q$ and $p_X$ are surjective and verifying via diagram chasing that its kernel coincides with the image of $[f_G \; i_X]\colon  G_Z \oplus R_X \xrightarrow{} G_X$.

In other words, we have a free presentation of $\coker f$
\[
G_Z \oplus R_X \xrightarrow{[f_G \; i_X]} G_X \twoheadrightarrow \coker f ,
\]
and we can use it to determine the barcode decomposition of $\coker f$. More precisely, observing that $\coker f$ is isomorphic to the homology at the middle term of the free chain complex
\[
G_Z \oplus R_X \xrightarrow{[f_G \; i_X]} G_X \longrightarrow 0 ,
\]
we can compute the barcode decomposition of $\coker f$ by using the persistent homology algorithm on a matrix $M_f$ representing the morphism $[f_G \; i_X]$, as we detail in Section \ref{subsec:main_mono}. The persistent homology algorithm determines ``pairings'' of the basis elements of $G_Z \oplus R_X$ with the basis elements of $G_X$, which corresponds to the start- and end-point pairs of the bars of $\coker f$.

In this section, we are interested in particular morphisms between persistence modules, which we call bar-to-bar morphisms.

\begin{defn}
\label{def:bartobar-morph}
A morphism $f\colon Z\to X$ of persistence modules is \textbf{bar-to-bar} if there are barcode decompositions $Z=\bigoplus_{i=1}^m Z_i$ and $X=\bigoplus_{j=1}^n X_j$ and there exist a subset $I\subseteq \{1,\ldots ,m\}$ and an injective function $\alpha \colon I\to \{1,\ldots, n\}$ such that 
\begin{equation}
\label{eq:bartobar_f}
f = \bigoplus_{i\in I} f_i \oplus \bigoplus_{i\in \{1,\ldots ,m\} \setminus I} g_i \oplus \bigoplus_{j\in \{1,\ldots ,n\} \setminus \alpha (I)} h_j , 
\end{equation}
where each $f_i \coloneqq  f\vert_{Z_i}$ is a nonzero morphism $Z_i \to X_{\alpha (i)}$, and where $g_i$ denotes the zero morphism $Z_i \to 0$ and $h_j$ denotes the zero morphism $0 \to X_j$.
\end{defn}

\begin{rmk}
\label{rmk:ker_coker_bartobar}
If $f$ is a bar-to-bar morphism as in (\ref{eq:bartobar_f}), then $\ker f$ and $\coker f$ are easily determined recalling the case of a morphism between two bars (see Section \ref{subsec:persmod}), namely:
\begin{equation*}
\ker f = \bigoplus_{i\in I} \ker f_i \oplus \bigoplus_{i\in \{1,\ldots ,m\} \setminus I} Z_i , \qquad 
\coker f =  \bigoplus_{i\in I} \coker f_i \oplus \bigoplus_{j\in \{1,\ldots ,n\} \setminus \alpha (I)} X_j .
\end{equation*}
Furthermore, if $f$ is a monomorphism, the fact that $\ker f$ vanishes implies that $I=\{1,\ldots ,m \}$, and the existence of the injective function $\alpha$ implies $m\le n$. Dually, $\alpha (I)=\{1,\ldots ,n \}$ and $n\le m$ if $f$ is an epimorphism.
\end{rmk}

The main result of this section is the following (Theorem~\ref{thm:mono-bar}): given any monomorphism $f\colon Z\hookrightarrow X$, there is a bar-to-bar monomorphism $f_b \colon  Z\hookrightarrow X$ such that $\| \coker f_b \|_p \le \| \coker f \|_p$ for any $p\in [1,\infty]$. A dual statement (Theorem~\ref{thm:epi-bar}) holds for kernels of epimorphisms.

\subsection{Finding monomorphisms with smaller cokernels}
\label{subsec:main_mono}
To prove our inequalities between $p$-norms of cokernels, we modify a strategy used in \cite[Sect.\ 7.1]{skraba2020wasserstein} to obtain new inequalities between $p$-norms of persistence modules, based on the rearrangement inequality (Theorem \ref{thm:rearr_ineq}) and on the comparison of pairings in certain barcode decompositions using the persistent homology algorithm.
For simplicity, we fix the field with two elements $\mathbb{F}_2$ as the base field in this subsection, but our results work for any base field.

Let $Z$ and $X$ be persistence modules and $f \colon Z \hookrightarrow X$ a monomorphism of persistence modules. Fix $\{z_i\}_{i = 1}^m$ and $\{x_j\}_{j = 1}^n$ sets of generators of $Z$ and $X$, respectively, and denote by $Z = \bigoplus_{i = 1}^m K(a_i^z, b_i^z)$ and $X = \bigoplus_{j = 1}^n K(a_j^x, b_j^x)$ the respective barcode decompositions. 
That is, for every $z_i$, $a_i^z$ is the \textbf{degree} of $z_i \in Z_{a_i^z}$ and $b_i^z$ is the end-point of the bar generated by $z_i$, and similarly for the $x_j$. In this section, we assume for the ease of exposition that $X$ has no infinite bars in its decomposition. All the results we present can be adapted to the general case by setting $b_j^x = \infty$ whenever $x_j$ generates an infinite bar.
Throughout this section, we will consider an example monomorphism $f$, which we represent as follows:
\begin{equation}
\label{E:ExampleMono}
\begin{tikzpicture}[scale = 0.4, baseline = {([yshift = .5ex](0, 2))}]
\draw[line width=1pt]
  (5, 10) node[left]{$a_1^z$} -- (10, 10) node[right]{$b_1^z = b_3^x$}
  (7, 9) node[left]{$a_2^z$} -- (10, 9) node[right]{$b_2^z = b_3^x$}
  (8, 8) node[left]{$a_3^z$} -- (13, 8) node[right]{$b_3^z = b_6^x$}
;

\draw[line width=1pt]
  (4, 5) node[left]{$a_1^x$} -- (6, 5) node[right]{$b_1^x$}
  (2, 4) node[left]{$a_2^x$} -- (10, 4) node[right]{$b_2^x$}
  (3, 3) node[left]{$a_3^x$} -- (10, 3) node[right]{$b_3^x$}
  (5, 2) node[left]{$a_4^x$} -- (11, 2) node[right]{$b_4^x$}
  (0, 1) node[left]{$a_5^x$} -- (12, 1) node[right]{$b_5^x$}
  (1, 0) node[left]{$a_6^x$} -- (13, 0) node[right]{$b_6^x$}
;

\draw[->] (5.2, 9.8) -- (5.2, 5.2);
\draw[->] (8.2, 8.8) -- (8.2, 4.2);
\draw[->] (8.8, 9.8) -- (8.8, 3.2);
\draw[->] (9.2, 8.8) -- (9.2, 3.2);
\draw[->] (10.2, 7.8) -- (10.2, 2.2);
\draw[->] (11.2, 7.8) -- (11.2, 1.2);
\draw[->] (12.2, 7.8) -- (12.2, .2);

\node (z) at (17, 9) {$Z$};
\node (x) at (17, 3) {$X$};
\draw[->] (z) --node [right]{$f$} (x);
\end{tikzpicture}
\end{equation}
The persistence modules $Z$ and $X$ are represented in terms of their barcode decompositions. An arrow between bars indicates that the bar in the domain maps non-trivially to the bar in the codomain.

The main results of this subsection are based on matrix reduction arguments applied to a matrix $M_f$ associated with the morphism $[f_G \; i_X]\colon  G_Z \oplus R_X \rightarrow G_X$ (Section \ref{sec:free_pres_mono}), which we construct as follows.

\begin{defn}
Define the sets of labels $\Xgenlabel \coloneqq  \{x_j\}_{j = 1}^n$, $\Zgenlabel \coloneqq  \{z_i\}_{i = 1}^m$, and $\Xrellabel \coloneqq  \{r_j\}_{j = 1}^n$, where $\{z_i\}_{i = 1}^m$ and $\{x_j\}_{j = 1}^n$ are generators of $Z$ and $X$ respectively and $r_j$ corresponds to the generator of $R_X$ that is sent by $i_X$ to the bar generated by $x_j$ in $G_X$. The \textbf{degree} of $r_j$ is $b_j^x$.

The \textbf{presentation matrix of $f$} is an $n \times (m + n)$ matrix $M_f$ with rows labeled by $\Xgenlabel$ and columns labeled by $\Zgenlabel \sqcup \Xrellabel$, constructed as follows. For each $z_i$ in $\Zgenlabel$, we set the corresponding column of $M_f$ to be the column vector $f_{a_i^z}(z_i) \in X_{a_i^z}$ in the basis given by the nonzero elements of $\{X_{a_j^x \leq a_i^z}(x_j)\}_{j = 1}^n$. Note that if $X_{a_j^x \leq a_i^z}(x_j)$ is $0$, then $M_f(x_j, z_i)$ is also $0$.
For each $r_j$ in $\Xrellabel$, we set the corresponding column of $M_f$ to be the zero vector except with a $1$ on the row $x_j$. Finally, we reorder the rows and columns so that the degrees of the labels are nondecreasing.

We denote by $M_f(x, c)$ the entry of $M_f$ in row $x\in \Xgenlabel$ and column $c\in \Zgenlabel \sqcup \Xrellabel$.
\end{defn}

As an example, one presentation matrix of the example monomorphism $f$ from \eqref{E:ExampleMono} is
\begin{equation}
\label{E:AssociatedMatrix}
M_f = \begin{blockarray}{c@{}c@{}*{9}c}
&& z_1 & r_1 & z_2 & z_3 & r_2 & r_3 & r_4 & r_5 & r_6 \\
\begin{block}{c[c@{\;}*{9}c]}
x_5 && \tikzmark{a1}{0} &   & \tikzmark{a2}{0} & 1 &   &   &   & 1 &   \\
x_6 && 0 &   & 0 & 1 &   &   &   &   & 1 \\
x_2 && 0 &   & 1 & 0 & 1 &   &   &   &   \\
x_3 && 1 &   & 1 & 0 &   & 1 &   &   &   \\
x_1 && 1 & 1 & 0 & 0 &   &   &   &   &   \\
x_4 && \tikzmark{b1}{0} &   & 0 & \tikzmark{b2}{1} &   &   & 1 &   &   \\
\end{block}
\end{blockarray},
\begin{tikzpicture}[overlay, remember picture]
\draw (a1.north west) rectangle (b1.south east);
\draw (a2.north west) rectangle (b2.south east);
\end{tikzpicture}
\end{equation}
where the columns $\Zgenlabel = \{z_1, z_2, z_3\}$ are outlined, while the columns $\Xrellabel = \{r_1, \ldots, r_6\}$ are represented sparsely: blank spaces are zero coefficients. Note that the restriction of the matrix $M_f$ to the columns $\Xrellabel$ represents the morphism $i_X \colon R_X \to G_X$.

\begin{rmk}
\label{rmk:pers_algo}
As we mentioned in Section \ref{sec:free_pres_mono}, we want to determine the barcode decomposition of $\coker f$ by using the persistent homology algorithm on the matrix $M_f$ representing the morphism $[f_G \; i_X]$. More precisely, we are interested in methods to compute barcode decompositions based on matrix reduction via left-to-right column operation, like the so-called standard algorithm for persistent homology \cite{edelsbrunner2000topological,zomorodian2005computing} (see Algorithm 1 in \cite{otter2017roadmap} for a description). Even though
these methods
are usually presented for filtered simplicial complexes in the literature, they extend to graded free chain complexes as in our case. The barcode decomposition (of $\coker f$ in our case) can be read out from a reduced matrix, and does not depend on the way of reducing the matrix via left-to-right column operations (see Lemma \ref{lem:sigma_welldef}).
\end{rmk}

Let $\bar M_f$ be a complete reduction of $M_f$ by left-to-right column transformations, where a matrix is said to be \textbf{reduced} if no two columns have their lowest nonzero entry on the same row. Let $\sigma_f$ be the function that to the $k^{\text{th}}$ nonzero column of $\bar M_f$ associates the row of its lowest nonzero entry, for every $k\in \{1,\ldots ,n\}$. We know that $\sigma_f$ is a permutation on $\{1,\ldots ,n\}$ since the $n$ columns of $M_f$ in $\Xrellabel$ are linearly independent. 
In this section, we use square brackets for a permutation $\sigma = [\sigma (1) \cdots \sigma (n)]$ on $\{1,\ldots ,n\}$ expressed in one-line notation, to distinguish it from the notation for cycles, denoted by $(c_1 \, c_2 \, \cdots \, c_\ell)$.  
For the running example \eqref{E:AssociatedMatrix}, we get 
\begin{equation*}
\bar M_f = \left[ \enspace \begin{matrix}
\tikzmark{a1}{0} &   & 0 & \tikzmark{a2} 1 &   &   &   & \fbox{$1$} &   \\
0 &   & 0 & 1 &   &   & \fbox{$1$} &   & 0 \\
0 &   & \fbox{$1$} & 0 & 0 &   &   &   &   \\
1 & \fbox{$1$} & 0 & 0 &   & 0 &   &   &   \\
\fbox{$1$} & 0 & 0 & 0 &   &   &   &   &   \\
\tikzmark{b1}{0} &   & \tikzmark{b2}{0} & \fbox{$1$} &   &   & 0 &   &   \\
\end{matrix} \right]
\begin{tikzpicture}[overlay, remember picture]
\draw (a1.north west) rectangle (b1.south east);
\draw (a2.north east) rectangle (b2.south west);
\end{tikzpicture},
\end{equation*}
where we have outlined the lowest nonzero coefficient of each column, and so $\sigma_f = [543621]$. 
We do not need to specify the order of transformations in this reduction thanks to the following lemma, which is a consequence of the pairing uniqueness lemma of \cite[Sect.\ 3]{cohen2006vines}.

\begin{lem}
\label{lem:sigma_welldef}
The permutation $\sigma_f$ is well-defined.
In particular, it does not depend on the choice of a sequence of left-to-right column operations to obtain a reduced matrix from $M_f$.
\end{lem}

By design of the persistent homology algorithm, a barcode decomposition of $\coker f$ is completely determined by $\sigma_f$ together with the degrees of the generators of $Z$ and $X$. In Corollary \ref{coro:coker_f_b} we will provide a precise statement.

From the matrix $M_f$ we define the \textbf{bar-to-bar matrix} $M_b$ by Algorithm \ref{alg:btb}. The bar-to-bar matrix $M_b$ is the presentation matrix of a \textbf{bar-to-bar monomorphism} $f_b\colon Z\hookrightarrow X$ having the same domain and codomain as $f$. 

Algorithm \ref{alg:btb} also partially reduces $M_f$ and constructs an injective function $r_{\max} \colon \Zgenlabel \to \Xrellabel$. 
Given a column $z$ in $\Zgenlabel$, we call $r_{\max}(z)$ its \textbf{rightmost matched column} in $\Xrellabel$. 
Informally, Algorithm \ref{alg:btb} computes the bar-to-bar matrix $M_b$ by setting to zero each entry in the columns $z$ of $M_f$ in $\Zgenlabel$ except for the nonzero entry on the unique row $x$ such that $M_f(x,r_{\max}(z))=1$. For example, starting with the matrix $M_f$ from \eqref{E:AssociatedMatrix}, we get
\begin{equation*}
\left[ \enspace \begin{matrix}
\tikzmark{a1}{0} & \hphantom{\fbox{1}} & \tikzmark{a2}{0} & 1 & \hphantom{\fbox{1}} & \hphantom{\fbox{1}} & \hphantom{\fbox{1}} & 1 & \hphantom{\fbox{1}} \\
0 &   & 0 & \tikzmark{s1}{1} &   &   &   &   & \tikzmark{t1}{\fbox{1}} \\
0 &   & \tikzmark{s2}{1} &   & \tikzmark{t2}{\fbox{1}} &   &   &   &   \\
\tikzmark{s3}{1} &   &   &   &   & \tikzmark{t3}{\fbox{1}} &   &   &   \\
1 & 1 & 0 & 0 &   &   &   &   &   \\
\tikzmark{b1}{0} &   & 0 & \tikzmark{b2}{1} &   &   & 1 &   &   \\
\end{matrix} \right],
\begin{tikzpicture}[overlay, remember picture]
\draw (a1.north west) rectangle (b1.south east);
\draw (a2.north west) rectangle (b2.south east);
\draw[->] (s1) -- (t1);
\draw[->] (s2) -- (t2);
\draw[->] (s3) -- (t3);
\end{tikzpicture}
\end{equation*}
where the arrows represent the function $r_{\max}$. The corresponding matrix $M_b$ is
\begin{equation*}
\left[ \enspace
\begin{matrix}
\tikzmark{a1}{0} & \hphantom{\fbox{1}} & \tikzmark{a2}{0} & 0 & \hphantom{\fbox{1}} & \hphantom{\fbox{1}} & \hphantom{\fbox{1}} & 1 & \hphantom{\fbox{1}} \\
0 &   & 0 & 1 &   &   &   &   & 1 \\
0 &   & 1 & 0 & 1 &   &   &   &   \\
1 &   & 0 & 0 &   & 1 &   &   &   \\
0 & 1 & 0 & 0 &   &   &   &   &   \\
\tikzmark{b1}{0} &   & 0 & \tikzmark{b2}{0} &   &   & 1 &   &   \\
\end{matrix} \right].
\begin{tikzpicture}[overlay, remember picture]
\draw (a1.north west) rectangle (b1.south east);
\draw (a2.north west) rectangle (b2.south east);
\end{tikzpicture}
\end{equation*}

\begin{algorithm}
\caption{Bar-to-bar algorithm}\label{alg:btb}
\textbf{Input:} a presentation matrix $M_f$ of a monomorphism $f$

\textbf{Output:} a partially reduced matrix $M^*_f$, the associated bar-to-bar matrix $M_b$, and a function $r_{\max} \colon \Zgenlabel \to \Xrellabel$
\begin{algorithmic}[1]
\State Let $M_b \coloneqq  M_f$
\State Let $M_f^* \coloneqq  M_f$
\State Set the columns $\Zgenlabel$ of $M_b$ to $0$
\For{$r \in \Xrellabel$ in decreasing order}
\State Let $x$ be the row associated to $r$ (that is, $M_f^*(x,r)=1$)
\If{$\exists z \in \Zgenlabel$ such that $M_f^*(x, z) = 1$ and $r_{\max}(z)$ is undefined}
\State Let $z$ be minimal such that $M_f^*(x, z) = 1$ and $r_{\max}(z)$ is undefined
\State Set $M_b(x, z) = 1$
\State Define $r_{\max}(z) \coloneqq  r$
\For{$z' > z$ such that $M_f^*(x, z') = 1$}
\State Reduce column $z'$ in $M_f^*$ by column $z$ to set to zero the entry in row $x$
\For{$r' \in \Xrellabel$ and $x'$ the row associated to $r'$, such that $r' < z'$ and $M_f^*(x', z') = 1$}
\State Reduce column $z'$ in $M_f^*$ by column $r'$
\EndFor 
\EndFor
\EndIf
\EndFor
\end{algorithmic}
\end{algorithm}

The following two propositions prove useful facts regarding Algorithm \ref{alg:btb}.
\begin{prop}
\label{prop:columnsCZ}
In a presentation matrix $M_f$ of a monomorphism $f\colon Z\hookrightarrow X$, all columns in $\Zgenlabel$ are nonzero. Moreover, for every column $z \in \Zgenlabel$, all columns in the set
\begin{equation*}
\Gamma (z) \coloneqq  \{r \in \Xrellabel \mid r \text{ and } z \text{ have a nonzero entry on the same row}  \}
\end{equation*}
have degree strictly larger than the degree of $z$,
and  $\lvert \Gamma (z)\rvert$ equals the number of nonzero entries of $z$.
\end{prop}
\begin{proof}
Since $f$ is a monomorphism, it cannot send a generator of a bar of $Z$ to zero, hence the columns in $\Zgenlabel$ are nonzero. A nonzero entry in a column $z\in \Zgenlabel$ indicates that the corresponding generator of a bar of $Z$ maps non-trivially to the vector space generated 
by $X_{a^x \le a^z} (x)$ for some $x$ generating a bar in $X$, where $a^x$ is the degree of $x$ and $a^z$ is the degree of $z$. 
This implies that the end-point of the bar of $X$ generated by $x$ has degree strictly larger than the degree of $z$. Lastly, the cardinality of $\Gamma (z)$ equals the number of nonzero entries of $z$ because the columns in $\Xrellabel$ form a permutation matrix of rank $n$.
\end{proof}

\begin{prop}
\label{prop:rmax}
Let $f\colon Z\hookrightarrow X$ be a monomorphism and let $M_f$ be a presentation matrix of $f$. The execution of Algorithm \ref{alg:btb} on $M_f$ returns a well-defined function $r_{\max}\colon \Zgenlabel \to \Xrellabel$ that is injective.
Furthermore, for every column $z\in \Zgenlabel$, the column $r_{\max}(z)$ is to the right of $z$.
\end{prop}
\begin{proof}
We prove that, for every column $z\in \Zgenlabel$, $r_{\max}(z)$ is well-defined and to the right of $z$. We proceed by induction on a natural number $m \geq 0$, proving the result for all monomorphisms $f \colon Z \hookrightarrow X$ with presentations such that $\lvert \Zgenlabel\rvert = m$.

If $m=1$ and $\Zgenlabel=\{z\}$, then the algorithm sets $r_{\max}(z)$ to be the rightmost column in $\Xrellabel$ having a nonzero entry on the same row as a nonzero entry of $z$, which exists and is to the right of $z$ by Proposition \ref{prop:columnsCZ}.

Now suppose that the statement holds for every monomorphism presentation matrix with $m$ columns in $\Zgenlabel$. Let $M_f$ be a presentation matrix such that $\lvert \Zgenlabel\rvert =m+1$. Algorithm \ref{alg:btb} performs a `for' loop (line 4) until the `if' statement (line 6) is true, which by Proposition \ref{prop:columnsCZ} must happen before the algorithm terminates. Let $r_0$ be the rightmost column in $\Xrellabel$ such that there is a (minimal, i.e.\ leftmost) $z\in \Zgenlabel$ with $M_f(x,z)=1$, where $x$ is the row associated to $r_0$.
Again by Proposition \ref{prop:columnsCZ}, column $r_0$ is to the right of column $z$. The reductions in lines 11-14 of the algorithm transform $M_f$ into a matrix $M^*_f$ presenting a different monomorphism $f'\colon Z\hookrightarrow X$.
The morphism $f'$ coincides with $f$ on all generators of $Z$ except for generator $z'$, which is mapped to the nonzero element $f_{a^{z'}} (z') + f_{a^{z'}} (Z_{a^z \le a^{z'}} (z))$, where $a^z$ and $a^{z'}$ respectively denote the degrees of $z$ and $z'$.  
We see that $f'$ is a monomorphism via the following pointwise argument. For every degree $a$, the linear function $f_a\colon Z_a \to X_a$ has $\ker f_a = 0$, hence it maps nonzero elements in $\{Z_{a^z_i \le a} (z_i)\}_{i=1}^m$ to linearly independent elements $\{y_j\}$ in $\text{span}(X_{a_j^x \leq a}(x_j))_{j = 1}^n$. We see that $f'_a\colon Z_a \to X_a$ satisfies the same linear independence property (which implies $\ker f'_a = 0$) because the set of image elements coincides, except for possibly an element $y'$ replaced by $y'+y$, where $y$ is a different element of the set.

In $M^*_f$, the only column in $\Zgenlabel$ with nonzero entry in row $x$ is $z$.
By removing column $z$ and row $x$, we obtain a matrix with $m$ columns in $\Zgenlabel$ which is again a presentation matrix of a monomorphism. 
By induction hypothesis we know that the algorithm determines a function $r'_{\max} \colon  \Zgenlabel\setminus \{z\} \to \Xrellabel$ whose image does not contain $r_0$ and the columns to its right. The function $r'_{\max}$ extends to a function $r_{\max} \colon  \Zgenlabel \to \Xrellabel$ by defining $r_{\max}(z)\coloneqq r_0$. Finally, we observe that the function $r_{\max}$ is injective by construction.
\end{proof}

Let us now go back to the reduction of presentation matrices.
As with $M_f$, we can reduce $M_b$ by left-to-right column transformations to get a reduced matrix $\bar M_b$. We denote by $\sigma_b$ the permutation on $\{1,\ldots ,n\}$ associated with the nonzero columns of $\bar M_b$, which is well-defined because the matrix $M_b$ only has columns with at most one nonzero coefficient and has the same set of columns in $\Xrellabel$ as $M_f$. In our running example, computing $\bar M_b$ gives us $\sigma_b = [453261]$. 

After reduction via left-to-right column operations, the matrices $\bar M_f$ and $\bar M_b$ have non-zero columns with the same set of labels, as we will prove in Proposition \ref{P:PermutationInequality}.

\begin{defn}
Let $n \geq 1$ be an integer and $\sigma$ a permutation on $\{1, \ldots, n\}$. An \textbf{inversion} of $\sigma$ is a pair $(i, j)$ of elements of $\{1, \ldots, n\}$ such that $i < j$ and $\sigma(i) > \sigma(j)$. 

Given a permutation $\sigma$, we also give the name \textbf{inversion} to a transposition $(i \, j)$ such that $i < j$ and $\sigma(i) < \sigma(j)$: composing $\sigma$ by $(i \, j)$ on the right creates an inversion.
\end{defn}

Using inversions we can define a poset structure on permutations: we write $\sigma \leq \sigma'$ if there exist $k\ge 0$ and a composition of transpositions $\tau = \tau_1 \cdots \tau_k$ such that $\sigma \tau = \sigma'$ and, for all $i \leq k$, $\tau_i$ is an inversion of the permutation $\sigma \tau_1 \cdots \tau_{i - 1}$. 
In what follows, we often call $\tau$ simply a \textbf{composition of inversions of} $\sigma$ when it satisfies this property. Notice that $\le$ is a partial order on $S_n$, the symmetric group on $\{1,\ldots ,n\}$. With respect to this order, the identity permutation is the smallest element and the reverse permutation $[n \; n-1 \; \ldots \; 2 \; 1]$ is the largest element. 

\begin{prop}
\label{P:PermutationInequality}
Let $f\colon Z \hookrightarrow X$ be a monomorphism, $M_f$ be a presentation matrix of $f$ and $M_b$ be the bar-to-bar matrix computed via Algorithm \ref{alg:btb}. Let $\bar M_f$ and $\bar M_b$ be reduced matrices obtained from $M_f$ and $M_b$ respectively, and let $\sigma_f$ and $\sigma_b$ be the associated permutations. Then, the following facts hold: 
\begin{itemize}
\item the nonzero columns of the reduced matrices $\bar M_f$ and $\bar M_b$ are in the same positions, 
\item $\sigma_f \geq \sigma_b$, that is, $\sigma_f = \sigma_b \tau$ with $\tau$ a composition of inversions of $\sigma_b$.
\end{itemize}
\end{prop}
\begin{proof}
Since we can replace $M_f$ with the output $M_f^*$ of Algorithm \ref{alg:btb}, which has the same associated permutation $\sigma_f$ (as it is obtained by partially reducing $M_f$), and following the proof of Proposition \ref{prop:rmax}, we can assume that $M_f$ satisfies the following property: 
let $z_0$ be the unique column of $M_f$ in $\Zgenlabel$ such that the column $r_0 \coloneqq r_{\max}(z_0)$ is maximal in the total order on columns; then the only row $x_0$ such that $M_f(x_0,r_{\max}(z_0)) = 1$ has exactly one other nonzero entry, which is $M_f(x_0,z_0) = 1$.
We prove the claims by induction on the number of columns in $\Zgenlabel$. 

If $\Zgenlabel = \varnothing$, then there is nothing to prove: $M_f = M_b$ and they are reduced, so $\sigma_f = \sigma_b$.

Otherwise, we execute Algorithm \ref{alg:btb} to get the bar-to-bar matrix $M_b$ and the function $r_{\max}$. 
Let $z_0$ be the unique column of $M_f$ in $\Zgenlabel$ such that $r_0 \coloneqq  r_{\max}(z_0)$ is maximal in the total order on columns.
By removing column $z_0$, we obtain a presentation matrix $M_f'$ of a monomorphism $f'$ with a set of columns $\Zgenlabel'$ strictly contained in $\Zgenlabel$, to which we can apply our induction hypothesis: $\bar M_f'$ and $\bar M_b'$ have the same nonzero columns, and $\sigma_f' = \sigma_b' \tau$ for some composition of inversions $\tau$ of $\sigma_b'$. 
The matrix $M_b'$, computed by using Algorithm \ref{alg:btb} on $M_f'$, can be equivalently obtained by removing column $z_0$ from $M_b$, since $M_f$ satisfies the property stated at the beginning of the proof. See Example \ref{E:InductionMatrices} for matrices $M_f'$, $M_b'$, $\bar M_f'$ and $\bar M_b'$ in the running example.

Let $x_0$ be the only row such that $M_f(x_0,r_0)=1$. By the execution of Algorithm \ref{alg:btb}, no other column of $M_f'$ has a nonzero coefficient on row $x_0$, and so we deduce that the reductions of the matrices $M_f'$ and $M_b'$ do not affect column $r_0$.
Since by inductive hypothesis $\bar{M}_f'$ and $\bar{M}_b'$ have the same nonzero columns, this implies that column $r_0$ does not appear in the inversions of $\tau$,  
meaning that $\tau = (s_1 \, t_1)(s_2 \, t_2)\cdots (s_{k} \, t_{k})$ with $s_i \ne c_{r_0}'$ and $t_i \ne c_{r_0}'$ for all $i\in \{1,\ldots ,k \}$, where $c_{r_0}'$ denotes the relative position in $\{1,\ldots ,n\}$ of column $r_0$ in the (totally ordered) set of nonzero columns of the reduced matrix $\bar M_b'$.

Now, let $M_{g}$ be the matrix $M_f$ where we modify the column $z_0$ by setting to zero all its entries except the one on row $x_0$. We reduce the matrix $M_f$ first as for $M_f'$, and then we reduce the column $z_0$ by columns to its left, which does not affect the nonzero coefficient on row $x_0$: we denote the resulting matrix by $M_f''$. $\bar M_f$ is then obtained by completing the reduction using column $z_0$.
We reduce $M_{g}$ and $M_b$ in similar fashion, following $M_f'$ and $M_b'$, respectively. We observe the following facts.
\begin{itemize}
\item The nonzero columns of $\bar M_f$, $\bar M_{g}$, and $\bar M_b$ are the nonzero columns of $\bar M_f'$ and $\bar M_b'$, except we replace $r_0$ with $z_0$. This is clear by construction for the matrices $\bar M_{g}$ and $\bar M_b$, as the column $z_0$ coincides with $r_0$. 
For the matrix $\bar M_f$, observe that for every nonzero entry $M_f(x,z_0)$ on column $z_0$, there is a nonzero entry $M_f(x,r)$ in a column $r$ to the left of $r_0$, which implies that $r_0$ gets zeroed out after the reduction as it is linearly dependent with a number of columns to its left. 

\item $\sigma_f = \sigma_{g} \tau'$ where $\tau' \coloneqq  (c_{z_0} \, c_1) (c_1 \, c_2) \cdots (c_{k - 1} \, c_k)$ and $c_1, \ldots, c_k, c_{r_0}$ are the relative positions in $\{1,\ldots ,n\}$ of the nonzero columns of $M_f''$ 
whose lowest nonzero entry is modified (is moved to a different row) when reducing to $\bar M_f$, with $c_{z_0}$ and $c_{r_0}$ respectively denoting the relative positions of column $z_0$ and $r_0$ in the set of nonzero columns of $M_f''$. 
\item $\sigma_{g} = \sigma_f' \gamma^{-1}$ and $\sigma_b = \sigma_b' \gamma^{-1}$ where $\gamma \coloneqq (c_{z_0} \, c_{z_0}+1 \, \cdots \, c_{r_0})$ represents a cyclic permutation of the nonzero columns between $z_0$ and $r_0$.
\end{itemize}
See Example \ref{E:InductionMatrices} for concrete examples of these relationships. We deduce that
\begin{align*}
\sigma_f & = \sigma_{g} \tau' \\
& = \sigma_f' \gamma^{-1} \tau' \\
& = \sigma_b' \tau \gamma^{-1} \tau' \\
& = \sigma_b \gamma \tau \gamma^{-1} \tau'.
\end{align*}
By the definition of $\tau'$, it is a composition of inversions of $\sigma_{g}$. We conclude the induction step by showing that $\gamma \tau \gamma^{-1}$ is a composition of inversions of $\sigma_b$. 

More precisely, we know that $\tau = (s_1 \, t_1) \cdots (s_k \, t_k)$ is a composition of inversions of $\sigma'_b$, meaning that $(s_i \, t_i)$ is an inversion of the permutation $\sigma'_b (s_1 \, t_1) \cdots (s_{i-1} \, t_{i-1})$, for every $i\in \{1,\ldots ,k\}$, and we want to prove that $\gamma \tau \gamma^{-1} = (\gamma (s_1) \, \gamma (t_1)) \cdots (\gamma (s_k) \, \gamma (t_k))$ is a composition of inversions of $\sigma_b$, meaning that $(\gamma (s_i) \, \gamma (t_i))$ is an inversion of the permutation $\sigma_b (\gamma (s_1) \, \gamma (t_1)) \cdots (\gamma (s_{i-1}) \, \gamma (t_{i-1}))$, for every $i\in \{1,\ldots ,k\}$. 
First, we observe that $s_i < t_i$ implies $\gamma (s_i)< \gamma (t_i)$, since as observed earlier the relative position $c_{r_0}'$ of column $r_0$ in the set of nonzero columns of $\bar M_b'$ does not appear in $\tau$. 
Let us now denote 
\begin{gather*}
    \sigma'_{i-1} \coloneqq  \sigma'_b (s_1 \, t_1) \cdots (s_{i-1} \, t_{i-1}) , \\
    \sigma_{i-1} \coloneqq  \sigma_b (\gamma (s_1) \, \gamma (t_1)) \cdots (\gamma (s_{i-1}) \, \gamma (t_{i-1})) .
\end{gather*}
We have to prove that $\sigma'_{i-1}(s_i) < \sigma'_{i-1} (t_i)$ implies $\sigma_{i-1}(\gamma (s_i)) < \sigma_{i-1} (\gamma (t_i))$. This is a consequence of the equalities
\begin{equation*}
\sigma_{i-1}(\gamma (s_i)) = \sigma_{b} \gamma  (s_1 \, t_1) \gamma^{-1} \gamma  \cdots \gamma^{-1} \gamma (s_{i-1} \, t_{i-1}) \gamma^{-1} \gamma  (s_i) = \sigma'_{i-1} (s_{i})
\end{equation*}
and of similar equalities for $t_i$.
\end{proof}

\begin{ex}
\label{E:InductionMatrices}
\def\matrixscale{.6}
From the example matrix $M_f$ in \eqref{E:AssociatedMatrix}, the induction hypothesis of Proposition \ref{P:PermutationInequality} with $\Zgenlabel' = \Zgenlabel \setminus \{z_3\}$ gives the matrices
\begin{equation*}
M_f' = \scalebox{\matrixscale}{$\left[ \enspace
\begin{matrix}
\tikzmark{a1}{0} & \hphantom{\fbox{1}} & \tikzmark{a2}{0} & \hphantom{0}  &   &   &   & 1 &   \\
0 &   & 0 &   &   &   &   &   & 1 \\
0 &   & 1 &   & 1 &   &   &   &   \\
1 &   & 0 &   &   & 1 &   &   &   \\
1 & 1 & 0 &   &   &   &   &   &   \\
\tikzmark{b1}{0} &   & \tikzmark{b2}{0} &  &   &   & 1 &   &   \\
\end{matrix}
\enspace \right]
\begin{tikzpicture}[overlay, remember picture]
\draw (a1.north west) rectangle (b1.south east);
\draw (a2.north west) rectangle (b2.south east);
\end{tikzpicture}$},
M_b' = \scalebox{\matrixscale}{$\left[ \enspace
\begin{matrix}
\tikzmark{a1}{0} & \hphantom{\fbox{1}} & \tikzmark{a2}{0} & \hphantom{0}  &   &   &   & 1 &   \\
0 &   & 0 &   &   &   &   &   & 1 \\
0 &   & 1 &   & 1 &   &   &   &   \\
1 &   & 0 &   &   & 1 &   &   &   \\
0 & 1 & 0 &   &   &   &   &   &   \\
\tikzmark{b1}{0} &   & \tikzmark{b2}{0} &  &   &   & 1 &   &   \\
\end{matrix}
\enspace \right],
\begin{tikzpicture}[overlay, remember picture]
\draw (a1.north west) rectangle (b1.south east);
\draw (a2.north west) rectangle (b2.south east);
\end{tikzpicture}$}
\end{equation*}
where the column $z_3$ is omitted, and the reduced matrices
\begin{equation*}
\bar{M}_f' = \scalebox{\matrixscale}{$\left[ \enspace
\begin{matrix}
\tikzmark{a1}{0} &   & \tikzmark{a2}{0} &  &   &   &   & \fbox{1} &   \\
0 &   & 0 &   &   &   &   &   & \fbox{1} \\
0 &   & \fbox{1} &   & 0 &   &   &   &   \\
1 & \fbox{1} & 0 &   &   & 0 &   &   &   \\
\fbox{1} & 0 & 0 &   &   &   &   &   &   \\
\tikzmark{b1}{0} &   & \tikzmark{b2}{0} &  &   &   & \fbox{1} &   &   \\
\end{matrix}
\enspace \right]
\begin{tikzpicture}[overlay, remember picture]
\draw (a1.north west) rectangle (b1.south east);
\draw (a2.north west) rectangle (b2.south east);
\end{tikzpicture}$},
\bar{M}_b' = \scalebox{\matrixscale}{$\left[ \enspace
\begin{matrix}
\tikzmark{a1}{0} &   & \tikzmark{a2}{0} &  &   &   &   & \fbox{1} &   \\
0 &   & 0 &   &   &   &   &   & \fbox{1} \\
0 &   & \fbox{1} &   & 0 &   &   &   &   \\
\fbox{1} &   & 0 &   &   & 0 &   &   &   \\
0 & \fbox{1} & 0 &   &   &   &   &   &   \\
\tikzmark{b1}{0} &   & \tikzmark{b2}{0} &  &   &   & \fbox{1} &   &   \\
\end{matrix}
\enspace \right].
\begin{tikzpicture}[overlay, remember picture]
\draw (a1.north west) rectangle (b1.south east);
\draw (a2.north west) rectangle (b2.south east);
\end{tikzpicture}$}
\end{equation*}
We find that $\sigma_f' = [543612]$ and $\sigma_b' = [453612]$, and so $\sigma_f' = \sigma_b' (1 \, 2)$, where $(1, 2)$ is indeed an inversion.

The induction step of Proposition \ref{P:PermutationInequality} reduces the matrices
\begin{equation*}
M_f = \scalebox{\matrixscale}{$\left[ \enspace
\begin{matrix}
\tikzmark{a1}{0} & \hphantom{\fbox{1}} & \tikzmark{a2}{0} & 1 & \hphantom{\fbox{1}} &   &   & 1 &   \\
0 &   & 0 & 1 &   &   &   &   & 1 \\
0 &   & 1 & 0 & 1 &   &   &   &   \\
1 &   & 0 & 0 &   & 1 &   &   &   \\
1 & 1 & 0 & 0 &   &   &   &   &   \\
\tikzmark{b1}{0} &   & 0 & \tikzmark{b2}{1} &   &   & 1 &   &   \\
\end{matrix}
\enspace \right]
\begin{tikzpicture}[overlay, remember picture]
\draw (a1.north west) rectangle (b1.south east);
\draw (a2.north west) rectangle (b2.south east);
\end{tikzpicture}$},
M_g = \scalebox{\matrixscale}{$\left[ \enspace
\begin{matrix}
\tikzmark{a1}{0} & \hphantom{\fbox{1}} & \tikzmark{a2}{0} & 0 & \hphantom{\fbox{1}} &   &   & 1 &   \\
0 &   & 0 & 1 &   &   &   &   & 1 \\
0 &   & 1 & 0 & 1 &   &   &   &   \\
1 &   & 0 & 0 &   & 1 &   &   &   \\
1 & 1 & 0 & 0 &   &   &   &   &   \\
\tikzmark{b1}{0} &   & 0 & \tikzmark{b2}{0} &   &   & 1 &   &   \\
\end{matrix}
\enspace \right]
\begin{tikzpicture}[overlay, remember picture]
\draw (a1.north west) rectangle (b1.south east);
\draw (a2.north west) rectangle (b2.south east);
\end{tikzpicture}$},
M_b = \scalebox{\matrixscale}{$\left[ \enspace
\begin{matrix}
\tikzmark{a1}{0} & \hphantom{\fbox{1}} & \tikzmark{a2}{0} & 0 & \hphantom{\fbox{1}} &   &   & 1 &   \\
0 &   & 0 & 1 &   &   &   &   & 1 \\
0 &   & 1 & 0 & 1 &   &   &   &   \\
1 &   & 0 & 0 &   & 1 &   &   &   \\
0 & 1 & 0 & 0 &   &   &   &   &   \\
\tikzmark{b1}{0} &   & 0 & \tikzmark{b2}{0} &   &   & 1 &   &   \\
\end{matrix}
\enspace \right],
\begin{tikzpicture}[overlay, remember picture]
\draw (a1.north west) rectangle (b1.south east);
\draw (a2.north west) rectangle (b2.south east);
\end{tikzpicture}$}
\end{equation*}
to
\begin{equation*}
\bar{M}_f = \scalebox{\matrixscale}{$\left[ \enspace
\begin{matrix}
\tikzmark{a1}{0} &   & 0 & \tikzmark{a2}{1} &   &   &   & \fbox{1} &   \\
0 &   & 0 & 1 &   &   & \fbox{1} &   & 0 \\
0 &   & \fbox{1} & 0 & 0 &   &   &   &   \\
1 & \fbox{1} & 0 & 0 &   & 0 &   &   &   \\
\fbox{1} & 0 & 0 & 0 &   &   &   &   &   \\
\tikzmark{b1}{0} &   & \tikzmark{b2}{0} & \fbox{1} &   &   & 0 &   &   \\
\end{matrix}
\enspace \right]
\begin{tikzpicture}[overlay, remember picture]
\draw (a1.north west) rectangle (b1.south east);
\draw (a2.north east) rectangle (b2.south west);
\end{tikzpicture}$},
\bar{M}_g = \scalebox{\matrixscale}{$\left[ \enspace
\begin{matrix}
\tikzmark{a1}{0} &   & \tikzmark{a2}{0} & 0 &   &   &   & \fbox{1} &   \\
0 &   & 0 & \fbox{1} &   &   &   &   & 0 \\
0 &   & \fbox{1} & 0 & 0 &   &   &   &   \\
1 & \fbox{1} & 0 & 0 &   & 0 &   &   &   \\
\fbox{1} & 0 & 0 & 0 &   &   &   &   &   \\
\tikzmark{b1}{0} &   & 0 & \tikzmark{b2}{0} &   &   & \fbox{1} &   &   \\
\end{matrix}
\enspace \right]
\begin{tikzpicture}[overlay, remember picture]
\draw (a1.north west) rectangle (b1.south east);
\draw (a2.north west) rectangle (b2.south east);
\end{tikzpicture}$},
\bar{M}_b = \scalebox{\matrixscale}{$\left[ \enspace
\begin{matrix}
\tikzmark{a1}{0} &   & \tikzmark{a2}{0} & 0 &   &   &   & \fbox{1} &   \\
0 &   & 0 & \fbox{1} &   &   &   &   & 0 \\
0 &   & \fbox{1} & 0 & 0 &   &   &   &   \\
\fbox{1} &   & 0 & 0 &   & 0 &   &   &   \\
0 & \fbox{1} & 0 & 0 &   &   &   &   &   \\
\tikzmark{b1}{0} &   & 0 & \tikzmark{b2}{0} &   &   & \fbox{1} &   &   \\
\end{matrix}
\enspace \right].
\begin{tikzpicture}[overlay, remember picture]
\draw (a1.north west) rectangle (b1.south east);
\draw (a2.north west) rectangle (b2.south east);
\end{tikzpicture}$}
\end{equation*}
We find that $\sigma_f = [543621]$, $\sigma_{g} = [543261]$, and $\sigma_b = [453261]$. Thus $\sigma_f = \sigma_{g} (4 \, 5)$, $\sigma_{g} = \sigma_f' (4 \, 5 \, 6)$, and $\sigma_b = \sigma_b' (6 \, 5 \, 4)$.
\end{ex}

\begin{coro}
\label{coro:coker_f_b}
Let $f\colon Z\hookrightarrow X$ be a monomorphism, and let $f_b\colon Z\hookrightarrow X$ be the associated bar-to-bar monomorphism. Let $a_1\le a_2 \le \ldots \le a_n$ be the start-points of the bars of $X$, and let $b_1\le b_2 \le \ldots \le b_n$ be the degrees of the non-zero columns of $\bar M_f$.  
Then 
\begin{equation*}
\coker f = \bigoplus_{j=1}^n K(a_j , b_{\sigma_f (j)}) \qquad \text{ and }
\qquad \coker f_b = \bigoplus_{j=1}^n K(a_j , b_{\sigma_b (j)}) .
\end{equation*}
\end{coro}
\begin{proof}
By Proposition \ref{P:PermutationInequality}, the real numbers $b_1\le b_2 \le \ldots \le b_n$ are also the degrees of the non-zero columns of $\bar M_b$. By design of the persistent homology algorithm, the barcode decomposition of $\coker f$ and $\coker f_b$ is then determined by pairing start-points $\{a_j\}$ with end-points $\{b_j\}$ following the permutations $\sigma_f$ and $\sigma_b$ respectively, and the claim follows.
\end{proof}

We state below the rearrangement inequality following \cite{vince1990rearrangement}. Since the statement we need is slightly different from those we found in the literature, we include a short proof, which is a slight modification of the argument in \cite{vince1990rearrangement} and can be found also in \cite[p.~82]{steele2004cauchy}.

\begin{thm}[Rearrangement inequality]
\label{thm:rearr_ineq}
Let $g_1,g_2,\ldots ,g_n$ be real valued functions defined on an interval $I\subseteq \mathbb{R}$ such that
$g_{k+1} -g_k$ is a non-decreasing function, for all $k\in \{1,\ldots,n-1\}$,
and let $b_1\le b_2\le  \ldots \le  b_n$ be a sequence of elements of $I$. If $\rho \le \sigma$ in $S_n$, then
\begin{equation*}
\sum_{k=1}^n g_k (b_{\rho(k)}) \ge \sum_{k=1}^n g_k (b_{\sigma (k)}) .
\end{equation*}
\end{thm}

\begin{proof}
Since the argument we present can be iterated, it is enough to prove the statement for $\sigma = \rho \tau$ where $\tau = (i \, j)$ is an inversion: $i< j$ and $\rho (i) < \rho (j)$. We have
\begin{align*}
\sum_{k=1}^n g_k (b_{\rho(k)}) - \sum_{k=1}^n g_k (b_{\sigma (k)})  &= g_i(b_{\rho (i)}) + g_j(b_{\rho (j)}) - g_i(b_{\sigma (i)}) - g_j(b_{\sigma (j)}) \\
&= g_i(b_{\rho (i)}) + g_j(b_{\rho (j)}) - g_i(b_{\rho (j)}) - g_j(b_{\rho (i)}) \\
&= \left( g_j(b_{\rho (j)}) - g_i(b_{\rho (j)}) \right) - \left( g_j(b_{\rho (i)}) - g_i(b_{\rho (i)}) \right) \ge 0 ,
\end{align*}
where the last inequality follows from $b_{\rho (i)}\le b_{\rho (j)}$ and from the fact that $g_j - g_i$ is non-decreasing.
\end{proof}

\begin{coro}
\label{coro:barlength_ineq}
Let $a_1\le a_2\le  \ldots \le  a_n$ and $b_1\le b_2\le  \ldots \le  b_n$ be sequences of real numbers, and let $p\in [1,\infty )$.  If $\rho \le \sigma$ in $S_n$, then
\begin{equation*}
\sum_{k=1}^n \lvert a_k - b_{\rho (k)}\rvert^p \le \sum_{k=1}^n \lvert a_k - b_{\sigma (k)}\rvert^p .
\end{equation*}
\end{coro}
\begin{proof}
Let $h_k (x) = \lvert a_k -x\rvert^p$. It is easy to check that the function $h_{k+1} -h_k$ is non-increasing for all $k \in \{1, \ldots ,n-1\}$,
so we can apply Theorem \ref{thm:rearr_ineq} to the sequence of functions $g_k \coloneqq  -h_k$.
\end{proof}

\begin{thm}
\label{thm:mono-bar}
For any monomorphism $f \colon Z \hookrightarrow X$ it is possible to determine (via Algorithm \ref{alg:btb}) a bar-to-bar monomorphism $f_b \colon Z \hookrightarrow X$ such that $\| \coker f_b \|_p \le \| \coker f \|_p$, for all $p\in [1,\infty]$.
\end{thm}
\begin{proof}
First, assume $p\in [1,\infty)$. The persistence modules $\coker f$ and $\coker f_b$ have barcode decompositions as in Corollary \ref{coro:coker_f_b}. Then, the claim follows from Corollary \ref{coro:barlength_ineq} applied to the permutations $\sigma_b \le \sigma_f$ (Proposition \ref{P:PermutationInequality}). The claim for $p=\infty$ follows from taking the limit for $p\to \infty$ of both sides of the inequality $\left\| \coker b \right\|_p \le \left\| \coker f \right\|_p$, recalling that $\displaystyle\lim_{p\to\infty}\|u\|_p=\|u\|_\infty$ for any vector $u\in \mathbb{R}^n$ (Section \ref{subsec:pnorms}). 
\end{proof}

\subsection{Bar-to-bar epimorphisms}
\label{subsec:main_epi}
A similar result to Theorem \ref{thm:mono-bar} exists for epimorphisms and their kernels. To show this, we use a duality argument. The dualization of persistence modules has been studied extensively, see e.g.\ \cite{bauer2015induced,miller2020essential,BauerSchmahl2023}. Here we dualize tame persistence modules indexed by $[0, \infty)$, which requires some special handling since in the setting of this article bars are only supported on left-closed right-open intervals. In this subsection, we abuse the terminology introduced in Section \ref{subsec:persmod} by calling bars more general persistence modules with interval support and pointwise dimension at most $1$, including persistence modules supported on intervals of the form $(a,b]$ instead of $[a,b)$. To simplify the exposition, we explicitly work with finite direct sums of bars instead of general tame persistence modules, which are only equal to direct sums of bars up to isomorphism. 

\begin{defn}
Let $\mathsf{Bar}_{[0, \infty)}$ be the full subcategory of tame persistence modules whose objects are finite direct sums of bars. Similarly, let $\mathsf{Bar}_{(-\infty, 0]}$ be the category of finite direct sums of bars indexed by the poset $(-\infty, 0]$ with the usual order. These categories are abelian and we can represent morphisms as matrices of morphisms between summands.
\end{defn}

We consider the contravariant functor $(-)^\vee \colon \mathsf{Bar}_{[0, \infty)} \to \mathsf{Bar}_{(-\infty, 0]}$ sending an object $X$ in $\mathsf{Bar}_{[0, \infty)}$ to the functor $X^\vee \colon (-\infty, 0] \to \Vect_\mathit{K}$ defined, for all $s \leq t$ in $(-\infty, 0]$, by $X^\vee_t \coloneqq \hom(X_{-t}, K)$ and $X^\vee_{s \leq t} \coloneqq \hom(X_{-t \leq -s}, K)$. 
Similarly, a contravariant functor $\mathsf{Bar}_{[0, \infty)} \to \mathsf{Bar}_{(-\infty, 0]}$ is defined, which we also denote by $(-)^\vee$ with an abuse of notation. Both functors $(-)^\vee$ are exact. 

If $X=K(a,b)$, we observe that $X^\vee$ is also a bar, but its support (i.e., the set of poset elements for which the functor $X^\vee$ takes a nonzero value) is the left-open, right-closed interval $(-b,-a]$. In this article, we are considering bars whose support is a left-closed, right-open interval in $\mathbb{R}$. To fix this, we can consider the \textbf{pointwise direct limit} functor $\varinjlim_{[0,-)} (-) \colon \mathsf{Bar}_{[0, \infty)} \to \mathsf{Bar}_{[0, \infty)}$ sending $X$ to the persistence module whose value at $a$ is $\varinjlim_{t < a} X_t$ and whose transition functions are naturally defined. 
If $X=K(a,b)$, applying the pointwise direct limit functor yields the bar supported on $(a,b]$, whose dual $K(a,b)^\vee$ is the bar $K(-b,-a)$ in $\mathsf{Bar}_{(-\infty, 0]}$, which is supported on $[-b,-a)$. 
Similarly, one defines the pointwise direct limit functor $\varinjlim_{(-\infty,-)} (-) \colon \mathsf{Bar}_{(-\infty, 0]} \to \mathsf{Bar}_{(-\infty, 0]}$. The pointwise direct limit functors are exact. Applying the functor $(-)^\vee$ after the pointwise direct limit functor is therefore an exact functor, which we denote by $(-)^\dagger$.

To summarize, we have contravariant exact functors
\[
(-)^\dagger \colon \mathsf{Bar}_{[0, \infty)}  \to  \mathsf{Bar}_{(-\infty, 0]}  \qquad \text{and} \qquad  (-)^\dagger \colon \mathsf{Bar}_{(-\infty, 0]}  \to  \mathsf{Bar}_{[0, \infty)} 
\]
sending a bar $K(a,b)$ supported on $[a,b)$ to the bar $K(-b,-a)$ supported on $[-b,-a)$, and
extended to the rest of the objects by additivity. 
These functors send morphisms to their transpose (when viewed as matrices). More precisely, given a morphism $f \colon \bigoplus_i K(a_i,b_i) \to \bigoplus_j K(c_j,d_j)$ written as the matrix $[f_{i,j}]_{i,j}$ with $f_{i,j} \colon K(a_i,b_i) \to K(c_j,d_j)$ for all $i$ and $j$, the morphism $f^\dagger \colon \bigoplus_j K(-d_j,-c_j) \to \bigoplus_i K(-b_i,-a_i)$ can be written as the matrix $[f_{j, i}^\dagger]_{i, j}$.

As a consequence of the exactness of $(-)^\dagger$, for any morphism $f\colon X\to Y$ in $\mathsf{Bar}_{[0, \infty)}$ (resp. in $\mathsf{Bar}_{(-\infty,0]}$) we have
\[
(\ker f)^\dagger \cong \coker f^{\dagger} \qquad \text{and} \qquad (\coker f)^\dagger \cong \ker f^{\dagger}.
\]
Since the functors $(-)^\dagger$ send morphisms to their transpose, they send bar-to-bar morphisms to bar-to-bar morphisms. In particular, they send bar-to-bar monomorphisms to bar-to-bar epimorphisms, and vice-versa. It is also clear that the functors $(-)^\dagger$ preserve $p$-norms: $\| X^\dagger\|_p = \| X\|_p$. 

Moreover, we can apply the theory of Section \ref{subsec:main_mono} to the category $\mathsf{Bar}_{(-\infty, 0]}$, and in particular apply Theorem \ref{thm:mono-bar}. 
We conclude with the following result:

\begin{thm}
\label{thm:epi-bar}
For any epimorphism $f \colon Z \twoheadrightarrow X$ of persistence modules, it is possible to determine a bar-to-bar epimorphism $f_b \colon Z \twoheadrightarrow X$ such that $\left\| \ker f_b \right\|_p \le \left\| \ker f \right\|_p$, for all $p\in [1,\infty]$.
\end{thm}

\begin{proof}
As for the case of monomorphisms, we assume that $X$ and $Z$ are finite direct sums of finite bars. We apply Theorem \ref{thm:mono-bar} to $f^\dagger \colon X^\dagger \hookrightarrow Z^\dagger$ to get a bar-to-bar monomorphism $g$. We set $f_b \coloneqq g^\dagger \colon Z \twoheadrightarrow X$ a bar-to-bar epimorphism, and we observe that
\[
\| \ker f_b \| = \| \coker g \| \leq \| \coker f^\dagger \| = \| \ker f \|.
\]
\end{proof}

\subsection{Comparing bar-to-bar morphisms with induced matchings}
\label{subsec:canonical}

In \cite{bauer2015induced} the authors introduce a construction similar to bar-to-bar morphisms, namely \textbf{induced matchings}. In particular, given persistence modules $X$ and $Z$ such that there exists a monomorphism $Z \hookrightarrow X$, the authors define a \textbf{canonical injection} from the multiset of bars of $Z$ to the multiset of bars of $X$, where for all real numbers $b \in [0,\infty]$ the $i^{\text{th}}$ longest bar of $Z$ with end-point $b$ is sent to the $i^{\text{th}}$ longest bar of $X$ with the same end-point. This injection induces a bar-to-bar monomorphism, which we define here:

\begin{defn}\label{def:inducedmono}
Let $Z$ and $X$ be persistence modules such that there exists a monomorphism $Z \hookrightarrow X$, and fix barcode decompositions $Z = \bigoplus_{i = 1}^m K(a_i^z, b_i^z)$ and $X = \bigoplus_{j = 1}^n K(a_j^x, b_j^x)$. The induced matching of \cite{bauer2015induced} then corresponds to a canonical injection $\varphi \colon \{1, \ldots, m\} \hookrightarrow \{1, \ldots, n\}$.

We define the \textbf{monomorphism induced by the canonical injection $\varphi$} as the monomorphism $f_{\varphi}\colon Z \hookrightarrow X$ given by
\[
f_{\varphi} = \bigoplus_{i = 1}^m (K(a_i^z, b_i^z) \hookrightarrow K(a_{\varphi(i)}^x, b_{\varphi(i)}^x)) \oplus \bigoplus_{j\in \{1,\ldots ,n\}\setminus \image \varphi} (0 \hookrightarrow K(a_{j}^x, b_{j}^x)).
\]
Note that this is a bar-to-bar monomorphism.
\end{defn}

\begin{rmk}
Let $f \colon Z\hookrightarrow X$ be a monomorphism. If the bars of $X$ all have distinct end-points, then the monomorphism induced by the canonical injection (Definition \ref{def:inducedmono}) coincides, up to isomorphism, with the bar-to-bar monomorphism $f_b$ as determined in Section \ref{subsec:main_mono}. This is because, in this case, there is only one bar-to-bar monomorphism from $Z$ to $X$ (up to isomorphism).
\end{rmk}

\begin{rmk}
In general, this is not necessarily the case.
For example, starting from the monomorphism $f \colon K(2,3)\hookrightarrow K(1,3)\oplus K(2,3)$ defined by
\[
f=(K(2,3)\hookrightarrow K(2,3))\oplus (0\hookrightarrow K(1,3)),
\]
then we have $f_{\varphi}=(K(2,3)\hookrightarrow K(1,3))\oplus (0\hookrightarrow K(2,3))$, while $f_b =f$. 

More generally, monomorphisms induced by canonical injections do not commute with direct sums of monomorphisms \cite[Example 5.8]{bauer2015induced}, while the bar-to-bar monomorphisms we introduced in Section \ref{subsec:main_mono} do, as an immediate consequence of their definition via Algorithm \ref{alg:btb}. That is, given $f \colon Z\hookrightarrow X$ and $f' \colon Z'\hookrightarrow X'$, it is not true in general that $(f\oplus f')_{\varphi}=f_{\varphi}\oplus f'_{\varphi}$, while $(f\oplus f')_{b}=f_{b}\oplus f'_{b}$ holds. 
\end{rmk}

Maintaining a close relation between a given monomorphism $f \colon Z\hookrightarrow X$ and the bar-to-bar monomorphism $f_b$ is instrumental to proving Proposition \ref{P:PermutationInequality}, the key technical result of Section \ref{subsec:main_mono}, and in turn Theorem \ref{thm:mono-bar} and the dual Theorem \ref{thm:epi-bar}.
Since the canonical injection \cite{bauer2015induced} does not depend on the specific given monomorphism between $Z$ and $X$, but just on the existence of such a monomorphisms, it is not possible to prove Proposition \ref{P:PermutationInequality} using the same strategy with the canonical injection instead of the bar-to-bar monomorphism we introduced. 

We conclude this section with two corollaries to Theorem \ref{thm:mono-bar} and Theorem \ref{thm:epi-bar}. 
Using Theorem \ref{thm:mono-bar}, the problem of minimizing the $p$-norm of cokernels of all possible monomorphisms between two given persistence modules can be solved by restricting to bar-to-bar monomorphisms. This simplification allows for an easy combinatorial proof of the fact that, among all bar-to-bar monomorphisms, the one induced by the canonical injection minimizes the $p$-norm of the cokernel. A dual result holds for epimorphisms.

\begin{coro}\label{coro:mono-canonical}
Let $Z$ and $X$ be two persistence modules such that there exists a monomorphism $Z \hookrightarrow X$. Then $\min_{f \colon Z \hookrightarrow X} \| \coker f \|_p = \| \coker f_{\varphi} \|_p$ for all $p\in [1,\infty]$, where $f_{\varphi}$ is the bar-to-bar monomorphism induced by the canonical injection (Definition \ref{def:inducedmono}).
\end{coro}
\begin{proof}
By Theorem \ref{thm:mono-bar}, it suffices to show the inequality $\| \coker f_{\varphi} \|_p\le \| \coker f \|_p$ for bar-to-bar monomorphisms $f \colon Z\hookrightarrow X$. Let $\{b_k\}_{k=1}^\ell$ be the set of distinct end-points of $X$. Then, by hypothesis that there exists a monomorphism $Z \hookrightarrow X$, we can decompose $Z = \bigoplus_{k=1}^\ell Z^{(k)}$ and $X = \bigoplus_{k=1}^\ell X^{(k)}$ where $Z^{(k)}$ (resp.\ $X^{(k)}$) is the direct sum of the bars in $Z$ (resp.\ $X$) with end-point $b_k$. Given a bar-to-bar monomorphism $f \colon Z \hookrightarrow X$, this induces monomorphisms $f^{(k)} \colon Z^{(k)} \hookrightarrow X^{(k)}$. We then observe that $f = \bigoplus_{k=1}^\ell f^{(k)}$, and so
\[
\| \coker f\|_p =  \left\| \left( \| \coker f^{(k)} \|_p \right)_{k=1,\ldots, \ell} \right\|_p.
\]
Since $p$-norms are nondecreasing with respect to the coordinate-wise order on $[0,\infty)^\ell$ (Section \ref{subsec:pnorms}), proving that $\| \coker f^{(k)}_{\varphi} \|_p \leq \| \coker f^{(k)} \|_p$ for each $k \in \{1, \ldots, \ell\}$ implies that $\| \coker f_{\varphi} \|_p \leq \| \coker f \|_p$.

Thus it suffices to prove the result in the case where the bars of $Z$ and $X$ all have the same end-point, which is what we assume for the rest of the proof. Denote by $b_0$ this common end-point and write $Z = \bigoplus_{i = 1}^m K(a_i^z, b_0)$ and $X = \bigoplus_{j = 1}^n K(a_j^x, b_0)$, where the $a_i^z$ and $a_j^x$ are in nondecreasing order. Define $a_i^z = b_0$ for $i \in \{m + 1, \ldots, n\}$.  Then every bar-to-bar monomorphism $f \colon Z \hookrightarrow X$ has the form
\[
f = \bigoplus_{i = 1}^n (K(a_i^z, b_0) \hookrightarrow K(a_{\alpha(i)}^x, b_0)),
\]
where $\alpha \colon \{1,\ldots, n\} \to \{1,\ldots ,n\}$ is a permutation and $K(b_0, b_0)$ denotes the zero module. By Remark \ref{rmk:ker_coker_bartobar}, 
\[
\coker f =  \bigoplus_{i = 1}^n K(a_{\alpha(i)}^x, a_i^z). 
\]
In particular, the bar-to-bar monomorphism induced by the canonical injection corresponds to the permutation $\alpha = \id$. For all $p \in [1, \infty)$, we then apply the rearrangement inequality of Corollary \ref{coro:barlength_ineq} to deduce that $\| \coker f_\varphi \|_p \leq \| \coker f \|_p$. The case $p = \infty$ follows by taking the limit (as in the proof of Theorem \ref{thm:mono-bar}).
\end{proof}

We now dualize the definitions and results for the case of epimorphisms.
An epimorphism $f \colon Z\twoheadrightarrow X$ induces a canonical injection \cite{bauer2015induced} from the multiset of bars of $X$ to the multiset of bars of $Z$, where for all $a\in [0,\infty]$, the $i^\text{th}$ longest bar of $X$ with start-point $a$ is sent to the $i^\text{th}$ longest bar of $Z$ with the same start-point. 

\begin{defn}\label{def:inducedepi}
Let $Z$ and $X$ be persistence modules such that there exists an epimorphism $Z \twoheadrightarrow X$, and fix barcode decompositions $Z = \bigoplus_{i = 1}^m K(a_i^z, b_i^z)$ and $X = \bigoplus_{j = 1}^n K(a_j^x, b_j^x)$. The induced matching of \cite{bauer2015induced} then corresponds to a canonical injection $\psi \colon \{1, \ldots, n\} \hookrightarrow \{1, \ldots, m\}$.
We define the \textbf{epimorphism induced by the canonical injection} $\psi$ as the epimorphism
$f_{\psi}\colon Z \twoheadrightarrow X$ given by
\[
f_{\psi} = \bigoplus_{j = 1}^n (K(a_{\psi(j)}^z, b_{\psi(j)}^z) \twoheadrightarrow K(a_{j}^x, b_{j}^x)) \oplus \bigoplus_{i\in \{1,\ldots ,m\}\setminus \image \psi} ( K(a_{i}^x, b_{i}^x)\twoheadrightarrow 0).
\]
\end{defn}

As in the case of monomorphisms, given an epimorphism $f \colon Z \twoheadrightarrow X$, the associated bar-to-bar epimorphisms $f_{\psi}$ and $f_b$ (see Section \ref{subsec:main_epi}) are not necessarily the same. 

\begin{coro}\label{coro:epi-canonical}
Let $Z$ and $X$ be two persistence modules such that there exists an epimorphism $Z \twoheadrightarrow X$. Then $\min_{f\colon Z\twoheadrightarrow X} \| \ker f \|_p = \| \ker f_{\psi} \|_p$ for all $p\in [1,\infty]$, where $f_{\psi}$ is a bar-to-bar epimorphism induced by the canonical injection (Definition \ref{def:inducedepi}).
\end{coro}
\begin{proof}
This proof is dual to that of Corollary \ref{coro:mono-canonical}.
By Theorem \ref{thm:epi-bar}, it suffices to show the inequality $\| \ker f_{\psi} \|_p\le \| \ker f \|_p$ for bar-to-bar epimorphisms $f \colon Z\twoheadrightarrow X$. Let $\{a_k\}_{k=1}^\ell$ be the set of distinct start-points of $Z$. We can decompose $Z = \bigoplus_{k=1}^\ell Z^{(k)}$ and $X = \bigoplus_{k=1}^\ell X^{(k)}$ where $Z^{(k)}$ (resp.\ $X^{(k)}$) is the direct sum of the bars in $Z$ (resp.\ $X$) with start-point $a_k$. Given a bar-to-bar epimorphism $f \colon Z \twoheadrightarrow X$, this induces epimorphisms $f^{(k)} \colon Z^{(k)} \twoheadrightarrow X^{(k)}$ such that $f = \bigoplus_{k=1}^\ell f^{(k)}$, hence
\[
\| \ker f\|_p =  \left\| \left( \| \ker f^{(k)} \|_p \right)_{k=1,\ldots, \ell} \right\|_p.
\]
As in the proof of Corollary \ref{coro:mono-canonical}, it is sufficient to prove the claim for each $f^{(k)}$, hence for the rest of the proof we assume that the bars of $Z$ and $X$ all have the same start-point. 
Denote by $a_0$ this common start-point and write $Z = \bigoplus_{i = 1}^m K(a_0, b_i^z)$ and $X = \bigoplus_{j = 1}^n K(a_0, b_j^x)$, where the $b_i^z$ and $b_j^x$ are in nonincreasing order. Define $b_j^x = a_0$ for $i \in \{n + 1, \ldots, m\}$. 
Then every bar-to-bar epimorphism $f \colon Z \twoheadrightarrow X$ has the form
\[
f = \bigoplus_{i = 1}^m (K(a_0,b_i^z) \twoheadrightarrow K(a_0,b_{\alpha(i)}^x)),
\]
where $\alpha \colon \{1,\ldots, m\} \to \{1,\ldots ,m\}$ is a permutation and $K(a_0, a_0)$ denotes the zero module. By Remark \ref{rmk:ker_coker_bartobar}, 
\[
\ker f =  \bigoplus_{i = 1}^m K(b_{\alpha(i)}^x,b_i^z). 
\]
In particular, the bar-to-bar epimorphism induced by the canonical injection corresponds to the permutation $\alpha = \id$. For all $p \in [1, \infty)$, we then apply the rearrangement inequality of Corollary \ref{coro:barlength_ineq} to deduce that $\| \ker f_\varphi \|_p \leq \| \ker f \|_p$. The case $p = \infty$ follows by taking the limit.
\end{proof}


\section{Noise systems and Wasserstein pseudometrics}
\label{sec:Wass_noise}

In this section we study algebraic Wasserstein pseudometrics between persistence modules. 
After introducing in Section \ref{subsec:pseudom_dp} a generalization of the pseudometrics associated with a noise system,  
we study in Section \ref{subsec:pnorms_C} noise systems determined by  $p$-norms of persistence modules and regular contours.
Section \ref{subsec:alg_p_pC} is devoted to the associated algebraic Wasserstein pseudometrics. 
For some choices of parameters, these pseudometrics have a combinatorial interpretation, as we show in Section \ref{subsec:alg_and_comb_WpC}. 
Finally, in Section \ref{subsec:alg_pC_distance} we present formulas to compute the algebraic Wasserstein pseudometric between persistence modules in some specific cases.

\subsection{Pseudometrics associated to noise systems}
\label{subsec:pseudom_dp}
Given a noise system $\ns$ and $p \in [1, \infty]$, in this section we will introduce pseudometrics $d^p_{\ns}$ between persistence modules.
These pseudometrics are a simple generalization to $p>1$ of the pseudometric associated to a noise system in \cite{scolamiero2017multidimensional} (see Section \ref{subsec:noise_induced_distances}), where $p=1$.
Although the statements in this section hold true for tame functors 
indexed by $[0,\infty)^r$ for every positive natural number $r$, as in \cite{scolamiero2017multidimensional}, we will limit the presentation to $r=1$, since this is the setting of the following sections.

\begin{defn}
Let $X$ and $Y$ be persistence modules. A \textbf{span} of $X,Y$ is a triplet $(Z,f, g)$ with $Z$ a persistence module and $f\colon  Z\to X$ and $g\colon Z \to Y$ morphisms between persistence modules. A span of $X,Y$ is therefore a diagram in $\Tame$ of the form
\[
X \xleftarrow{f} Z \xrightarrow{g} Y
\]
\end{defn}

\begin{defn}
\label{def:eps span}
Let $X$ and $Y$ be persistence modules, and let $\ns$ be a noise system. A span $X \xleftarrow{f} Z \xrightarrow{g} Y$ is called a \textbf{$(\varepsilon_1, \varepsilon_2, \varepsilon_3, \varepsilon_4)$-span} if 
\[
\ker f \in \ns_{\varepsilon_1}, \quad \coker f \in \ns_{\varepsilon_2}, \quad
\ker g \in \ns_{\varepsilon_3} \quad\text{and}\quad \coker g \in \ns_{\varepsilon_4.} 
\]
\end{defn}

\begin{defn}
\label{def:dpS}
Let $X$ and $Y$ be persistence modules, and let $\ns$ be a noise system. For $p\in [1,\infty]$ and $\varepsilon \in [0,\infty)$, we say that $X$ and $Y$ are $\varepsilon$-\textbf{close in $p$-norm} $\left\| \cdot \right\|_p$ if there exists a $(\varepsilon_1, \varepsilon_2, \varepsilon_3, \varepsilon_4)$-span $X \xleftarrow{f} Z \xrightarrow{g} Y$ for some $\varepsilon_1, \varepsilon_2, \varepsilon_3, \varepsilon_4 \in [0, \infty)$ such that $\left\| (\varepsilon_1, \varepsilon_2, \varepsilon_3, \varepsilon_4) \right\|_p \le \varepsilon$.
We define
\[
d^p_{\ns}(X,Y) \= \inf \left\{ \varepsilon \in [0,\infty) \mid \text{$X$ and $Y$ are $\varepsilon$-close in $p$-norm} \right\} ,
\]
adopting the convention $\inf \varnothing = \infty$. 
\end{defn}

Our next aim is to prove that $d^p_{\ns}$ is a pseudometric on $\Tame$. We start by generalizing Proposition 
$8.5$ in  \cite{scolamiero2017multidimensional} to our current framework. Even if the generalization is not difficult, we include the proof to highlight how the properties of $p$-norms on $\mathbb{R}^4$ are used. We note that a similar result can be obtained for a larger family of subadditive functions on $\mathbb{R}^4$ which include $p$-norms (see \cite{giunti2021amplitudes}, Section 2.1).
 
 \begin{prop}
\label{prop:additivity}
Let $F,G,H$ be persistence modules. Assume that $F$ and $G$ are $\varepsilon$-close in $p$-norm, and that $G$ and $H$ are $\tau$-close in $p$-norm. Then $F$ and $H$ are $(\varepsilon +\tau)$-close in $p$-norm.
\end{prop}
\begin{proof}
By assumption there exists a $(\varepsilon_1, \varepsilon_2, \varepsilon_3, \varepsilon_4)$-span $F \xleftarrow{f'} X \xrightarrow{f''} G$ with $\varepsilon_1, \varepsilon_2, \varepsilon_3, \varepsilon_4 \in [0, \infty)$ such that $\left\| (\varepsilon_1, \varepsilon_2, \varepsilon_3, \varepsilon_4) \right\|_p \le \varepsilon$ and a $(\tau_1, \tau_2, \tau_3, \tau_4)$-span $G \xleftarrow{g'} Y \xrightarrow{g''} H$ with $\tau_1, \tau_2, \tau_3, \tau_4 \in [0, \infty)$ such that $\left\| (\tau_1, \tau_2, \tau_3, \tau_4) \right\|_p \le \tau$.
Consider the following diagram, where the square is a pullback:
\[
\begin{tikzcd}[column sep=small, row sep=small]
& & Z \arrow[dl, "f"'] \arrow[dr, "g"] & & \\
& X \arrow[dl, "f'"'] \arrow[dr, "f''"] & & Y \arrow[dl, "g'"'] \arrow[dr, "g''"] & \\
F & & G & & H
\end{tikzcd}
\]
By \cite[Proposition 8.1]{scolamiero2017multidimensional}, $\ker f \in \ns_{\tau_1}$ and $\coker f \in \ns_{\tau_2}$, hence by \cite[Proposition 8.2]{scolamiero2017multidimensional} $\ker f' f \in \ns_{\varepsilon_1 +\tau_1}$ and $\coker f' f \in \ns_{\varepsilon_2 +\tau_2}$. By a similar argument, $\ker g'' g \in \ns_{\varepsilon_3 +\tau_3}$ and $\coker g'' g \in \ns_{\varepsilon_4 +\tau_4}$. This proves that $F$ and $H$ are $\eta$-close in $p$-norm, where $\eta \= \left\|(\varepsilon_1 +\tau_1,\varepsilon_2 +\tau_2,\varepsilon_3 +\tau_3,\varepsilon_4 +\tau_4)\right\|_p$. Our claim follows from the inequality
\begin{align*}
\left\| (\varepsilon_1 +\tau_1,\varepsilon_2 +\tau_2,\varepsilon_3 +\tau_3,\varepsilon_4 +\tau_4)\right\|_p & \le \left\| (\varepsilon_1 ,\varepsilon_2 ,\varepsilon_3 ,\varepsilon_4 )\right\|_p + \left\|(\tau_1 ,\tau_2 ,\tau_3 ,\tau_4 )\right\|_p \\
& \leq \varepsilon +\tau
\end{align*}
which expresses the subadditivity of $\left\| \cdot \right\|_p$ and the hypotheses.  
\end{proof}

We are now ready to prove that $d^p_{\ns}$ is a pseudometric on $\Tame$.
\begin{prop}
\label{prop:dp_pseudometric}
Given $p\in [1,\infty]$ and a noise system $\ns$, the function $d^p_{\ns}$ in Definition \ref{def:dpS} is a pseudometric  on $\Tame$ (see Section \ref{subsec:noise_induced_distances}). 
\end{prop}
\begin{proof} 
If $g\colon X\to Y$ is an isomorphism of persistence modules, the span $X \xleftarrow{\id} X \xrightarrow{g} Y$ shows that $d^p_{\ns}(X,Y)=0$. For all persistence modules $X$ and $Y$, the bijection between spans $X \xleftarrow{f} Z \xrightarrow{g} Y$ between $X$ and $Y$ and spans $Y \xleftarrow{g} Z \xrightarrow{f} X$ between $Y$ and $X$ implies that $d^p_\ns (X,Y)=d^p_\ns (Y,X)$. Proposition \ref{prop:additivity} shows that the triangle inequality holds true.
\end{proof}

\begin{rmk}
\label{rmk:equiv_pq}
Given a noise system $\ns$, the pseudometrics $d^p_{\ns}$ for all $p\in [1,\infty]$ are strongly equivalent. Assuming $p \leq q$, for any pair of persistence modules $X,Y$ we have
\[
d^q_{\ns} (X,Y) \le d^p_{\ns} (X,Y) \le  4^{\left( \frac{1}{p} - \frac{1}{q}\right)} d^q_{\ns}(X,Y) ,
\]
as can be easily concluded from the properties on  $p$-norms on $\mathbb{R}^4$
stated in Section \ref{subsec:pnorms}.
\end{rmk}

\subsection{\texorpdfstring{$p$}{p}-norms of persistence modules and contours}
\label{subsec:pnorms_C}
The aim of this section is to introduce and study a generalization of the notion of $p$-norm of a persistence module (see Section \ref{subsec:pnorms}) first introduced in \cite{skraba2020wasserstein}, that coincides with the original definition if $C$ is the standard contour (see Section \ref{subsec:contours}).
 
\begin{defn} 
\label{def:pCnorm}
Let $C$ be a regular contour. For $p\in [1,\infty]$, define the \textbf{$(p,C)$-norm} of a persistence module $X\cong \bigoplus_{i=1}^k K(a_i, b_i)$ by
\begin{equation*}
\label{eq:def_pnorm_C}
\left\| X\right\|_{p,C} \=  \begin{cases}
\left( \sum_{i=1}^k \ell(a_i,b_i)^{p} \right)^{\frac{1}{p}} &\text{for $p\in [1,\infty)$},\\
\max \{ \ell(a_i,b_i) \}_{i=1}^k &\text{for $p=\infty$} ,
\end{cases}
\end{equation*}
where $\ell(a_i,b_i)$ denotes the lifetime of the bar $K(a_i,b_i)$ with respect to $C$ (see Section \ref{subsec:contours}).
\end{defn}
We see that $\left\| X\right\|_{p,C} $ does not depend on the choice of barcode decomposition for $X$. 
For $p\in [1,\infty]$ and $\varepsilon \in [0,\infty)$, consider the class of tame persistence modules
\begin{equation*}
\ns^{p,C}_{\varepsilon} \= \{ X \in \Tame \mid \left\| X \right\|_{p,C} \le \varepsilon \} ,
\end{equation*}
and denote $\ns^{p,C} \= \{ \ns^{p,C}_{\varepsilon} \}_{\varepsilon \in [0,\infty)}$.
If $D$ is the standard contour (see Section \ref{subsec:contours}), then $\ell (a_i,b_i) = b_i - a_i$ and we have $\left\| X \right\|_{p,D} =\left\| X \right\|_{p}$ and $\ns^{p,D}=\ns^{p}$.
The main result in this subsection is showing that $\ns^{p,C}$ is a noise system (see Section \ref{subsec:noise_systems_distances}) whenever $C$ is an action, for any $p\in [1,\infty]$. 
For the standard contour, this result together with Proposition \ref{prop:dp_pseudometric} provide an alternative proof to the one in \cite{skraba2020wasserstein} that the algebraic $p$-Wasserstein distance 
is a pseudometric, as will be later highlighted in Remark \ref{rmk:defs_dp}.

Given a contour $C$, the function $C(0,-)\colon  [0,\infty) \to [0,\infty)$ is nondecreasing (and it is an increasing bijection if the contour is regular). Hence it can be viewed as a functor from the poset category $[0,\infty )$ to itself, that can be effectively used to re-parameterize $[0,\infty )$. For any persistence module $X$, the composition of functors $T_C(X)\coloneqq XC(0,-)\colon [0,\infty)\to \Vect_K$ is a persistence module. As we will show, $T_C(X)$ is in $\Tame$ whenever $X$ is in $\Tame$ and $C$ is a regular contour (Corollary \ref{coro:barcode_reparam}). 
The assignment $X\mapsto T_C(X)$ can be extended to a functor $T_C \colon  \Tame \to \Tame$ sending a morphism $f\colon X \to Y$ of persistence modules to the morphism $T_C(f)\colon T_C(X) \to T_C(Y)$ defined as the natural transformation between $T_C(X)$ and $T_C(Y)$ whose component at $a\in [0,\infty )$ is $T_C(f)_a = f_{C(0,a)}\colon  X_{C(0,a)} \to Y_{C(0,a)}$.

We now explain the relationship between the barcode decompositions of $X$ and $T_C(X)$ when $C$ is a regular contour.

\begin{prop}
\label{prop:bar_reparam}
Let $C$ be a regular contour, and let $\ell$ be the associated lifetime function. Consider a bar $K(a,b)$. Then
\[
T_C(K(a,b)) \cong K(\ell(0,a),\ell(0,b)) .
\]
\end{prop}
\begin{proof}
The functor $T_C(K(a,b))\colon [0,\infty ) \to \Vect_K$ sends $c\le d$ in $[0,\infty )$ to the linear function 
\[
K(a,b)_{C(0,c)\le C(0,d)} \colon   K(a,b)_{C(0,c)} \to K(a,b)_{C(0,d)},
\]
which is the identity on $K$ if $a\le C(0,c)\le C(0,d)<b$ and the zero function otherwise. Since $C$ is regular, $\ell (0,-)$ is a strictly increasing function, hence the condition $a\le C(0,c)\le C(0,d)<b$ is equivalent to $\ell(0,a)\le c\le d< \ell(0,b)$.
\end{proof}

\begin{coro}
\label{coro:barcode_reparam}
Let $X$ be a tame persistence module with barcode decomposition $\bigoplus_{i=1}^k K(a_i,b_i)$, and let $C$ be a regular contour. Then $T_C(X) \cong \bigoplus_{i=1}^k K(\ell(0,a_i),\ell(0,b_i))$. 
In particular $T_C(X)$ is also in $\Tame$.

\end{coro}
\begin{proof}

Since direct sums in $\Tame$ are defined pointwise (Section \ref{subsec:persmod}), if $\{X_i\}_{i\in I}$ is a finite collection of persistence modules and $C$ is a regular contour, then $T_C(\bigoplus_{i\in I} X_i) \cong \bigoplus_{i\in I} T_C(X_i)$, this together with Proposition \ref{prop:bar_reparam}, gives the following.

\begin{align*} 
    T_C(X) &\cong T_C(\bigoplus_{i=1}^k K(a_i,b_i)) \\
    &\cong \bigoplus_{i=1}^k T_C (K(a_i,b_i)) \\
    &\cong \bigoplus_{i=1}^k K(\ell(0,a_i),\ell(0,b_i)). \tag{by Proposition \ref{prop:bar_reparam}}
\end{align*}

Given that bars  $K(\ell(0,a_i),\ell(0,b_i))$ are in $\Tame$ and tameness is preserved by finite direct sums, $T_C(X)$ is also in $\Tame$.
\end{proof}

We now show that the functor $T_C$ is exact.

\begin{prop}
\label{prop:shortexact_XC}
Let $0\to X\to Y\to Z\to 0$ be an exact sequence in $\Tame$, and let $C$ be a regular contour. Then the sequence $0\to T_C(X)\to T_C(Y)\to T_C(Z)\to 0$ is also exact in $\Tame$.
\end{prop}
\begin{proof}
Exactness in $\Tame$ is defined pointwise: $0\to X\to Y\to Z\to 0$ is exact if and only if $0\to X_a\to Y_a\to Z_a\to 0$ is exact in $\Vect_K$, for every $a\in [0,\infty )$. As a consequence, $0\to X_{C(0,b)}\to Y_{C(0,b)}\to Z_{C(0,b)}\to 0$ is exact in $\Vect_K$, for every $b\in [0,\infty )$, hence by definition the sequence $0\to T_C(X)\to T_C(Y)\to T_C(Z)\to 0$ is exact. The fact that the exact sequence $0\to T_C(X)\to T_C(Y)\to T_C(Z)\to 0$ is in $\Tame$ follows from Corollary \ref{coro:barcode_reparam} and how $T_C$ is defined on morphisms.
\end{proof}

\begin{rmk}
\label{rmk:TC_exact}
As is clear from its proof, Proposition \ref{prop:shortexact_XC} in fact holds for the precomposition of persistence modules by any increasing bijection of $[0,\infty)$. 
\end{rmk}

In the rest of the article, we will focus on contours that are regular and actions (see Section \ref{subsec:contours}), called regular actions for brevity. We prove here a simple but important property of regular actions, and the associated lifetime function $\ell$, which is used to prove the subsequent results.  
\begin{lem}
\label{lem:add_lifespans}
If $C$ is a regular action, then $\ell (a,c)= \ell (a,b)+\ell (b,c)$ for any $a\le b\le c$ in $[0,\infty )$.
\end{lem}
\begin{proof}
Let $a\le b\le c$. Using the definitions and the assumption that $C$ is an action, we have $C(C(a,\ell (a,b)), \ell (b,c))= C(a,\ell (a,b) + \ell (b,c))$. Again by definition, we observe that the left-hand side equals $c$, and that $c = C(a,\ell (a,b) + \ell (b,c))$ implies $\ell (a,c) =\ell (a,b)+\ell (b,c)$.
\end{proof}

\begin{prop}
\label{prop:XCp}
Let $X$ be a persistence module, let $p\in [1,\infty]$, and let $C$ be a regular action. Then $\left\| X \right\|_{p,C} =\left\| T_C(X) \right\|_{p}$.
\end{prop}
\begin{proof}
Let $X\cong \bigoplus_{i=1}^k K(a_i,b_i)$. For any fixed $p\in [1,\infty)$, we have
\begin{align*}
    \left\| T_C (X) \right\|_{p} &= \left( \sum_{i=1}^k ( \ell (0,b_i)-\ell (0,a_i))^{p} \right)^{\frac{1}{p}} \\
    &= \left( \sum_{i=1}^k \ell (a_i,b_i)^{p} \right)^{\frac{1}{p}} \\
    &= \left\| X \right\|_{p,C} ,
\end{align*}
where the first equality is by Corollary \ref{coro:barcode_reparam}, the second one is by Lemma \ref{lem:add_lifespans}, and the third one is by definition of $\left\| \cdot \right\|_{p,C}$. The case $p=\infty$ is similar. 
\end{proof}

We are now ready to prove that $\ns^{p,C}$, with $C$ a regular  action, satisfies the axioms in the definition of noise system (see Section \ref{subsec:noise_systems_distances}). 

\begin{lem}
\label{lem:ns_ses1}
Let $0\to X\to Y\to Z\to 0$ be an exact sequence in $\Tame$, and let $C$ be a regular contour. Then $\left\| X \right\|_{p,C} \le \left\| Y \right\|_{p,C}$ and $\left\| Z \right\|_{p,C} \le \left\| Y \right\|_{p,C}$.
\end{lem}

For the standard contour, our statement coincides with Lemma 7.8 in \cite{skraba2020wasserstein}, which is easily proven using the induced matchings \cite{bauer2015induced} for monomorphisms and epimorphisms of persistence modules. 
For the sake of completeness, we include the proof for $\| \cdot \|_{p,C}$, which does not present any additional difficulty.

\begin{proof}
The existence of a monomorphism from $X$ to $Y$ implies the existence of the bar-to-bar monomorphism $f_{\varphi}\colon X\hookrightarrow Y$ of Definition \ref{def:inducedmono}, induced by the canonical injection \cite{bauer2015induced}. The monomorphism $f_{\varphi}$ decomposes as a finite direct sum of monomorphisms of the form $K(a',b)\hookrightarrow K(a,b)$, with $a\le a'$, and of the form $0\hookrightarrow K(a,b)$. By monotonicity of contours, $a\le a'$ implies $\ell (a',b) \le \ell (a,b)$. The inequality $\left\| X \right\|_{p,C} \le \left\| Y \right\|_{p,C}$ then follows from the definition of $\left\| \cdot \right\|_{p,C}$, since every term in the expression for $\left\| X \right\|_{p,C}$ is upper bounded by a term in the expression for $\left\| Y \right\|_{p,C}$.

The proof of the inequality $\left\| Z \right\|_{p,C} \le \left\| Y \right\|_{p,C}$ is obtained similarly, using the epimorphism induced by the canonical injection (see Definition \ref{def:inducedepi}).
\end{proof}

\begin{lem}
\label{lem:ns_ses2}
Let $0\to X\to Y\to Z\to 0$ be an exact sequence in $\Tame$, and let $C$ be a regular action. Then $\left\| Y \right\|_{p,C} \le \left\| X \right\|_{p,C} + \left\| Z \right\|_{p,C}$.
\end{lem}
\begin{proof}
First, we prove the statement assuming that $C$ is the standard contour. Let $0\to X\xrightarrow{f} Y\xrightarrow{g} Z\to 0$ be a short exact sequence of persistence modules, and let us show that $\left\| Y \right\|_{p} \le \left\| X \right\|_{p} + \left\| Z \right\|_{p}$. 
We consider the monomorphism $f$ and observe that $Z \cong \coker f$ implies that $Z$ and $\coker f$ have the same barcode decomposition, hence $\left\| Z \right\|_{p} = \left\| \coker f \right\|_{p}$. Theorem \ref{thm:mono-bar} tells us that, among all monomorphisms between two fixed persistence modules, the norm $\left\| \cdot \right\|_{p}$ of the cokernel is minimized by a bar-to-bar monomorphism. We therefore just need to prove that $\left\| Y \right\|_{p} \le  \left\| X \right\|_{p} +  \left\| \coker f \right\|_{p}$, for any bar-to-bar monomorphism $f$ between $X$ and $Y$. 

By Remark \ref{rmk:ker_coker_bartobar}, if $f\colon X\to Y$ is a bar-to-bar monomorphism, then there exist barcode decompositions $\bigoplus_{i=1}^m X_i$ and $\bigoplus_{j=1}^n Y_j$ of $X$ and $Y$, respectively, such that $m\le n$ and, up to permutation of the $Y_j$, there are monomorphisms $f_i \colon X_i \to Y_i$ between bars such that $\coker f = \bigoplus_{i=1}^m \coker f_i \oplus \bigoplus_{j=m+1}^n Y_j$. We observe that, for each bar $Y_i=K(a_i,b_i)$ of $Y$ with $i\in \{1,\dots ,m\}$, there is a bar $X_i=K(a'_i,b_i)$ of $X$ and a corresponding summand $\coker f_i$ of $\coker f$, which is a bar $K(a_i,a'_i)$ if $a_i<a'_i$, and it is the zero module if $a_i=a'_i$. Similarly, we observe that each bar $Y_j=K(a_j,b_j)$ of $Y$ with $j\in \{m+1,\ldots ,n\}$ is also a bar of $\coker f$. 
By definition, $\left\| Y \right\|_{p}$ is the $p$-norm of the following element of $\mathbb{R}^n$:
\[
(b_j-a_j)_{j\in \{1,\ldots ,n\}} = (((b_i-a'_i)+(a'_i-a_i))_{i\in\{1,\ldots ,m\}}, (b_{j}-a_{j})_{j\in\{m+1,\ldots ,n\}}).
\]
Then, by the triangular inequality of $p$-norms in $\mathbb{R}^n$, we have $\left\| Y \right\|_{p} \le  \left\| X \right\|_{p} +  \left\| \coker f \right\|_{p}$, which completes the proof when $C$ is the standard contour.

Let now $C$ be any regular action. By Proposition \ref{prop:shortexact_XC}, exactness of $0\to X\to Y\to Z\to 0$ implies exactness of $0\to T_C(X)\to T_C(Y)\to T_C(Z)\to 0$. Applying the previous part of the proof to the latter exact sequence yields $\left\| T_C(Y) \right\|_{p} \le \left\| T_C(X) \right\|_{p} + \left\| T_C(Z) \right\|_{p}$, which by Proposition \ref{prop:XCp} coincides with our claim.
\end{proof}

For the standard contour, the statement of Lemma \ref{lem:ns_ses2} is given in Remark 7.32 of \cite{skraba2020wasserstein}. However, to our knowledge, we provide the first proof of this inequality that does not assume the fact that the $p$-norm of persistence modules induces a pseudometric. Indeed in  \cite{skraba2020wasserstein} the fact that the algebraic Wasserstein distance satisfies the triangular inequality is used as an hypothesis.

We can now prove the main result of this subsection.

\begin{thm}
\label{thm:SpC_noisesystem}
For any $p\in [1,\infty]$ and any regular action $C$, $\ns^{p,C}$ is a noise system.
\end{thm}
\begin{proof} We show that $\ns^{p,C}$ satisfies all axioms of the definition of noise system (see Section \ref{subsec:noise_systems_distances}). Since the norm $\left\| \cdot \right\|_{p,C}$ of the zero module $0$ is zero, we have $0 \in \ns_{\varepsilon}^{p,C}$, for all $\varepsilon \in [0,\infty )$. 
By definition of $\ns^{p,C}$, it is clear that $\ns_{\tau}^{p,C} \subseteq \ns_{\varepsilon}^{p,C}$ whenever $\tau \le \varepsilon$.
Lemma \ref{lem:ns_ses1} and Lemma \ref{lem:ns_ses2} complete the proof, showing that $\ns^{p,C}$ satisfies both conditions on short exact sequences of persistence modules.
\end{proof}

\begin{rmk}
\label{rmk:closed_directsum}
For $p < \infty$, the noise system $\ns^{p,C}$ in not closed under direct sums (Section \ref{subsec:noise_systems_distances}), since $\left\| X\oplus Y \right\|_{p,C} = \left\| \left( \| X\|_{p,C}, \| Y\|_{p,C} \right) \right\|_{p}$ by equation (\ref{eq:concatenate_p}). 
\end{rmk}

\begin{rmk}
\label{rmk:hp_contours}
Let us briefly highlight the role of our hypotheses on contours, which are required to be regular actions in Theorem \ref{thm:SpC_noisesystem}. 
The regularity assumption ensures for instance that the associated lifetime function $\ell$ is well-defined, and that the functor $T_C$ is an endofunctor on $\Tame$. The assumption that $C(0,-)\colon [0,\infty )\to [0,\infty )$ is an increasing bijection is sufficient to prove many results of this subsection (see Remark \ref{rmk:TC_exact}), but we choose to assume the stronger condition of regularity on $C$ to facilitate a comparison with the results of \cite{chacholski2020metrics}, observing in addition that many examples of regular contours can be found, for example the contours of distance type (Section \ref{subsec:contours}) that are used in our experiments (see Section \ref{sec:stable_rank}).
The hypothesis that the considered contours are actions is motivated by the use of Proposition \ref{prop:XCp} in the proof of Lemma \ref{lem:ns_ses2}. If $C$ is not an action, the equalities of Lemma \ref{lem:add_lifespans} and Proposition \ref{prop:XCp} are replaced by the inequalities $\ell (a,c) \le \ell (a,b)+\ell (b,c)$, for any $a\le b \le c$, and $\| T_C(X)\|_p \le \| X\|_{p,C}$. A proof of Lemma \ref{lem:ns_ses2} removing the action hypothesis on $C$ eludes us.
\end{rmk}

\subsection{Contours and algebraic Wasserstein distances}
\label{subsec:alg_p_pC}
We now turn to considering the pseudometrics $d^{q}_{\ns^{p,C}}$ associated (as in Section \ref{subsec:pseudom_dp}) with the noise systems $\ns^{p,C}$ introduced in Section \ref{subsec:pnorms_C}, for fixed $p,q\in [1,\infty ]$ and a regular action $C$. 
We also refer to these pseudometrics as \textbf{algebraic Wasserstein distances}.
First, we show that the functor $T_C$ introduced in Section \ref{subsec:pnorms_C} allows us to switch between a pseudometric $d^{q}_{\ns^{p,C}}$ and the pseudometric $d^{q}_{\ns^{p}}$ associated with the standard contour. More precisely, we show that $T_C$ can be viewed as an isometry
\[
T_C \colon  (\Tame, d^{q}_{\ns^{p,C}}) \to (\Tame, d^{q}_{\ns^{p}}) .
\]

Let us recall that, if $C$ is a regular contour, the function $C(0,-):[0,\infty )\to [0,\infty )$ is an increasing bijection. Its inverse $\ell(0,-)\coloneqq  C(0,-)^{-1}$ is therefore an increasing bijection as well. 
Mimicking the definition of $T_C$ given in Section  \ref{subsec:pnorms_C}, we can define a functor $T_{\ell} \colon  \Tame \to \Tame$ given by precomposition by the increasing function $\ell (0,-)$. 
By Proposition \ref{prop:shortexact_XC}, the functor $T_C \colon  \Tame \to \Tame$ preserves kernels and cokernels, and $T_{\ell}$ has the same property by Remark \ref{rmk:TC_exact}. Furthermore,  since $C(0,-)$ and $\ell(0,-)$ are inverse to each other, the compositions $T_C T_\ell$ and $T_\ell T_C$ are the identity functor $1_{\Tame}$ on $\Tame$. 

To prove the following result, it is convenient to define the \textbf{$(p,q,C)$-cost} of a span $X \xleftarrow{f} Z \xrightarrow{g} Y$ of persistence modules as the element $c\in [0,\infty]$ defined by
\[
c \coloneqq  \left\| \left( \left\| \ker f \right\|_{p,C}, \left\| \coker f \right\|_{p,C} , \left\| \ker g \right\|_{p,C}, \left\| \coker g \right\|_{p,C} \right) \right\|_{q} .
\]

\begin{prop}
\label{prop:isometry_pq_pqC_alg}
Let $C$ be a regular action, and let $X,Y$ be persistence modules. Then
\[
d^{q}_{\ns^{p,C}} (X,Y) = d^{q}_{\ns^{p}} (T_C (X), T_C (Y)) .
\]
\end{prop}
\begin{proof}
Let $D$ denote the standard contour, and let us recall that the $(p,D)$-norm of a persistence module coincides with its $p$-norm (Section \ref{subsec:pnorms_C}). We describe a correspondence between spans having the same cost, calculated with respect to $(p,q,C)$ and $(p,q,D)$ respectively. 

Let $X \xleftarrow{f} Z \xrightarrow{g} Y$ be a span  and let $c$ be its $(p,q,C)$-cost.
Applying the functor $T_C$, we obtain the span $T_C(X) \xleftarrow{T_C(f)} T_C(Z) \xrightarrow{T_C(g)} T_C(Y)$, whose $(p,q,D)$-cost is 
\begin{align*}
c' & = \left\| \left( \left\| \ker T_C(f) \right\|_{p}, \left\| \coker T_C(f) \right\|_{p} , \left\| \ker T_C(g) \right\|_{p}, \left\| \coker T_C(g) \right\|_{p} \right) \right\|_{q} \\
&=  \left\| \left( \left\| T_C(\ker f) \right\|_{p}, \left\| T_C(\coker f) \right\|_{p} , \left\| T_C(\ker g) \right\|_{p}, \left\| T_C(\coker g) \right\|_{p} \right) \right\|_{q} \\
&= c ,
\end{align*}
where the second equality holds because the functor $T_C$ preserves kernels and cokernels, and the last equality holds by Proposition \ref{prop:XCp}. 

To prove the other direction of the correspondence, we start from a span $T_C(X) \xleftarrow{\varphi} T_C(Z) \xrightarrow{\psi} T_C(Y)$ whose $(p,q,D)$-cost is
\[
k \coloneqq  \left\| \left( \left\| \ker \varphi \right\|_{p}, \left\| \coker \varphi \right\|_{p} , \left\| \ker \psi \right\|_{p}, \left\| \coker \psi \right\|_{p} \right) \right\|_{q} ,
\]
and we exhibit a span between $X$ and $Y$ whose  $(p,q,C)$-cost equals $k$. Applying the functor $T_\ell$, we obtain the span $X \xleftarrow{T_\ell (\varphi)} Z \xrightarrow{T_\ell (\psi)} Y$. To determine the $(p,q,C)$-cost of this span we observe that
\[
\left\| \ker T_\ell (\varphi) \right\|_{p,C} = \left\| T_\ell (\ker \varphi ) \right\|_{p,C} = \left\| T_C T_\ell (\ker \varphi ) \right\|_{p} = \left\| \ker \varphi \right\|_{p} ,
\]
where the first equality holds because $T_\ell$ preserves kernels, the second equality is by Proposition \ref{prop:XCp}, and the third equality holds because $T_C T_\ell = 1_{\Tame}$. Since similar equalities hold for $\coker T_\ell (\varphi)$, $\ker T_\ell (\psi)$, and $\coker T_\ell (\psi)$, the $(p,q,C)$-cost of the span $X \xleftarrow{T_\ell (\varphi)} Z \xrightarrow{T_\ell (\psi)} Y$ equals $k$. 
\end{proof}

\begin{rmk} 
\label{rmk:defs_dp}
Some of the pseudometrics between persistence modules that have been studied by other authors fall within the framework we have presented in this subsection and in Section \ref{subsec:pseudom_dp}.
If $C$ is a regular contour, the pseudometric denoted by $d_C$ in \cite[Sect.\ 6]{chacholski2020metrics} coincide with our  pseudometrics of the type $d^1_{\ns^{\infty,C}}$. In particular, for the standard contour (Section \ref{subsec:contours}) the pseudometric $d^1_{\ns^{\infty}}$ coincides with the standard pseudometric already introduced in \cite{scolamiero2017multidimensional}.
As we already mentioned, the algebraic pseudometrics introduced in \cite[Sect.\ 7]{skraba2020wasserstein} are of the form $d^p_{\ns^p}$, thus coinciding with our pseudometrics with the choice $p=q$ and for the standard contour. 
In \cite{giunti2021amplitudes}, the authors propose a framework to study distances on abelian categories which is equivalent to noise systems on abelian categories. The authors of \cite{bubenik2022exact} also study distances on abelian categories, introducing the notion of exact weight, which is more general than noise systems as the first axiom on short exact sequences is relaxed. The so-called path metric associated with an exact weight is defined for zigzags of morphisms of arbitrary finite length, but for the particular case of path metrics on noise systems considering spans is sufficient. In this case, the path metric coincides with a pseudometric of the form $d^1_{\ns}$. In particular, the path metric $d_{\mu \circ \dim}$ between persistence modules studied in \cite[Sect.\ 4]{bubenik2022exact} coincides with $d^1_{\ns^1}$ in our notations, while the $p$-Wasserstein distances introduced by the authors are different from our pseudometrics $d^q_{\ns^{p,C}}$.
\end{rmk}

\subsection{Algebraic and combinatorial \texorpdfstring{$(p,C)$}{(p,C)}-Wasserstein distances}
\label{subsec:alg_and_comb_WpC}
In this subsection we consider Wasserstein distances between persistence diagrams. Here, we call these pseudometrics \emph{combinatorial} Wasserstein distances, to distinguish them from the \emph{algebraic} pseudometrics $d^q_{\ns^{p,C}}$ defined on the class of persistence modules. We introduce a new family
of combinatorial Wasserstein distances, parametrized by $p,q\in [1,\infty]$ and a regular action $C$, which generalize the Wasserstein distances commonly used in persistence theory. Finally, we prove isometry results involving the combinatorial Wasserstein distances and the algebraic Wasserstein distances $d^q_{\ns^{p,C}}$ introduced in Section \ref{subsec:pnorms_C}.

Let $U \coloneqq \{(a,b)\in [0,\infty )\times [0,\infty ] \mid a\leq b\}$ be a subset of the extended plane. 
A \textbf{persistence diagram} is a finite multiset $D=\{ x_i \}_{i\in S}$ of elements of $U$. Since $D$ is a multiset, it may happen that $x_i = x_k$ for some $i\ne k$. The \textbf{diagonal} $\Delta$ of $[0,\infty )$ is the set $\Delta \coloneqq  \{ (a,a) \mid a \in [0,\infty ) \} \subset U$.  For all $p\in [1,\infty]$, we denote by $d_p$ the metric on $U$ induced by the $p$-norm, defined by $d_p (x,y) \coloneqq  \left\| x-y \right\|_p$ for all $x,y\in U$, and we denote $d_p (x,\Delta)\coloneqq  \inf_{z\in \Delta} d_p (x,z)$. As is easy to show, if $x=(a,b)$, then $d_p (x,\Delta)= d_p (x,\overline{z})$ with $\overline{z}\coloneqq  (\frac{a+b}{2},\frac{a+b}{2})$.

Let $D=\{x_i\}_{i\in \{1,\ldots, m\}}$ and $D'=\{x'_j\}_{j\in \{1,\ldots, n\}}$ be persistence diagrams. For any $p,q\in [1,\infty]$, the $(p,q)$-\textbf{Wasserstein distance} between $D$ and $D'$ is defined by
\begin{multline*}
W^q_p (D,D') \coloneqq  \\
\inf_{\alpha} \left\| \left( \left\| (d_p (x_i,x'_{\alpha (i)}))_{i\in I} \right\|_q , \left\| (d_p (x_i,\Delta))_{i\in \{1,\ldots ,m\}\setminus I} \right\|_q , \left\| (d_p (\Delta, x'_j))_{j\in \{1,\ldots ,n\}\setminus \alpha (I)} \right\|_q  \right) \right\|_q ,
\end{multline*}
where the infimum is over all injective functions $\alpha \colon  I \to \{1,\ldots ,n\}$, with $I\subseteq \{1,\ldots ,m\}$.

\begin{rmk}
We note that in the literature, the letters $p$ and $q$ are sometimes interchanged with respect to our notation of the parameters of Wasserstein distances between persistence diagrams. This is the case for instance in  \cite[Def.\ 2.7]{skraba2020wasserstein}. Our choice of notation is motivated by symmetry with the definition of algebraic Wasserstein distances, where a norm $\|\cdot \|_q$ is used to ``aggregate'' costs expressed with respect to a norm $\|\cdot \|_p$.
\end{rmk}

Let $\mathcal{D}$ denote the set of all persistence diagrams.  
We define the class function $\Dgm \colon  \Tame \to \mathcal{D}$ sending any persistence module $X$ to the persistence diagram $\Dgm (X)$ such that $X \cong \bigoplus_{(a,b)\in \Dgm (X)} K(a,b)$, where we note that in the right-hand term each bar $K(a,b)$ appears the same number of times as the multiplicity of $(a,b)$ in the multiset $\Dgm (X)$. By virtue of the barcode decomposition theorem (Theorem \ref{thm:barcode_decomp}), the function $\Dgm \colon \Tame \to \mathcal{D}$ induces a bijection
between the set $\Tame/_{\sim}$ of isomorphism classes of persistence modules and $\mathcal{D}$. 

As proven in \cite{skraba2020wasserstein}, if $p=q$ then the algebraic distance $d^q_{\ns^p}$ between persistence modules coincides with the combinatorial distance $W^q_p$ between the associated persistence diagrams.

\begin{thm}[\cite{skraba2020wasserstein}]
\label{thm:comb_alg}
For any $p\in [1,\infty]$ and for any persistence modules $X$ and $Y$ we have
\[
d^{p}_{\ns^{p}}(X,Y) = W^{p}_{p} (\Dgm (X), \Dgm (Y)) .
\]
\end{thm}

It is worth observing that the equality of Theorem \ref{thm:comb_alg} does not hold when $p\ne q$. 
For example, we can consider the persistence modules 
\[
X=K(a_1,a_1+\ell_1)\oplus K(a_2, a_2+\ell_2)\oplus K(a_3, a_3+\ell_3)
\]
with $\ell_1, \ell_2, \ell_3$ positive real numbers, and $0$, the zero module. Then, assuming $q<\infty$, 
\[
d^q_{\ns^p} (X,0) = \left(  \left\| \left( \frac{\ell_1}{2},\frac{\ell_2}{2},\frac{\ell_3}{2} \right) \right\|_{p}^{q} +  \left\| \left( \frac{\ell_1}{2},\frac{\ell_2}{2},\frac{\ell_3}{2} \right) \right\|_{p}^{q}  \right)^{\frac{1}{q}}\]
(as we will prove in Lemma \ref{lem:dX0}), while
\[
W^q_p (\Dgm (X), \Dgm (0)) = \left( \left\| \left( \frac{\ell_1}{2},\frac{\ell_1}{2} \right) \right\|_{p}^{q} +  \left\| \left( \frac{\ell_2}{2},\frac{\ell_2}{2} \right) \right\|_{p}^{q}  +  \left\| \left( \frac{\ell_3}{2},\frac{\ell_3}{2} \right) \right\|_{p}^{q}\right)^{\frac{1}{q}} .
\]

Given a regular contour $C$, we now define a function $\tau_{C} \colon U\to U$ as follows: for $x=(a,b)\in U$, we set $\tau_C (x) = (\ell (0,a), \ell (0,b))$, where $\ell (0,-)$ is the lifetime function associated with $C$ (Section \ref{subsec:contours}). If $D$ is a persistence diagram, then by applying $\tau_C$ to each element of $D$ we obtain a persistence diagram that we denote by $\tau_C (D)$.  Hence, we have a function $\mathcal{D}\to \mathcal{D}$ which we denote again by $\tau_C$, with a slight abuse of notation. If $C$ is the standard contour, then $\tau_C$ is the identity function and in particular $\tau_C (D)=D$. Figure \ref{fig:transformed_pd} illustrates a persistence diagram transformed by applying $\tau_C$ for a contour $C$ of distance type.

Given a regular contour $C$, we define the \textbf{combinatorial $(p,C)$-Wasserstein distance} $W^p_{p,C}$ pulling back the pseudometric $W^p_p$ via $\tau_C \colon  \mathcal{D}\to \mathcal{D}$. Explicitly, for all persistence diagrams $D$ and $D'$, we define $W^p_{p,C}(D,D')\coloneqq  W^p_p (\tau_C (D), \tau_C (D'))$. If $C$ is a regular action, then as a consequence of Corollary \ref{coro:barcode_reparam} we have $\Dgm (T_C(X))=\tau_C (\Dgm (X))$, for every persistence module $X$. This implies, by virtue of Proposition \ref{prop:isometry_pq_pqC_alg} and Theorem \ref{thm:comb_alg}, that 
\[
d^p_{\ns^{p,C}} (X,Y) = W^p_{p,C} (\Dgm (X),\Dgm (Y)),
\]
for all persistence modules $X$ and $Y$.

To summarize, for any $p\in [1,\infty]$ and any regular action $C$, we have a commutative diagram of isometries
\begin{equation*}
\label{diag:isom_d_W}
\begin{tikzcd}
	{(\Tame, d^{p}_{\ns^{p,C}})} && ({\mathcal{D}, W^{p}_{p,C})} \\
	{(\Tame, d^{p}_{\ns^{p}})} && ({\mathcal{D}, W^{p}_{p})}
	\arrow["{T_C}"', from=1-1, to=2-1]
	\arrow["{\tau_C}", from=1-3, to=2-3]
	\arrow["{\Dgm}", from=1-1, to=1-3]
	\arrow["{\Dgm}", from=2-1, to=2-3]
\end{tikzcd}
\end{equation*}

\begin{figure}
     \centering
         \includegraphics[width=0.7\textwidth]{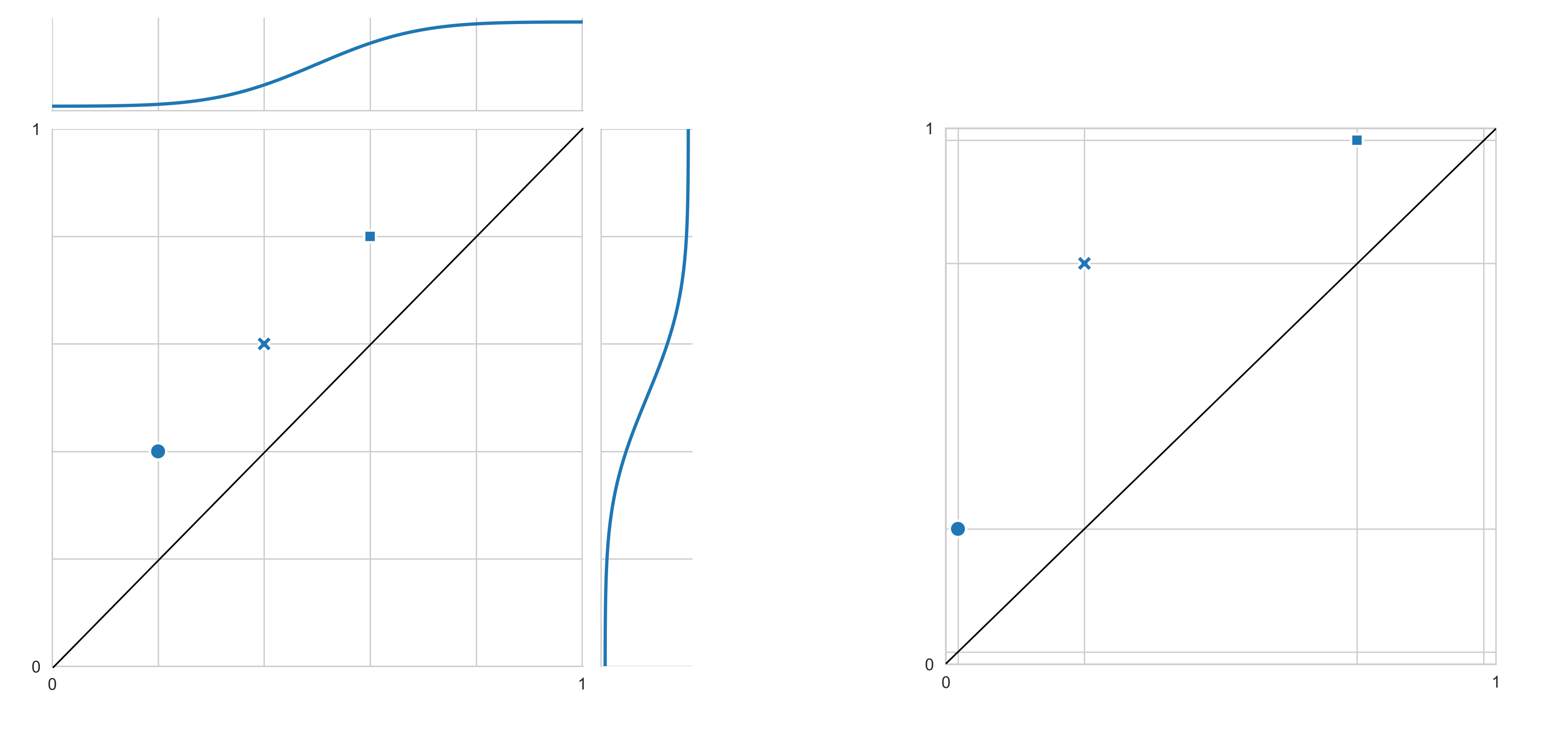}
    \caption{\textbf{Left:} A persistence diagram $D = \{ (0.2, 0.4), (0.4, 0.6), (0.6, 0.8)\}$. A contour $C$ of distance type parametrized by a Gaussian density ($\mu=0.5, \sigma=0.15$) is chosen and the corresponding function $f(x) = \ell(0, x)$ (i.e.\ the Gaussian cumulative distribution function) is shown above and to the right of the persistence diagram.
    \textbf{Right:} The transformed persistence diagram $\tau_C(D) = \{ (\ell(0, 0.2), \ell(0, 0.4)), (\ell(0, 0.4), \ell(0, 0.6)), (\ell(0, 0.6), \ell(0, 0.8))\}$. The regular grid from the left diagram has also been transformed to illustrate how $\tau_C$ stretches the plane.}
    \label{fig:transformed_pd}
\end{figure}

\subsection{Algebraic parametrized Wasserstein distances}
\label{subsec:alg_pC_distance}
The equivalence between algebraic and combinatorial Wasserstein distances for the case $p=q$, described in Section \ref{subsec:alg_and_comb_WpC} or in \cite{skraba2020wasserstein} for the standard contour, implies that in general Wasserstein distances cannot be expressed by an explicit (i.e., not involving an optimization problem) formula depending on the barcode decompositions of the persistence modules we are comparing. However for some special classes of persistence modules this is the case.
The focus of this section is to present such formulas for the exact computation of algebraic Wasserstein distances.
To avoid distinguishing the cases $q<\infty$ and $q=\infty$ in stating the results of this subsection, for $q=\infty$ we set by convention $\frac{1}{q}=0$ and $2^{\frac{1-q}{q}}=2^{-1}$.

\begin{lem}
\label{lem:dX0} 
For all persistence modules $X$ and all $p,q\in [1,\infty]$ we have
\[
d^q_{\ns^{p}}(X,0)= 2^\frac{1-q}{q}\left\| X \right\|_{p} .
\]
\end{lem}
\begin{proof}
Let $X=\bigoplus_{i=1}^k K(a_i,b_i)$ be a barcode decomposition of $X$, consider a persistence module of the form  $Z=\bigoplus_{i=1}^k K(\frac{a_i+b_i}{2},b_i)$ and a bar-to-bar morphism $f=\bigoplus_{i=1}^k f_i\colon  Z\to X$, with each $f_i \colon K(\frac{a_i+b_i}{2},b_i)\rightarrow K(a_i,b_i)$ a monomorphism between bars. Denote by $g=\bigoplus_{i=1}^k g_i\colon  Z\to 0$ the zero map. The existence of the span $X \xhookleftarrow{f} Z \xrightarrow{g} 0$ implies that $X$ and $0$ are $2^\frac{1-q}{q}\left\| X \right\|_p$ close in $q$-norm (Definition \ref{def:dpS}), proving that   
$d^q_{\ns^{p}}(X,0)\leq 2^\frac{1-q}{q}\left\| X \right\|_p$. Indeed 
$\ker f  = \coker g = 0$ and $\left\| \coker f \right\|_p = \left\| \ker g \right\|_p = \left\| (\frac{b_i-a_i}{2})_{i\in \{1,\ldots ,k\}} \right\|_{p}$. 
The bound is obtained by computing  $\left\| \left( \left\| (\frac{b_i-a_i}{2})_{i\in \{1,\ldots ,k\}} \right\|_{p}, \left\| (\frac{b_i-a_i}{2})_{i\in \{1,\ldots ,k\}} \right\|_{p} \right) \right\|_{q}=2^\frac{1-q}{q} \left\| X \right\|_{p}$.

To prove the converse inequality, let us show that if $d^q_{\ns^{p}}(X,0) < \varepsilon$ then   $2^\frac{1-q}{q}\left\| X \right\|_p < \varepsilon$. If $d^q_{\ns^{p}}(X,0)< \varepsilon$, then there exists a $(\varepsilon_1,\varepsilon_2,
\varepsilon_3, 0)$-span $X \xleftarrow{\varphi} Z \twoheadrightarrow 0$ for some $\varepsilon_1, \varepsilon_2, \varepsilon_3$ in $[0, \infty)$ such that $\left\| (\varepsilon_1,\varepsilon_2,\varepsilon_3) \right\|_q < \varepsilon$. Note that
$X \hookleftarrow \image \varphi \twoheadrightarrow 0$
is then a $(0,\varepsilon_2,\varepsilon_3, 0)$-span. Consider the short exact sequence $\image \varphi \hookrightarrow X \twoheadrightarrow \coker \varphi$. Since $\coker \varphi \in \ns^p_{\varepsilon_2}$ and $\image \varphi \in \ns^p_{\varepsilon_3}$, 
by the third axiom of noise systems we get $X\in \ns^p_{\varepsilon_2+\varepsilon_3}$, and so we get $\left\|X \right\|_p\leq \varepsilon_2+\varepsilon_3$ by definition of $\ns^p$. Furthermore, by inequalities (\ref{eq:monotonicity_p}) between $p$-norms on $\mathbb{R}^2$,  $\varepsilon_2+\varepsilon_3=\left\|(\varepsilon_2,\varepsilon_3)\right\|_1 \leq 2^{1-\frac{1}{q}}\left\|(\varepsilon_2,\varepsilon_3)\right\|_q < 2^{1-\frac{1}{q}} \varepsilon$. Therefore we have $\left\|X \right\|_p < 2^{1-\frac{1}{q}} \varepsilon$ or equivalently $2^\frac{1-q}{q}\left\|X \right\|_p\ < \varepsilon$. We conclude that $d^q_{\ns^{p}}(X,0)\geq 2^\frac{1-q}{q}\left\| X \right\|_p$, and therefore $d^q_{\ns^{p}}(X,0) = 2^\frac{1-q}{q}\left\| X \right\|_p$.  
\end{proof}

\begin{rmk}
\label{rmk:dist_0_p=q}
The formula  $d^q_{\ns^{p}}(X,0) = 2^\frac{1-q}{q}\left\| X \right\|_p$ of Lemma \ref{lem:dX0}
 was already shown for the case $p=q$ in \cite{skraba2020wasserstein} by using the correspondence between combinatorial and algebraic Wasserstein distances.
\end{rmk}

The proof of Lemma \ref{lem:dX0} can be easily extended to the case of a regular action $C$. In this case, we have 
\begin{equation}
\label{eq:pC_dist0}
d^{q}_{\ns^{p,C}} (X,0) = d^{q}_{\ns^{p}} (T_C (X), 0) = 2^\frac{1-q}{q}\left\| T_C (X) \right\|_{p} = 2^\frac{1-q}{q}\left\| X \right\|_{p,C} , 
\end{equation}
where the first equality holds by Proposition \ref{prop:isometry_pq_pqC_alg}, the second by Lemma \ref{lem:dX0} and the third by Proposition \ref{prop:XCp}.
Similar arguments can be applied to all the results of this subsection. For exposition purposes we consider the case of the standard contour throughout the section and collect generalizations of the main results at the end of the subsection in Proposition \ref{prop:closest_rank_C}.

\begin{prop}
\label{prop:dXYV}
Let $X,Y,V$ be persistence modules. Then, for every $p,q\in [1,\infty]$,
\[
d^q_{\ns^{p}} (X\oplus V, Y\oplus V) \le d^q_{\ns^{p}} (X,Y).
\]
\end{prop}
\begin{proof}
It suffices to observe that for any span $X \xleftarrow{f} Z \xrightarrow{g} Y$, the span $X\oplus V \xleftarrow{f\oplus 1} Z\oplus V \xrightarrow{g\oplus 1} Y\oplus V$ has the same cost. 
\end{proof}

\begin{rmk}
\label{rmrk:changing_matching}
Note that by considering $Y=0$, Proposition \ref{prop:dXYV}
gives \[
d^q_{\ns^{p}} (X\oplus V, V) \le d^q_{\ns^{p}} (X,0)=2^\frac{1-q}{q}\left\| X \right\|_{p}
\]
The converse inequality
$d^q_{\ns^{p}} (X\oplus V, V) \geq d^q_{\ns^{p}} (X,0)=2^\frac{1-q}{q}\left\| X \right\|_{p}$ does not hold in general, as illustrated in the following example. 
Consider $p=q=2$, $X=K(0,6)$ and $V= K(1,5)\oplus K(2,4)$. 
By Lemma \ref{lem:dX0} we have that $d^q_{\ns^{p}} (X,0)=\frac{1}{\sqrt 2} \cdot 6 = \sqrt{18}$. However, $X\oplus V$ and $Y\oplus V$ are $\sqrt 6$-close via the following  $(0, \sqrt{3},\sqrt{3},0)$-span
$$ K(0,6)\oplus K(1,5) \oplus K(2,4) \xleftarrow{f_1\oplus f_2 \oplus f_3} K(1,6)\oplus K(2,5) \oplus K(3,4) \xrightarrow{g_1\oplus g_2 \oplus g_3} K(1,5)\oplus K(2,4) \oplus 0$$
implying that $d^q_{\ns^{p}} (X\oplus V, Y\oplus V)\leq \sqrt{6} < \sqrt{18}=d^q_{\ns^{p}} (X,Y).$
This example is based on the fact that given a span $X \xleftarrow{f} Z \xrightarrow{g} Y$ realizing the distance between $X$ and $Y$, the span  $X\oplus V \xleftarrow{f\oplus 1} Z\oplus V \xrightarrow{g\oplus 1} Y\oplus V$
not always is the one achieving the distance between $X\oplus V$ and $Y\oplus V$. 
\end{rmk}

Let $\{K(a_i,b_i)\}_{i\in \{1,\ldots , k\}}$ be a sequence of bars ordered non-decreasingly by length, that is, $b_1-a_1 \le b_2-a_2 \le \cdots \le b_k-a_k$.
For $j\in \{1,\ldots , k\}$, consider $Z\coloneqq \bigoplus_{i=1}^j K(a_i,b_i)$ and $Y\coloneqq \bigoplus_{i=j+1}^k K(a_i,b_i)$. The remainder of this section is devoted to proving that, in this case, $d^q_{\ns^{p}}(Y\oplus Z,Y) = d^q_{\ns^{p}}(Z,0) = 2^\frac{1-q}{q}\left\| Z\right\|_p.$
In Section \ref{sec:stable_rank}, this result will be used for the computation of the stable rank of a persistence module with respect to $d^q_{\ns^{p}}$.

\begin{prop}
\label{prop:mono-epi-rank}
Let $\ns$ be a noise system.
For any $(\varepsilon_1,\varepsilon_2,\varepsilon_3,\varepsilon_4)$-span $X \xleftarrow{} Z \xrightarrow{} Y$ of persistence modules 
there is a mono-epi $(0,\varepsilon_2,\varepsilon_3,0)$-span $X \xhookleftarrow{} \image f \twoheadrightarrow P$  such that $\rank (P) \le \rank (Y)$. 
\end{prop}
\begin{proof}
By Theorem \ref{thm:mono-bar} and Remark \ref{rmk:ker_coker_bartobar}, if $U\hookrightarrow V$ is a monomorphism between persistence modules, then $\rank (U) \le \rank (V)$, and similarly if  $V\twoheadrightarrow U$ is an epimorphism, then $\rank (U) \le \rank (V)$. Let $X \xleftarrow{f} Z \xrightarrow{g} Y$ be a $(\varepsilon_1,\varepsilon_2,\varepsilon_3,\varepsilon_4)$-span of persistence modules, and consider the following diagram in $\Tame$, where the square is a push-out:
\[
\begin{tikzcd}[column sep=small, row sep=small]
& & Z \arrow[dl, "f"', two heads] \arrow[dr, "g", two heads] & & \\
& \image f \arrow[dl, "j"', hook'] \arrow[dr, "g'", two heads] & & \image g \arrow[dl, "f'", two heads] \arrow[dr, "i", hook] & \\
X & & P & & Y 
\end{tikzcd}
\]
Since $f'$ is an epimorphism and $i$ is a monomorphism, $\rank (P)\le \rank (\image g) \le \rank (Y)$. We consider the span $X \xhookleftarrow{j} \image f \overset{g'}\twoheadrightarrow P$. Clearly, the kernel of the corestriction $g\colon Z\twoheadrightarrow \image g$ still belongs to $\ns_{\varepsilon_3}$, and its cokernel is zero. Then, by Proposition $8.1$ in \cite{scolamiero2017multidimensional}, $\ker g' \in \ns_{\varepsilon_3}$ and $\coker g'=0$. The kernel of $j$ is $0$, while its cokernel belongs to $\ns_{\varepsilon_2}$, as it coincides with the cokernel of $f \colon Z\to X$.
\end{proof}

\begin{lem}
\label{lem:cost_XY}
Let $p,q\in [1,\infty]$, and let $[a_i,b_i]$ be nonempty intervals in $[0,\infty )$, for $i\in \{1,\ldots ,j\}$. The function $\gamma \colon  \prod_{i=1}^j [a_i,b_i] \to [0,\infty )$ defined by
\[
\gamma (x_1, \ldots , x_j) \coloneqq  \left\| \left( \left\| (x_1 - a_1, \ldots ,x_j - a_j)\right\|_{p}, \left\| (b_1 - x_1, \ldots , b_j - x_j) \right\|_{p}  \right) \right\|_{q }
\]
has a global minimum at $(\frac{a_1 +b_1}{2},\ldots ,\frac{a_j +b_j}{2})$. 
\end{lem}
\begin{proof}
The function $\gamma$ is continuous with a compact domain, so it admits a global minimum by the extreme value theorem. Moreover, it is convex because norms are convex functions.

Write $a = (a_1, \ldots, a_j)$, $b = (b_1, \ldots, b_j)$ and $x = (x_1, \ldots, x_j)$ in $\mathbb{R}^j$. Since $\gamma (x)=\gamma (a+b-x)$ for every $x$, the function $\gamma$ is invariant under point reflection through $\frac{a+b}{2}$. By convexity, we conclude that $\frac{a+b}{2}$ is a global minimum of $\gamma$.
\end{proof}

\begin{prop}
\label{prop:closest_rank}
Let $X=\bigoplus_{i=1}^k K(a_i,b_i)$, with the bars ordered non-decreasingly by length. Let $j\in \{1,\ldots ,k\}$, and let $p,q\in [1,\infty]$. Then, any persistence module $Y$ with $\rank (Y)\le \rank (X) - j$ is such that
\begin{equation*}
\label{eq:closest_rank>=}
d^q_{\ns^{p}}(X,Y) \ge 2^{\frac{1-q}{q}} \textstyle{\left\| \bigoplus_{i=1}^j K(a_i,b_i) \right\|_{p} }.
\end{equation*}
\end{prop}
\begin{proof}
We prove the claim by contradiction. Suppose that there exists a persistence module $Y$ such that $\rank (Y)\le \rank (X) - j$ and 
\begin{equation*}
\label{eq:closest_rank<}
d^q_{\ns^{p}}(X,Y) < 2^{\frac{1-q}{q}} \textstyle{\left\| \bigoplus_{i=1}^j K(a_i,b_i) \right\|_{p} }.
\end{equation*}
By definition, there exists a span $X \xleftarrow{f} Z \xrightarrow{g} Y$ such that
\begin{equation}
\label{eq:span_cost_closest_rank}
\left\| \left( \left\| \ker f \right\|_{p}, \left\| \coker f \right\|_{p} , \left\| \ker g \right\|_{p},  \left\| \coker g \right\|_{p}  \right) \right\|_{q} < 2^{\frac{1-q}{q}} \textstyle{\left\| \bigoplus_{i=1}^j K(a_i,b_i) \right\|_{p} } . 
\end{equation}
By Proposition \ref{prop:mono-epi-rank} we can assume (possibly after replacing $Y$ with a persistence module of smaller or equal rank) that the span above is mono-epi, that is, of the form $X \xhookleftarrow{f} Z \overset{g}\twoheadrightarrow Y$. By Theorems \ref{thm:mono-bar} and \ref{thm:epi-bar}, we can moreover assume that $f$ and $g$ are bar-to-bar morphisms. 

Thus, we can consider a barcode decomposition $Z=\bigoplus_{i=1}^k Z_i$, with some of the $Z_i$ possibly zero, and a barcode decomposition $Y=\bigoplus_{i=1}^k Y_i$, with at least $j$ of the $Y_i$ equal to zero by assumption, together with morphisms between bars $K(a_i,b_i) \xhookleftarrow{f_i} Z_i \overset{g_i}\twoheadrightarrow Y_i$ such that $f=\bigoplus_{i=1}^k f_i$ and $g=\bigoplus_{i=1}^k g_i$. 
Let $I\subseteq \{1,\ldots ,k\}$, with $\lvert I\rvert \ge j$, be the subset of the indices $i$ such that $Y_i=0$. For every
$i\in I$, we have $K(a_i,b_i) \xhookleftarrow{f_i} Z_i \overset{g_i}\twoheadrightarrow 0$, with $Z_i =K(x_i,b_i)$ for some $a_i \le x_i \le b_i$, where $K(b_i,b_i)$ denotes the zero module. 
Since $\ker f = \bigoplus_{i=1}^k \ker f_i$ and $\coker f =  \bigoplus_{i=1}^k \coker f_i$, by Remark \ref{rmk:ker_coker_bartobar} we observe that $\bigoplus_{i\in I} K(a_i,x_i)$ is a direct summand of $\coker f$, and similarly that $\bigoplus_{i\in I} K(x_i,b_i)$ is a direct summand of $\ker g$, which gives
\begin{align*}
\left\| \coker f \right\|_{p} &\ge \left\| \oplus_{i\in I} K(a_i,x_i) \right\|_{p} = \left\| (x_i-a_i)_{i\in I} \right\|_{p} ,  \\
\left\| \ker g \right\|_{p} &\ge \left\| \oplus_{i\in I} K(x_i,b_i) \right\|_{p} = \left\| (b_i-x_i)_{i\in I} \right\|_{p} .
\end{align*}
If $b_i <\infty$ for all $i\in I$, it is easy to show using Lemma \ref{lem:cost_XY} that the cost of the span is
\[
\left\| \left( \left\| \coker f \right\|_{p}, \left\| \ker g \right\|_{p}  \right) \right\|_{q} \ge  2^{\frac{1-q}{q}} \textstyle{\left\| (b_i-a_i)_{i\in I} \right\|_{p} } = 2^{\frac{1-q}{q}} \textstyle{\left\| \bigoplus_{i\in I} K(a_i,b_i) \right\|_{p} },
\]
and the same inequality clearly holds if $b_i=\infty$ for some $i\in I$. However, since $\lvert I\rvert \ge j$, the right-hand side of the inequality cannot be smaller than
\[
2^{\frac{1-q}{q}} \textstyle{\left\| (b_i-a_i)_{i\in \{1,\ldots ,j\}} \right\|_{p} } = 2^{\frac{1-q}{q}} \textstyle{\left\| \bigoplus_{i=1}^j K(a_i,b_i) \right\|_{p} } ,
\]
and this contradicts (\ref{eq:span_cost_closest_rank}).
\end{proof}

\begin{prop} 
\label{prop:d_X_deleting_bars}
Let $X=\bigoplus_{i=1}^k K(a_i,b_i)$, with the bars ordered non-decreasingly by length. Let $j\in \{1,\ldots ,k\}$, and let $Y=\bigoplus_{i=j+1}^k K(a_i,b_i)$ (with $Y=0$ when $j=k$). Then, for all $p,q\in [1,\infty]$,
\begin{equation}
\label{eq:j_short_bars}
d^q_{\ns^{p}}(X,Y) =2^{\frac{1-q}{q}} \textstyle{\left\| \bigoplus_{i=1}^j K(a_i,b_i) \right\|_{p} }.
\end{equation}
\end{prop}
\begin{proof}
Since $\rank (Y)= \rank (X)-j$, Proposition \ref{prop:closest_rank} gives us the inequality  
\[
d^q_{\ns^{p}}(X,Y) \ge 2^{\frac{1-q}{q}} \textstyle{\left\| \bigoplus_{i=1}^j K(a_i,b_i) \right\|_{p} } . 
\]
The other inequality, as already noticed in Remark \ref{rmrk:changing_matching}, follows from Proposition \ref{prop:dXYV} and Lemma \ref{lem:dX0} showing that $
d^q_{\ns^{p}} ( Z \oplus Y, Y) \le d^q_{\ns^{p}} (Z,0)=2^\frac{1-q}{q}\left\|  Z \right\|_{p}
$  with $Z=\bigoplus_{i=1}^j K(a_i,b_i)$.
\end{proof}

In the final part of this subsection we generalize some results from the case of the standard contour to the case of any regular action $C$.

\begin{defn}
\label{def:bars_ordered}
Let $C$ be a regular contour, and let $X=\bigoplus_{i=1}^k K(a_i,b_i)$. We say that (the barcode decomposition of) $X$ has \textbf{bars ordered non-decreasingly by lifetime} if $\ell (a_1,b_1)\le \ell (a_2,b_2)\le \cdots \le \ell (a_k,b_k)$, where
$\ell$ denotes the lifetime function associated with $C$ (see Section \ref{subsec:contours}).
\end{defn}

\begin{prop}
\label{prop:closest_rank_C}
Let $C$ be a regular action, and let $p,q\in [1,\infty]$. Let $X=\bigoplus_{i=1}^k K(a_i,b_i)$, with bars ordered non-decreasingly by lifetime, and let $j\in \{1,\ldots ,k\}$. Then, for all persistence modules $Y$,
\begin{enumerate}
\item if $\rank (Y)\le \rank (X) - j$, then
\begin{equation*}
\label{eq:closest_rank_C>=}
d^q_{\ns^{p,C}}(X,Y) \ge 2^{\frac{1-q}{q}} \textstyle{\left\| \bigoplus_{i=1}^j K(a_i,b_i) \right\|_{p,C} };
\end{equation*}
\item if $Y=\bigoplus_{i=j+1}^k K(a_i,b_i)$ (with the convention $Y=0$ when $j=k$), then
\[
d^q_{\ns^{p,C}}(X,Y) =2^{\frac{1-q}{q}} \textstyle{\left\| \bigoplus_{i=1}^j K(a_i,b_i) \right\|_{p,C} }.
\]
\end{enumerate}
\end{prop}
\begin{proof}
The first statement follows from
\begin{align*}
d^q_{\ns^{p,C}}(X,Y) &= d^q_{\ns^{p}}(T_C(X),T_C(Y)) \\ 
&\ge 2^{\frac{1-q}{q}} \textstyle{\left\| \bigoplus_{i=1}^j T_C (K(a_i,b_i))  \right\|_{p} } \\
&= 2^{\frac{1-q}{q}} \textstyle{\left\| \bigoplus_{i=1}^j K(a_i,b_i) \right\|_{p,C} }
\end{align*}
where we are using in sequence Proposition \ref{prop:isometry_pq_pqC_alg}, Proposition \ref{prop:closest_rank} (observing that the length of a bar $T_C(K(a,b))$ coincides with the lifetime $\ell(a,b)$ of $K(a,b)$, see Proposition \ref{prop:bar_reparam}), and Proposition \ref{prop:XCp}.
The second statement is proven similarly, using Proposition \ref{prop:d_X_deleting_bars}.
\end{proof}


\section{Wasserstein stable ranks: computations and stability}
\label{sec:stable_rank}
In Section \ref{sec:Wass_noise}
it was shown that the Wasserstein distances $d^q_{\ns^{p,C}}$ are pseudometrics on $\Tame$. They can therefore be used in the framework of hierarchical stabilization (see Section \ref{subsec:stabilization}) to produce stable invariants of persistence modules. The focus of this section is on one type of such invariants, the \textbf{Wasserstein stable ranks}, which are the hierarchical stabilization of the rank function with respect to Wasserstein distances $d^q_{\ns^{p,C}}.$ Denoting $d^q_{\ns^{p,C}}$ by $d$, the stability result for stable ranks (Proposition \ref{prop:stability_sr}) states that
for every pair of persistence modules $X$ and $Y$
$$
d(X,Y) \geq d_{\bowtie}(\widehat{\rank}_d(X),\widehat{\rank}_d(Y)).
$$

In the case where $p=q$ and $C$ is the standard contour, combining the above inequality with the stability results of
\cite{skraba2020wasserstein} gives several stability results of Wasserstein stable ranks with respect to perturbation of the original data. In particular, \cite[Theorem $4.8$]{skraba2020wasserstein} expresses stability with respect to sublevel set filtrations of monotone functions on cellular complexes, \cite[Theorem $5.1$]{skraba2020wasserstein} expresses stability with respect to the construction of cubical complexes from grey scale images, and \cite[Theorem $5.9$]{skraba2020wasserstein}, expresses stability with respect to Wasserstein distance between point clouds when using the Vietoris-Rips construction. 

In order to use the Wasserstein stable ranks in applications, it is important to be able to efficiently compute them as well as distances between them. In this section we use computations of Wasserstein distances from Section \ref{sec:Wass_noise} to derive a formula for the Wasserstein stable rank and propose a convenient formulation of the interleaving distance between stable ranks.

Having defined a rich family of Wasserstein distances $d^q_{\ns^{p,C}}$, it is natural to ask whether we can in a supervised learning context search for an optimal distance for a problem at hand.
Choosing a suitable parametrization of a contour and leveraging the simple expression of the interleaving distance between Wasserstein stable ranks, in Section \ref{sect:metriclearning} we set up a simple metric learning problem with the aim of observing the interaction between the parameter $p$ and the parameters related to the contour $C$ within the learning.
Preliminary results on the optimization of only a contour in a metric learning framework are presented in \cite{gavfert2018poster}.

\subsection{Computation of the stable rank with Wasserstein distances}
\label{sec:comp_sr_contour}
The results of this subsection provide explicit formulas to compute the stable rank with respect to the Wasserstein distances $d^q_{\ns^{p,C}}$ introduced in Section \ref{sec:Wass_noise}. 
We begin 
by considering the case $p<\infty$. As in the previous section, if $q=\infty$ we set by convention $\frac{1}{q}=0$ and $2^{\frac{1-q}{q}}=2^{-1}$.

\begin{prop}
\label{prop:computing_sr}
Let $p\in [1,\infty )$ and $q\in [1,\infty]$, let $C$ be a regular action, and let $d$ denote the pseudometric $d^q_{\ns^{p,C}}$. 
Let  $X = \bigoplus_{i=1}^k K(a_i,b_i)$, with bars ordered non-decreasingly by lifetime (Definition \ref{def:bars_ordered}), and let $n \coloneqq \lvert \{i\in \{1,\ldots ,k\}\mid b_i <\infty\}\rvert$ denote the number of finite bars of $X$.
Then, there exist real numbers $0=t_0 < t_1 < t_2 < \cdots < t_n$ such that the stable rank function $\widehat{\rank}_{d}(X) \colon  [0,\infty) \to [0,\infty)$ is constant on the intervals $[t_0,t_1)$, $[t_1,t_2)$,\ldots , $[t_{n-1},t_n)$, $[t_n,\infty)$, and 
\begin{equation*}
\label{eq:sr_ti} 
\widehat{\rank}_{d}(X)(t_j) = \rank (X) - j,
\end{equation*}
for every $j\in \{0,1,\ldots ,n\}$. Furthermore,
\[
t_j=2^{\frac{1-q}{q}} \textstyle{\left\| \bigoplus_{i=1}^j K(a_i,b_i) \right\|_{p,C}}= 2^{\frac{1-q}{q}}\left\| ( \ell(a_1,b_1),\ldots ,\ell(a_j,b_j) )\right\|_p 
\]
for every $j\in \{1,\ldots ,n\}$, where $\ell$ is the lifetime function associated with $C$.
\end{prop}
\begin{proof}
For every $j\in \{1,\ldots ,k\}$, by Proposition \ref{prop:closest_rank_C} 
$Y_j\coloneqq \bigoplus_{i=j+1}^k K(a_i,b_i)$ is the closest persistence module to $X$ (in the pseudometric $d^q_{\ns^{p,C}}$) such that $\rank (Y_j) = \rank (X) - j$. 
We have
\[
d^q_{\ns^{p,C}}(X,Y_j) =2^{\frac{1-q}{q}} \textstyle{\left\| \bigoplus_{i=1}^j K(a_i,b_i) \right\|_{p,C} } \eqqcolon t_j,
\]
with $t_j <\infty$ if, and only if, $j\in \{1,\ldots ,n\}$.
Lastly, we observe that $0=t_0 < t_1 < t_2 < \cdots < t_n$ as a consequence of the assumption $p<\infty$.
\end{proof}

In particular, when $p<\infty$, the value of the piecewise constant function $\widehat{\rank}_{d}(X)$ can only decrease by 1 at every discontinuity point $t_j$. For $p=\infty$, the stable rank has a slightly different behavior. Even though we can define the sequence of real numbers $(t_j)_j$ as in Proposition \ref{prop:computing_sr}, we only have $0=t_0 \le t_1 \le t_2 \le \cdots \le t_n$ instead of strict inequalities. Letting $s_m$ denote the $m^{\text{th}}$ smallest value in $\{ t_j\}_j$ we obtain a sequence $0=s_0 < s_1 < s_2 < \cdots < s_{n'}$ such that the stable rank with respect to the pseudometric $d\coloneqq d^q_{\ns^{\infty,C}}$ is constant on the intervals $[s_0,s_1)$,\ldots , $[s_{n'},\infty)$, taking the values 
\[
\widehat{\rank}_{d}(X)(s_m) = \rank (X) - \max \{j \mid t_j=s_m \} .
\]
An explicit formula for the stable rank in the case $p=\infty$ and $q=1$ was first given in \cite{chacholski2020metrics}.

\begin{rmk}
We observe that for a persistence module $X$ of rank $k$, once the $k$ bars in the barcode decomposition of $X$ have been ordered non-decreasingly by lifetime,  the complexity of computing the discontinuity points of the 
the Wasserstein stable rank using Proposition \ref{prop:computing_sr} is linear in $k$. Therefore the computational complexity of the Wasserstein stable rank is $O(k \log k)$, determined by the complexity of the sorting algorithm to order the bars non-decreasingly by lifetime.
\end{rmk}

\subsection{Interleaving distance between stable ranks}
\label{sect:interleaving formula}
The aim of this subsection is to propose a convenient expression for the interleaving distance (Section \ref{subsec:stabilization}) between two non-increasing piecewise constant functions. We assume functions to take only finitely many  values, that is the case of stable ranks which will be the object of our study. 
Let $f,g\colon  [0,\infty) \to [0,\infty)$ be non-increasing piecewise constant functions. 
If $\lim_{t\to \infty} f(t) \ne \lim_{t\to \infty}g(t)$, then $d_{\bowtie}(f,g)=\infty$. For the computation of the interleaving distance we can therefore assume that the functions $f$ and $g$ have the same limit value and denote it by $L$. 
Given a non-increasing piecewise constant function
$f\colon  [0,\infty) \to [0,\infty)$ with limit value $L$,  we define the non-increasing piecewise constant function 
$f^{-1}\colon   [L,\infty)  \to [0,\infty)$
with values $f^{-1}(y) \coloneqq  \inf \{ t \mid f(t)\le y\}$. If in addition the function $f$ is right-continuous, 
then $f^{-1}(y) = \min \{ t  \mid f(t)\le y\}$. 
We observe that for every right-continuous non-increasing piecewise constant function $f$ we have $f^{-1}(f(t))\le t$ for all $t$, and equality holds if $t$ is a discontinuity point of $f$. Moreover, $f(f^{-1}(y))\le y$ for all $y\ge L$, and equality holds if $y\in \image f$. Our focus in this subsection will be on the discontinuity points $\{ t_i\}$ of $f$ and on the values in $\image f$, rather than on the full domain and codomain of $f$, thus justifying our use of the notation $f^{-1}$.  

\begin{prop}
\label{prop:interleavingPCF}
Consider two right-continuous non-increasing piecewise constant functions $f,g\colon  [0,\infty) \to [0, \infty)$ having the same limit value $L$. Using the notation introduced above, we have:
\[
d_{\bowtie} (f,g) = \| f^{-1} - g^{-1}\|_{\infty}.
\]
\end{prop}
\begin{proof}
Let us define the following subset of $[0,\infty)$,
\[
A(f,g)\coloneqq \{ \varepsilon \in [0,\infty ) \mid  f(t)\ge g(t+\varepsilon) \text{ and } g(t)\ge f(t+\varepsilon), \text{ for all } t\in [0,\infty ) \}.
\]
Remember that, by definition, $d_{\bowtie}(f,g) = \inf A(f,g)$. 

We first prove that $d_{\bowtie} (f,g) \ge \| f^{-1} - g^{-1}\|_{\infty}$. 
Let $\varepsilon \in A(f,g)$. Then, for all $y \geq L$, we have $y\ge f(f^{-1}(y)) \ge g(f^{-1}(y) +\varepsilon)$. Composing by the non-increasing function $g^{-1}$ and recalling that $g^{-1} (g(t))\le t$ for all $t$, we obtain $f^{-1}(y) + \varepsilon \ge g^{-1}(y)$. We have thus shown that $g^{-1}(y) - f^{-1}(y) \le \varepsilon$, for all $y\ge L$ and $\varepsilon \in A(f,g)$, which implies $g^{-1}(y) - f^{-1}(y) \le d_{\bowtie}(f,g)$, for all $y\ge L$. By symmetry in the roles of $f$ and $g$, we conclude that $\lvert g^{-1}(y) - f^{-1}(y)\rvert \le d_{\bowtie}(f,g)$, for all $y\ge L$.

We now prove that $d_{\bowtie} (f,g) \le \| f^{-1} - g^{-1}\|_{\infty}$ by showing that $\varepsilon \coloneqq  \| f^{-1} - g^{-1}\|_{\infty}$ is in $A(f,g)$.
By the definition of $\varepsilon$, $g^{-1}(y) \leq f^{-1}(y) + \varepsilon$ for every $y\in [L, \infty)$. Moreover if $y=f(t)$ for $t\in [0,\infty)$, as discussed above, we have $f^{-1}(y)=f^{-1}(f(t)) \leq t$, which together with the fact that $g$ is a non-increasing function proves the following inequalities:
\[
f(t) = y \geq g(g^{-1}(y)) \geq g(f^{-1}(y) + \varepsilon) \geq g(t + \varepsilon).
\]
By symmetry, we also get $g(t) \geq f(t + \varepsilon)$, and we conclude that $\varepsilon \in A(f,g)$.
\end{proof}

Let $X=\bigoplus_{i=1}^k K(a_i^{x},b_i^{x})$ be a persistence module with bars 
ordered non-decreasingly by lifetime, and with $n$ finite bars. Let $f\coloneqq \widehat{\rank}_d(X)$ denote the corresponding Wasserstein stable rank with respect to the distance $d=d^q_{\ns^{p,C}}$, for some $p,q\in [1,\infty]$ and a regular action $C$. The sequence $0=t_0 \le t_1 \le t_2 \le \cdots \le t_n$ such that $f$ is constant on the intervals $[t_0,t_1)$, $[t_1,t_2)$,\ldots , $[t_{n-1},t_n)$, $[t_n,\infty)$ defined in Section \ref{sec:comp_sr_contour} is enough to encode $f^{-1}$ as a finite vector $\hat{f}^{-1} \coloneqq (\hat{f}^{-1}_i)_{i \in \{ 0, \ldots , n\} }$, where $$\hat{f}^{-1}_i \coloneqq t_{n-i}=2^{\frac{1-q}{q}} \| (\ell(a_1^x,b_1^x) , \ldots ,\ell(a_{n-i}^x,b_{n-i}^x)) \|_p$$ for $i\in \{0, \ldots, n-1\}$ and $\hat{f}^{-1}_{n}=0$. Indeed, the limit value of $f$ is the number $L=k-n$ of infinite bars of $X$; for all $i\in \{0, \ldots, n\}$ we have $f^{-1}(L+i)=\hat{f}^{-1}_{i}$, and for any $y\in [L,\infty)$ the value $f^{-1}(y)$ equals $f^{-1}(L+i)$, where $i\in \{0,\ldots ,n\}$ is the largest integer such that $L+i\le y$. 

Let $Y=\bigoplus_{i=1}^l K(a_i^y,b_i^y)$ be another persistence module with bars ordered non-decreasingly by lifetime, and with $m$ finite bars. Suppose that $X$ and $Y$ have the same number of infinite bars, $L = k-n = l-m$. Let $g\coloneqq \widehat{\rank}_d(Y)$ denote the Wasserstein stable rank of $Y$ with respect to the distance $d=d^q_{\ns^{p,C}}$.
The interleaving distance between $f$ and $g$ can then be written as the $L^{\infty}$ norm of the vector 
$(\hat{f}_i^{-1} - \hat{g}_i^{-1})_{i\in \{0, \ldots ,\text{min}(n,m)\} }$ with components
\begin{equation}
\label{formula_int_distance}
\hat{f}_i^{-1} - \hat{g}_i^{-1}=  2^{\frac{1-q}{q}}( \| (\ell(a_1^x,b_1^x) , \ldots ,\ell(a_{n-i}^x,b_{n-i}^x)) \|_p -  \| (\ell(a_1^y,b_1^y) , \ldots ,\ell(a_{m-i}^y,b_{m-i}^y)) \|_p)
\end{equation}
for $i\in \{0,\ldots, \text{min}(n,m)-1\}$, 
and last component  
\begin{equation*}
\hat{f}_{\min \{n,m\}}^{-1} - \hat{g}_{\min \{n,m\}}^{-1}=\begin{cases}
\| (\ell(a_1^y,b_1^y) , \ldots ,\ell(a_{m-n}^y,b_{m-n}^y)) \|_p & \text{if } \text{min}(n,m)=n<m \\
\| (\ell(a_1^x,b_1^x) , \ldots ,\ell(a_{n-m}^x,b_{n-m}^x)) \|_p & \text{if } \text{min}(n,m)=m<n \\
0 & \text{if } n=m.
\end{cases}
\end{equation*}
We observe that the considered vector encodes the function $(f^{-1} - g^{-1})$ on the  intervals $[L, L+1), \ldots, [\min(k, l), \min(k,l) + 1 )$ of length one where both $f^{-1}$ and $g^{-1}$ are constant. Since $f^{-1}$ and $g^{-1}$ are nonincreasing and for $x \geq \min(k,l)$ either $f^{-1}(x) = 0$ or $g^{-1}(x) = 0$, it is enough to consider those intervals.

\begin{rmk}
For two persistence modules $X$ and $Y$ both of rank $k$, the complexity of computing the interleaving distance is dominated by the sorting of the bars in the respective barcode decompositions of $X$ and $Y$, since forming the vector as in  (\ref{formula_int_distance}) and computing its $L^{\infty}$ norm can be done linearly in $k$. The computational complexity of the interleaving distance between Wasserstein stable ranks is thus $O(k \log k)$.
\end{rmk}

\begin{ex}\label{example:interleaving computation}
Consider a persistence module $Y = \bigoplus_{i=1}^3 K(a_i, b_i)$ with bars ordered non-decreasingly by lifetime and $X =K(a_0, b_0) \oplus Y$ such that $\varepsilon \coloneqq \ell(a_0,b_0)\leq \ell(a_1,b_1)$. 
By using the formula (\ref{formula_int_distance})
and observing that
\[
\| (\ell(a_0,b_0) , \ldots ,\ell(a_{i},b_{i}))\|_p -  \| (\ell(a_1,b_1) , \ldots ,\ell(a_i,b_i)) \|_p\leq \ell(a_0,b_0)
\]
for $i\in \{1,2,3\}$ by properties (\ref{eq:monotonicity_p}) and (\ref{eq:concatenate_p}) of $p$-norms, we see that the interleaving distance between  $\widehat{\rank}_d (X)$ and $\widehat{\rank}_d (Y)$ with $d=d^q_{\ns^{p,C}}$ is given by $2^{\frac{1-q}{q}} \varepsilon$.
Note that by Proposition
\ref{prop:closest_rank_C}
 we know $d^q_{\ns^{p,C}}(X,Y) =2^{\frac{1-q}{q}} \textstyle{\left\| K(a_0,b_0) \right\|_{p,C} } = 2^{\frac{1-q}{q}} \varepsilon$. Therefore in this case the interleaving distance between stable ranks with respect to Wasserstein distance coincides with the Wasserstein distance between $X$ and $Y$. Note however that this is not always the case.   
The Wasserstein stable ranks of $X$ and $Y$ with respect to $d^q_{\ns^{p,C}}$, with parameters $q=1$, $p=2$ and $C$ the standard contour, are shown in Figure \ref{fig:inv_stab_illus}, together with  their ``inverse'' functions which are used for the computation of the interleaving distance.

\begin{figure}[h]
     \centering
         \includegraphics[width=1\textwidth]{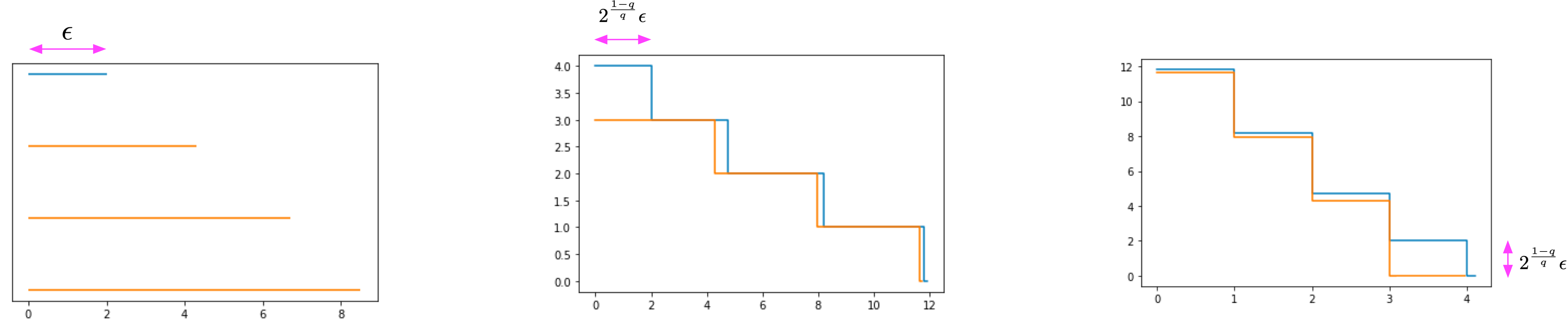}
        {\caption {Schematic representation of the computation of the interleaving distance in Example \ref{example:interleaving computation}.
        {\textbf{Left:} Barcode decomposition of $Y$ in orange and bar $K(a_0,b_0)$ in blue. \textbf{Middle:} Stable ranks computed with standard contour, $q=1$ and $p=2$. The functions $\widehat{\rank}_d(X)$ and $\widehat{\rank}_d(Y)$ are represented in blue and orange, respectively. \textbf{Right:} Inverse stable ranks for the computation of interleaving distance, with $\widehat{\rank}_d^{-1}(X)$ in blue and $\widehat{\rank}_d^{-1}(Y)$ in orange. The interleaving distance between stable ranks can be computed as $\|\widehat{\rank}_d^{-1}(X)-\widehat{\rank}_d^{-1}(Y)\|_{\infty}=2^{\frac{1-q}{q}} \varepsilon$, illustrated with the pink arrow.}
        \label{fig:inv_stab_illus}}}
\end{figure}
\end{ex}

Let us keep denoting $d^q_{\ns^{p,C}}$ by $d$. It follows from triangle inequality and Lemma \ref{lem:dX0} that:
$$d(X,Y)\geq 2^\frac{1-q}{q} \lvert \left\| X \right\|_p-\left\| Y \right\|_p \rvert.$$
This inequality can be refined by $$d(X,Y)\geq d_{\bowtie}(\widehat{\text{rank}}_{d}(X),\widehat{\text{rank}}_{d}(Y))\geq 2^\frac{1-q}{q} \lvert \left\| X \right\|_p-\left\| Y \right\|_p \rvert,$$
where the first inequality is given by the stability theorem of hierarchical stabilization (Proposition \ref{prop:stability_sr}) and the second inequality is provided by the characterization of interleaving distances between stable ranks in Proposition \ref{prop:interleavingPCF} and Equation \eqref{formula_int_distance} stating that the interleaving distance between the Wasserstein stable ranks of persistence modules $X$ and $Y$ is the $L^{\infty}$ norm of a vector with $0^{\text{th}}$ component $2^\frac{1-q}{q} \lvert \left\| X \right\|_p-\left\| Y \right\|_p \rvert$.

An example where the second inequality is strict is provided by Example \ref{example:interleaving computation} for $p>1$, while an example where this is an equality is provided in the case $Y=0$ by Lemma \ref{lem:dX0}. A simple example in which the first inequality is strict is provided instead by $X=K(0,1)$, $Y=K(0,2)$ and $q=2$. 

\begin{rmk}
\label{rmk:stability_Lp}
Since stable ranks are measurable functions $[0,\infty) \to [0,\infty)$, there are many pseudometrics to compare them other than the interleaving distance $d_{\bowtie}$. In particular, one can consider the standard $L^p$-pseudometrics, here denoted by $d_p (f,g)\coloneqq  \left( \int_{0}^{\infty} \lvert f(t)-g(t)\rvert^{p} \,dt \right)^{\frac{1}{p}}$. 
As shown in \cite[Prop. 2.1]{chacholski2020metrics}, the stability theorem of hierarchical stabilization 
implies the following bounds for $d_p$:
\[
c \, d(X,Y)^{\frac{1}{p}}\geq d_{p}(\widehat{\text{rank}}_{d}(X),\widehat{\text{rank}}_{d}(Y)),
\]
for any persistence modules $X$ and $Y$, where $c\coloneqq \max \{ \rank (X), \rank (Y)\}$ and $d$ denotes any pseudometric between persistence modules.  In this article we have chosen to work with the interleaving distance between Wasserstein stable ranks because of the strong stability result, 
expressed as a 1-Lipschitz condition. Lipschitz stability for Wasserstein distances other than $W_1$ can not be obtained for example by considering linear representations of persistence diagrams \cite{hofer2017deep, adams2017persistence,chen2015statistical,kusano2017kernel,reininghaus2015stable} as proved in Theorem 6.3 in \cite{skraba2020wasserstein}. 
The trade-off between stability and the possibility of exploiting a Banach or Hilbert space structure is still to be explored.
\end{rmk}

\subsection{Metric learning}
\label{sect:metriclearning}

We have defined distances $d^q_{\ns^{p,C}}$ between persistence modules, parametrized by $q$, $p$ and by a contour $C$, and computable stable rank invariants with corresponding interleaving distances. These distances can be pulled back to compare persistence modules in $\Tame$ via the function $\widehat{\rank}_{d}$, with $d=d^q_{\ns^{p,C}}$.
Recalling that the stable ranks depend on the pseudometric $d^q_{\ns^{p,C}}$, we now turn to the question of how to choose $p$ and $C$. The optimization of the parameter $q$ is not relevant, since it determines a constant multiplicative factor to the distance of each pair of persistence modules. We thus fix $q=1$ for a direct comparison with the original framework of noise systems.

For brevity, we write $d \coloneqq  d^1_{\ns^{p,C}}$ and $d_{\bowtie,p, C} (X,Y) \coloneqq  d_{\bowtie}(\widehat{\rank}_{d}(X),\widehat{\rank}_{d}(Y))$.
The field of metric learning provides a variety of loss functions suited for different machine learning problems. For example, if we consider a simple binary classification problem we have a dataset of persistence modules $\{ X_i \}_{i \in I}$ and the index set $I$ is partitioned into two sets $A$ and $B$, to represent the labeling. For this problem, a loss function (from \cite{zhao2019learning}), designed to yield small intra-class distances and large inter-class distances can be formulated as:

\begin{equation}
\label{eq:metriclearningloss}
    \mathcal{L} = 
    \frac{ \sum_{i,j \in A} (d_{\bowtie, p, C}(X_i, X_j))^2}{ \sum_{i \in A, j \in I} (d_{\bowtie, p, C}(X_i, X_j))^2 } + \frac{ \sum_{i,j \in B} (d_{\bowtie, p, C}(X_i, X_j))^2 }{ \sum_{i \in B, j \in I} (d_{\bowtie, p, C}(X_i, X_j))^2 } 
\end{equation}

In order to proceed we need to choose a family of contours that is practically searchable when minimizing the loss function above. We work with contours of distance type which are parametrized by densities (see Section \ref{subsec:contours}). In turn, in order to use gradient optimization methods, we want the densities to be parametrized by a finite real-valued parameter vector. To this aim we choose as densities unnormalized Gaussian mixtures $f(x) = \sum_{i=1}^{k} \lambda_i \mathcal{N}(x \mid \mu_i, \sigma_i)$ for some chosen $k$, where $\mathcal{N}$ is Gaussian with mean $\mu_i$ and standard deviation $\sigma_i$, and $\lambda_1 = 1$. 

In summary, the metric learning problem amounts to minimizing the loss function with respect to a parameter vector $\theta \in \mathbb{R}^{3k}$, i.e.\ $\theta = (\mu_1, \ldots, \mu_k, \sigma_1, \ldots, \sigma_k, \lambda_2, \ldots, \lambda_k, p)$, designed to learn conjointly the parameter $p$ and the parameters of the contour of the algebraic Wasserstein distance. The loss function is a simple function of the pairwise interleaving distances between Wasserstein stable ranks of persistence modules in the dataset.
As can be seen in  Proposition \ref{prop:interleavingPCF} and the expression (\ref{formula_int_distance}), the interleaving distance between stable ranks is the $L^{\infty}$ norm of differentiable functions  with respect to $\theta$ and is therefore differentiable almost everywhere with respect to $\theta$, implying the same behavior for the loss function. 
Hence the metric learning problem is amenable to gradient-based optimization methods such as gradient descent.


\section{Examples of analyses with Wasserstein stable ranks}
\label{sec:examples_analyses}
In a first experiment, we show how varying the parameter $p$ affects the distance space of the Wasserstein stable ranks and can serve as a way to weight the importance of long bars versus short bars, for a set of synthetic persistence modules. In a second experiment, we illustrate on a real-world dataset how learning the parameter $p$ together with the parameters of a contour can lead to more discriminative  Wasserstein stable ranks in a classification problem.

\subsection{Synthetic data}
\label{subsec:synthetic}
A straightforward way to apply persistent homology in the context of computer vision is to construct a complex (e.g.\ cubical complex) from the grid of pixels constituting an image. The complex is then filtered based on the grayscale intensity of the pixels (or based on the color channels for color images).

It is easy to see that, in this context, what should be considered as signal versus noise  in a barcode representation of the data is highly dependent on the application. For example, for classification of handwritten digits from the MNIST dataset \cite{garin2019topological, turkevs2021noise} the dominant topological features are often the most discriminative (for instance the existence of a 1-dimensional cycle may be enough to distinguish between digits $0$ and $1$). On the other hand, in biomedical imaging \cite{chung2018topological, qaiser2019fast} pathological states can translate into images with irregularities or lack of homogeneity, associated with high numbers of short-lived components as observed in \cite{garside2019topological}.

Inspired by these applications, we construct two much simpler synthetic datasets of images and associated persistence modules, with the goal of illustrating the effect of choosing the parameter $p$ when using Wasserstein stable ranks. 
The parameter $q$ is set to $1$ and the contour is fixed to be the standard contour. In other words, we study the effect of the parameter $p$ on how the function $\widehat{\rank}_d$, with $d=d^1_{\ns^p}$, maps persistence modules onto the space of stable ranks, endowed with the interleaving distance. Each dataset is composed of $100$ images together with their class label, A or B. Each image is composed of one block of high-intensity pixels and a number of blocks of low-intensity pixels (while the size of the pixel blocks does not have a direct impact on the following persistent homology analysis, the high-intensity block is made larger for visual clarity, see Figures \ref{fig:artificial1_imgs_bars}, \ref{fig:artificial2_imgs_bars}). The images are represented as cubical complexes on which super-level set filtration is performed and we analyze the $H_0$ barcodes obtained from this process. Since we use pixel intensity $[0, 255]$ and super-level sets are used, the resulting filtration scale is $[255, -\infty)$. This is capped to the minimum pixel value, $0$, and transformed as $255 - x$ to obtain a filtration scale $[0, 255]$ as can be seen in the barcodes in Figures \ref{fig:artificial1_imgs_bars}, \ref{fig:artificial2_imgs_bars}.

\begin{itemize}
    \item In Dataset 1 the pixels in the high-intensity block have slightly higher intensity in images from class A (uniformly distributed between $245$ and $255$) compared to images of class B (between $200$ and $210$). The low-intensity blocks however follow the same distribution for images of both classes (the number of blocks is uniformly distributed between $50$ and $100$ and the intensity is between $1$ and $10$). Sample images and barcodes are shown in Figure \ref{fig:artificial1_imgs_bars}.
    \item In Dataset 2 on the other hand, the intensity of the high-intensity blocks follows the same distribution for both classes (uniformly distributed between $100$ and $255$). The number of low-intensity blocks however follows a different distribution for Class A (between $20$ and $30$) and Class B (between $120$ and $130$). Their intensity is the same for both classes (between $1$ and $10$). Sample images and barcodes are shown in Figure \ref{fig:artificial2_imgs_bars}.
\end{itemize}

\begin{figure}[ht]
     \centering
         \includegraphics[width=0.72\textwidth]{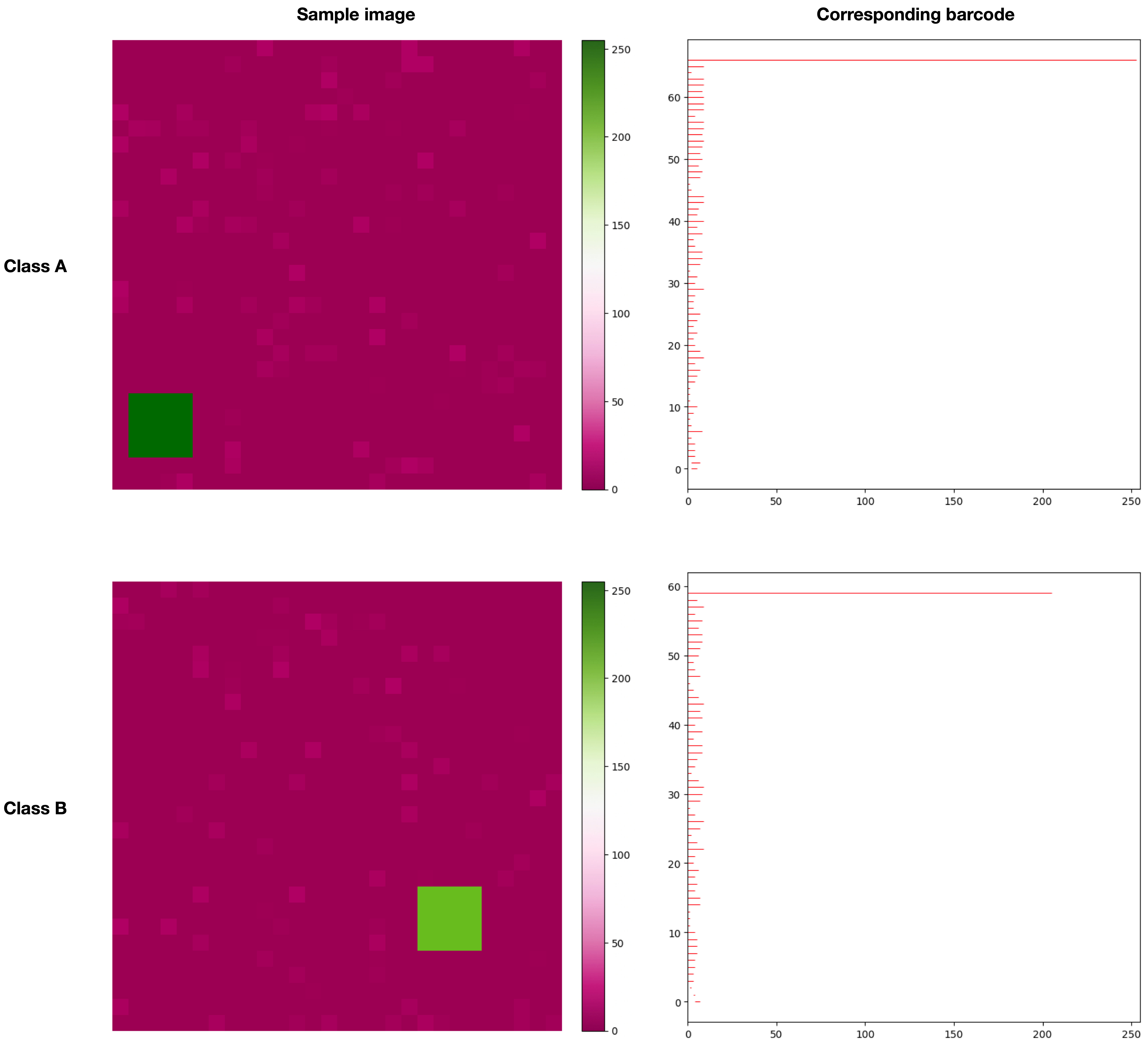}
        \caption{Dataset 1. \textbf{Left:} Sample images from classes A and B. \textbf{Right}: $H_0$ barcodes corresponding to the sample images.}
        \label{fig:artificial1_imgs_bars}
\end{figure}

\begin{figure}[ht]
     \centering
         \includegraphics[width=0.72\textwidth]{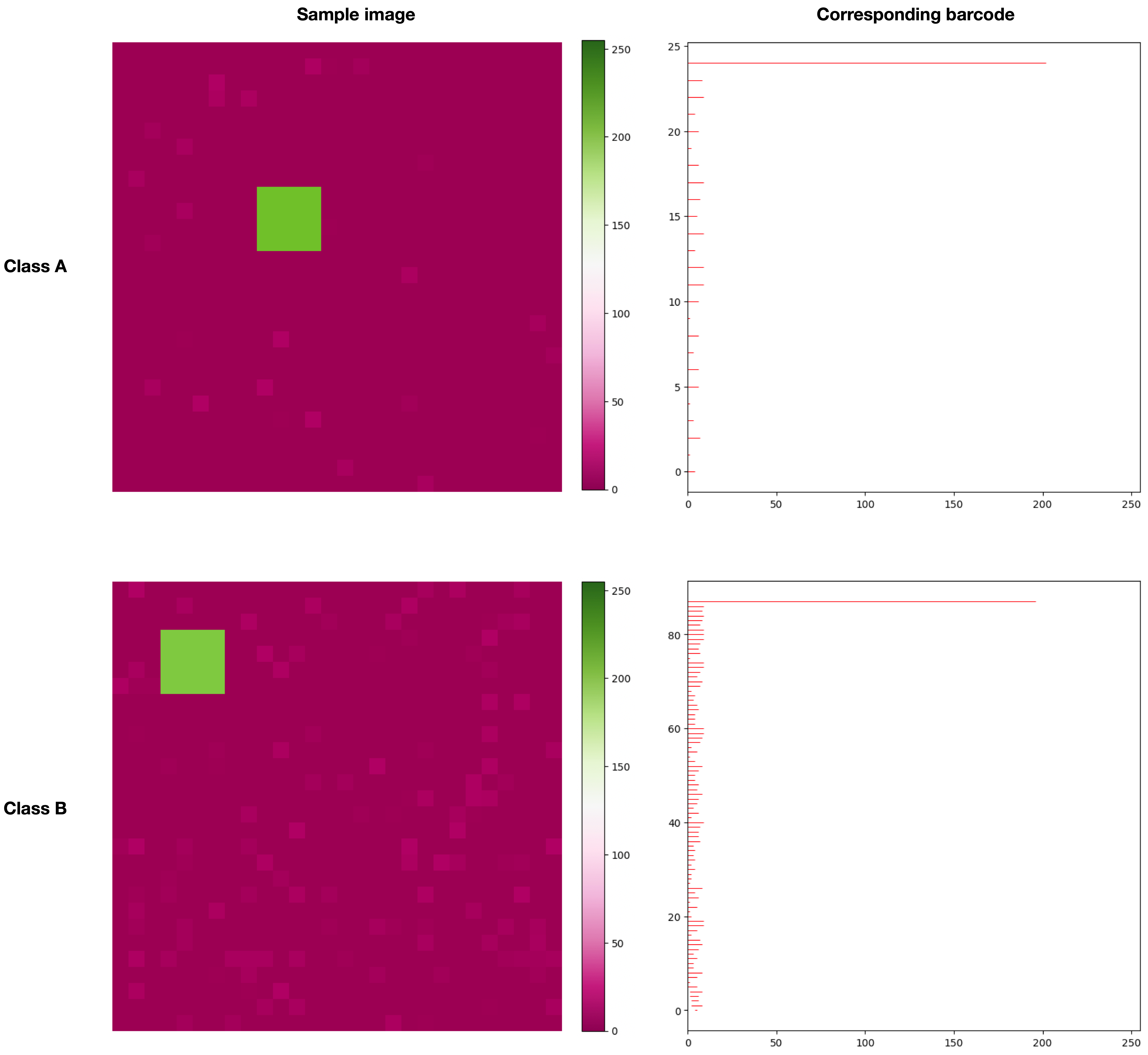}
        \caption{Dataset 2. \textbf{Left:} Sample images from classes A and B. \textbf{Right}: $H_0$ barcodes corresponding to the sample images.}
        \label{fig:artificial2_imgs_bars}
\end{figure}

This construction induces distributions of barcodes where barcode features (i.e. length of longest bar or number of bars in a given length range) are expected to be statistically distinct or indistinguishable.
In terms of the barcodes, for Dataset 1 the signal is by construction the single dominant topological feature (the long bar, whose length follows statistically different distributions between classes) while the noise is composed of the numerous short bars (corresponding to low intensity blocks, for which the number and intensity follows the same distribution in both classes).

In accordance with the intuition, for Dataset $1$, choosing a value of $p = \infty$ when generating the stable ranks effectively ``denoises'' the barcodes and organizes the space of Wasserstein stable ranks in a way where stable ranks of the same class are close to each other in interleaving distance but far from elements of the other class. Stable ranks corresponding to $p=1$ however fail to organize the corresponding distance space in this clear-cut way, being too sensitive to the noisy short bars in the barcodes. To illustrate this effect, in Figure \ref{fig:artificial1_clust} we show the hierarchical clustering (with average linkage, similar results were observed for complete and single linkage) corresponding to the distance spaces of Wasserstein stable ranks for $p=1$ and $p=\infty$.

On the contrary, for Dataset 2 the signal is by construction the number of short bars (numbers which follow statistically different distributions) while the noise is the single long bar (generated by blocks following the same distribution for both classes). In this case a choice of $p=1$ organizes the space of stable ranks such that elements of the same class cluster together, while $p = \infty$, being too sensitive to the (for this dataset) noisy long bar, fails to do so. This is illustrated in Figure \ref{fig:artificial2_clust}.

\begin{figure}[ht]
    \centering
    \includegraphics[width=0.48\textwidth]{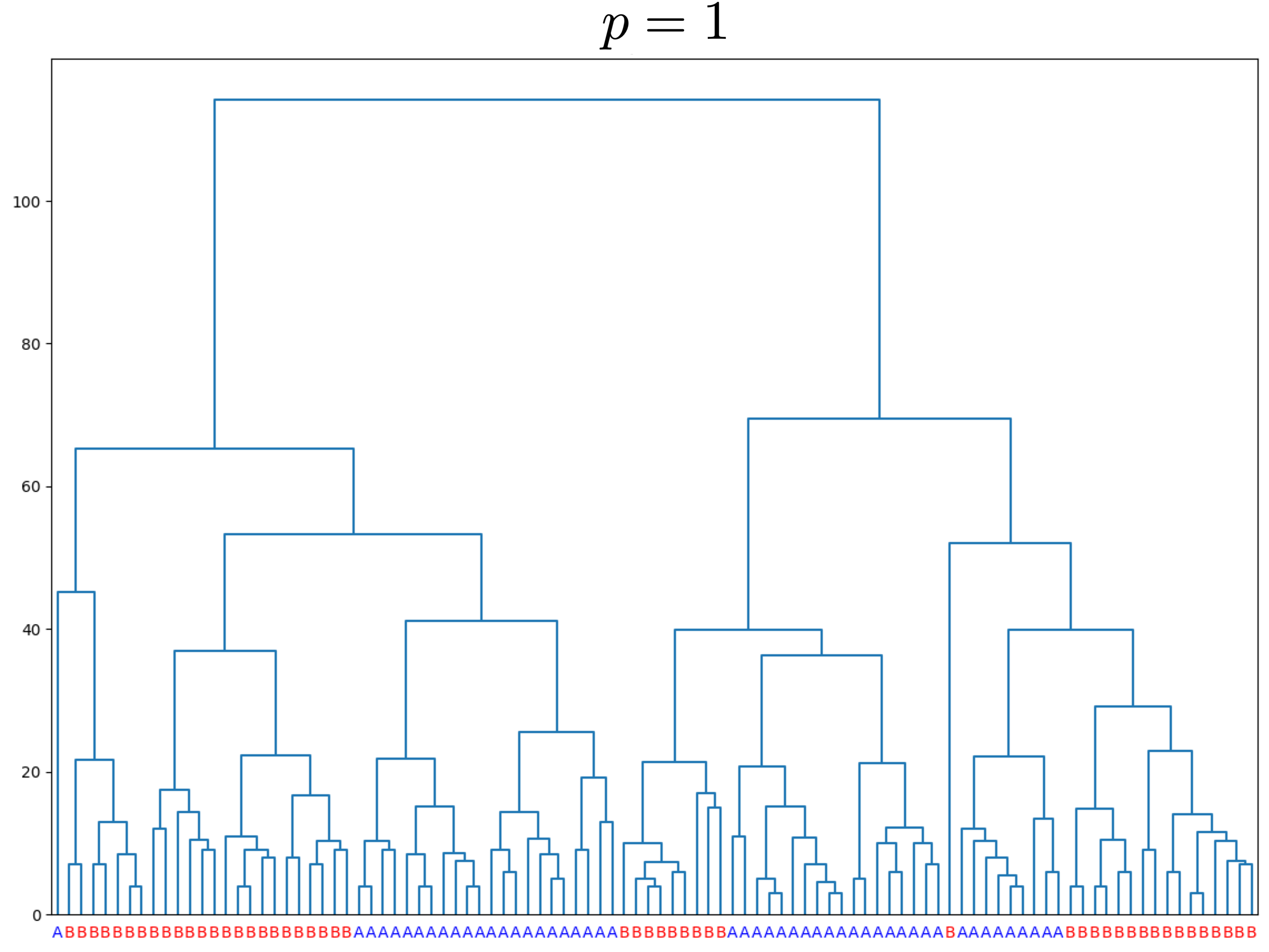}
    \includegraphics[width=0.48\textwidth]{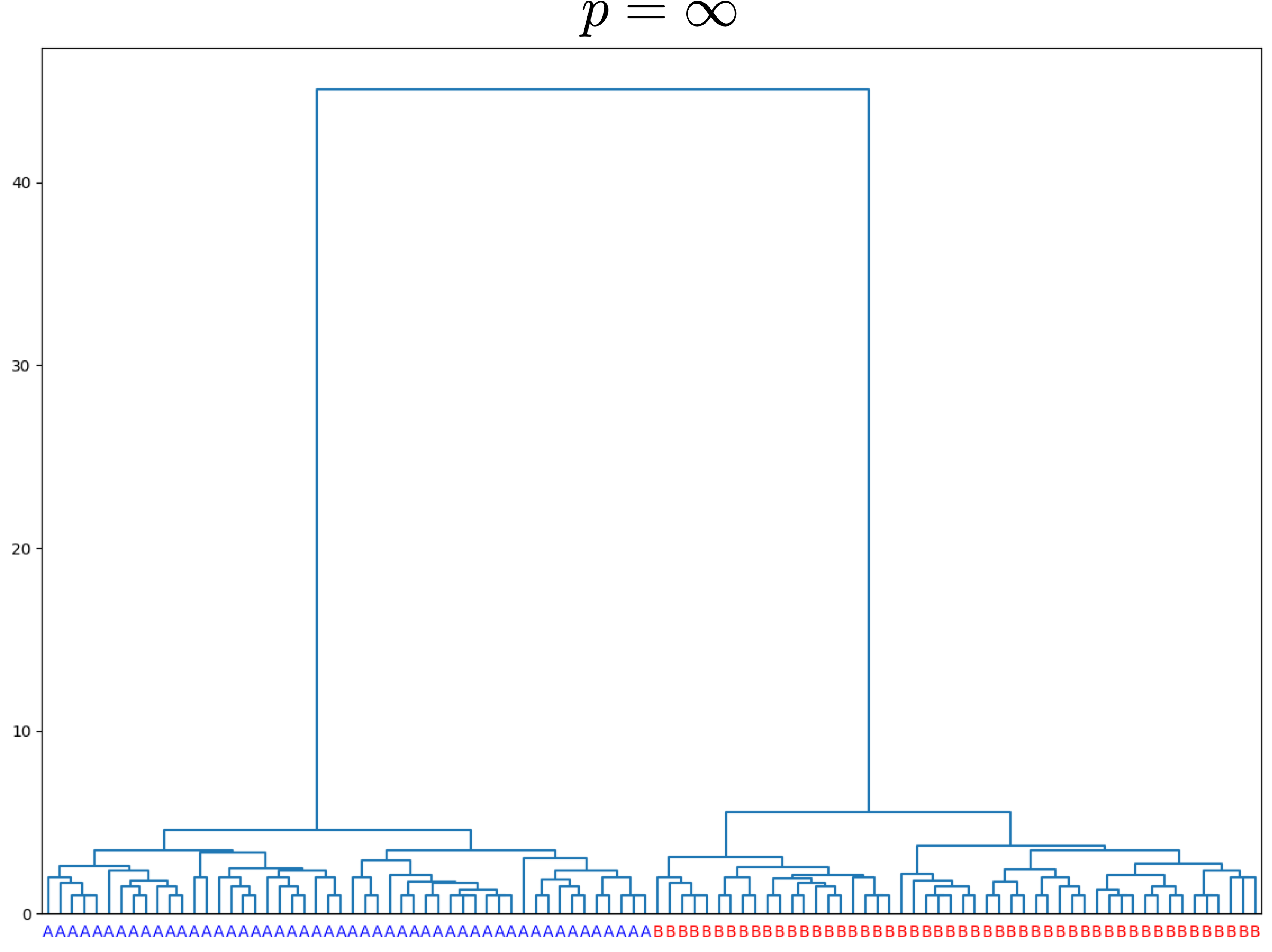}
    \caption{Dataset 1. Hierarchical clustering on the Wasserstein stable ranks for $p=1$ (left) and $p=\infty$ (right) with respect to the interleaving distance. The leaves (stable ranks in the dataset) are labeled and colored according to their class.}
    \label{fig:artificial1_clust}
\end{figure}

\begin{figure}[ht]
    \centering
    \includegraphics[width=0.48\textwidth]{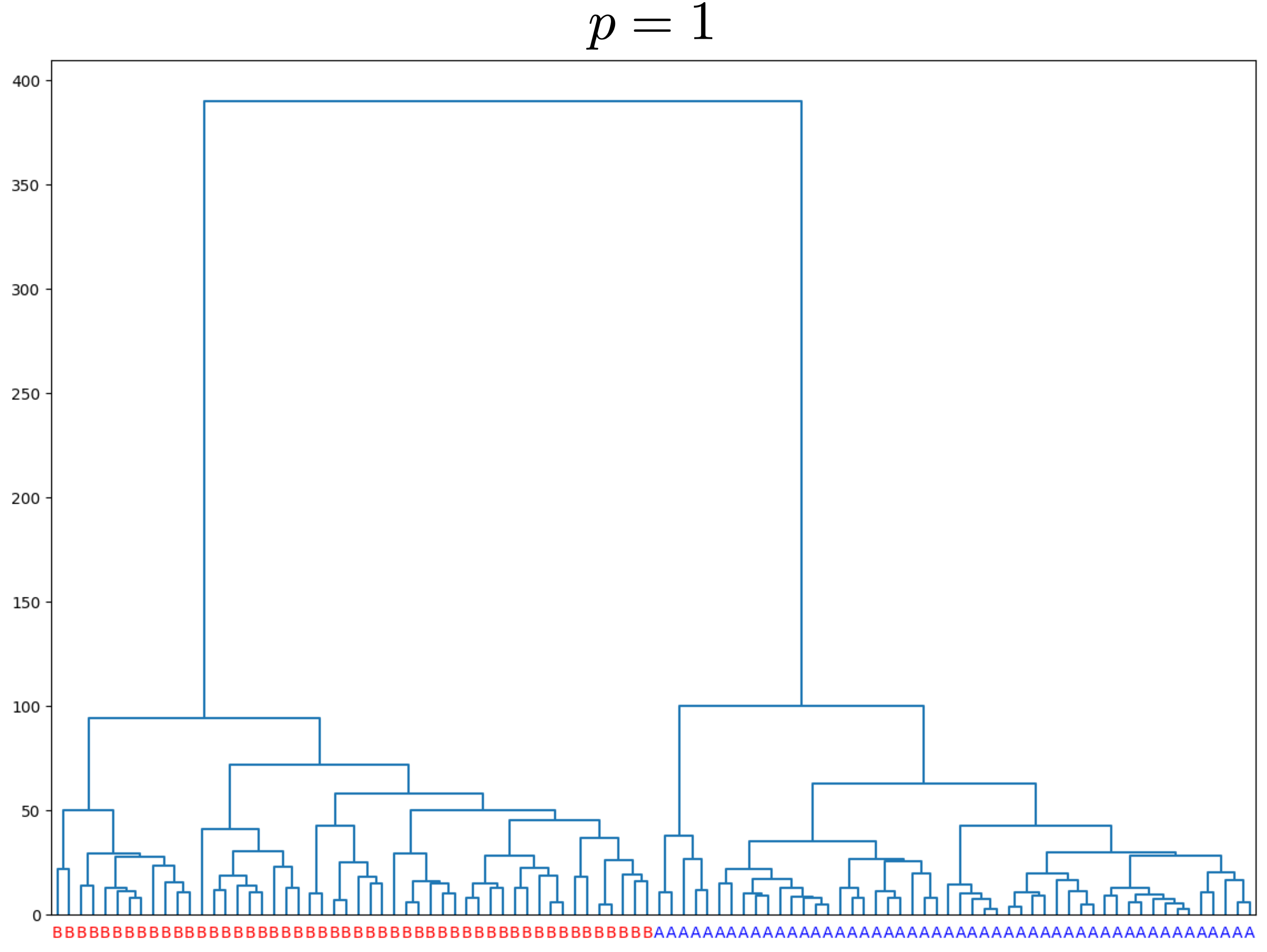}
    \includegraphics[width=0.48\textwidth]{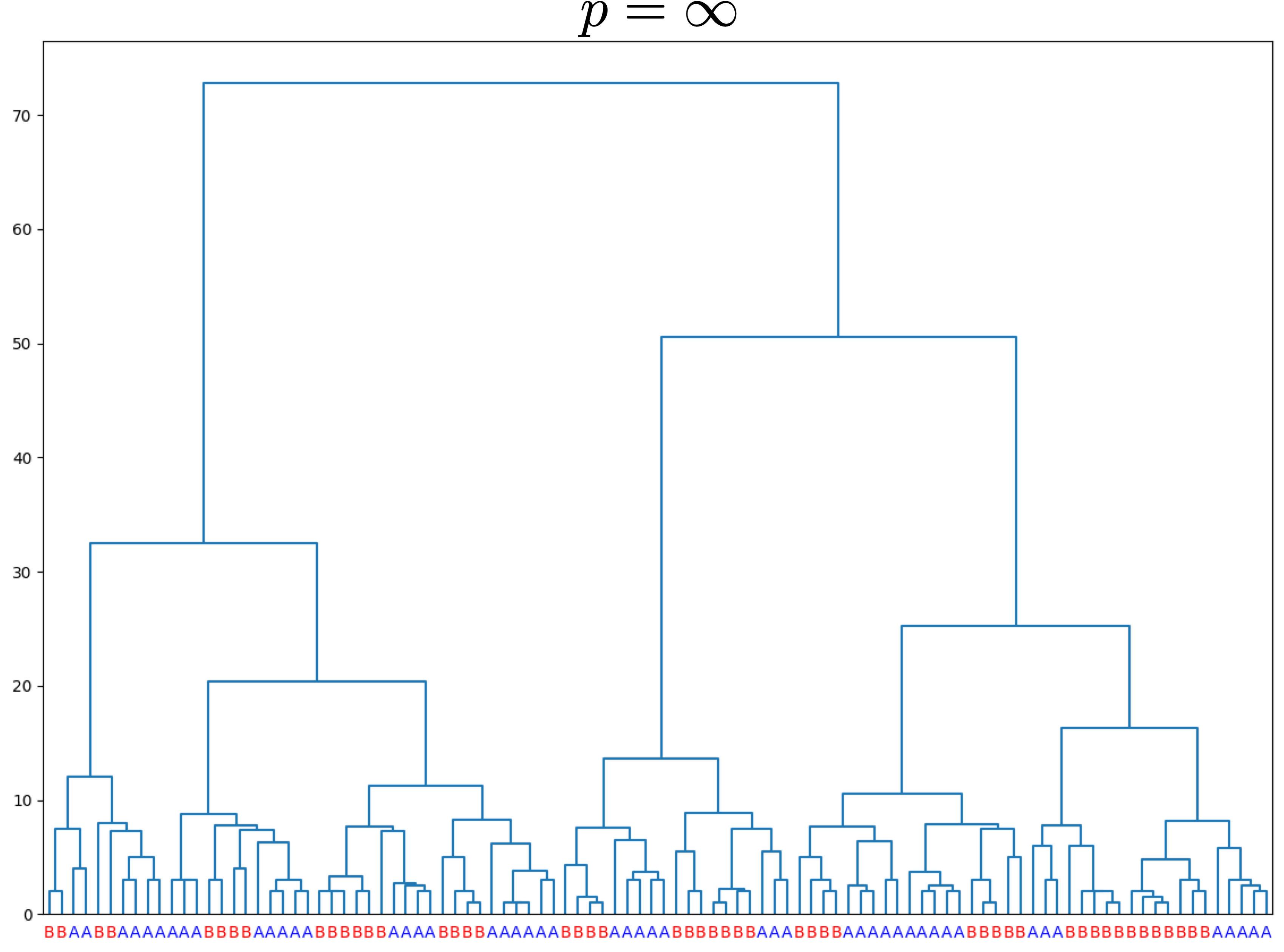}
    \caption{Dataset 2. Hierarchical clustering on the Wasserstein stable ranks for $p=1$ (left) and $p=\infty$ (right) with respect to the interleaving distance. The leaves (stable ranks in the dataset) are labeled and colored according to their class.}
    \label{fig:artificial2_clust}
\end{figure}

While the effect of changing $p$ on the structure of the distance space is clear for the parameters used to generate our artificial datasets, some class-based structure remains  on a small scale, also when choosing $p=\infty$. By increasing the amount of noise it is however possible to induce e.g.\ a nearest neighbor classifier to perform arbitrary poorly for the $p = \infty$ while still distinguishing the classes for $p=1$ (and vice versa for Dataset 1). 

The choice of the parameter value $p$, which we have demonstrated can have a large impact, is essentially related to the underlying distance between persistence modules. Using Wasserstein-stable invariants however has computational advantages, facilitates learning the right parameters for a particular problem and allows for a richer use of machine learning methods as we illustrate in the next subsection on a real-world dataset.

\subsection{Brain artery data}
\label{subsec:brain_artery}
In \cite{bullitt2010effects} a dataset of brain artery trees corresponding to 97 subjects aged 18 to 72 is introduced. Each data point is modeled as a tree embedded in $\mathbb{R}^3$. In \cite{bendich2016persistent} the dataset is further analyzed with 
 persistent homology. To be able to apply a sublevel set filtration on the tree, a real-valued function is defined on the vertices as the height of the vertex in the 3D-embedding. This is extended to a function on the edges by taking the maximum value of the weights of the vertices connected by the edge. After applying persistent homology, each tree is represented by a vector containing the sorted lengths of the $100$ longest bars in a barcode decomposition of the corresponding persistence module. This feature is further used to demonstrate, among other things, an age effect of brain artery structure, by showing that the projection of the vectors on the first principal component of the dataset is correlated with age.

The authors note that using vectors of sorted length was computationally more feasible than computing Wasserstein distances between the persistence diagrams and they are more amenable to statistical analysis. In addition, the authors observed that it was not necessary to use the whole vector of lengths to establish the correlation and in fact the topological features of medium length, rather than the longest ones, were the most discriminatory. 

Analyzing the dataset with stable ranks offers computational and statistical advantages. Moreover, for this problem where the discriminative information is not contained in the most persistent feature, considering other distances than the bottleneck ($p = \infty$) and more generally tuning the parameter $p$ might be beneficial. Finally, combining the tuning of the parameter $p$ with a contour might increase the power of the method.
Indeed the parameter $p$ and the contour, intuitively are related to different features of a persistence barcode: while the parameter $p$ globally weights the importance of long versus short bars as illustrated in Section \ref{subsec:synthetic}, the contour 
 allows to focus on the most informative  filtration scales.

While we also study age effects of brain artery structure, we choose to binarize the problem by creating two classes of similar size: \textit{young} ($age < 45$, 50 subjects) and \textit{old} ($age \geq 45$, 47 subjects) and treat the problem as a classification.

We start by studying the effect of varying $p$ alone. We compute the distances between Wasserstein stable ranks with standard contour. We classify the samples in the distance space thus obtained by using the $k$-nearest neighbors algorithm \cite{pedregosa2011scikit} (the parameter $k$ is chosen in a cross-validation procedure). Repeating this for various values of $p$ we observe a difference in the resulting accuracy, plotted in Figure \ref{fig:acc_different_p}, with the highest values obtained for $p$ in the medium range ($2-3$).
 This is in line with the conclusion in \cite{bendich2016persistent} that the highest persistent features alone have a small distinguishing power, while medium sized bars reflect variations in brain artery trees within subject of different ages.
In \cite{bendich2016persistent} only the length of bars in a barcode is used to compare barcodes of different classes. With our features parametrized by $p$ and a contour $C$, we can take into consideration both the length of bars and their position in the parameter space. 

\begin{figure}
     \centering
         \includegraphics[width=0.55\textwidth]{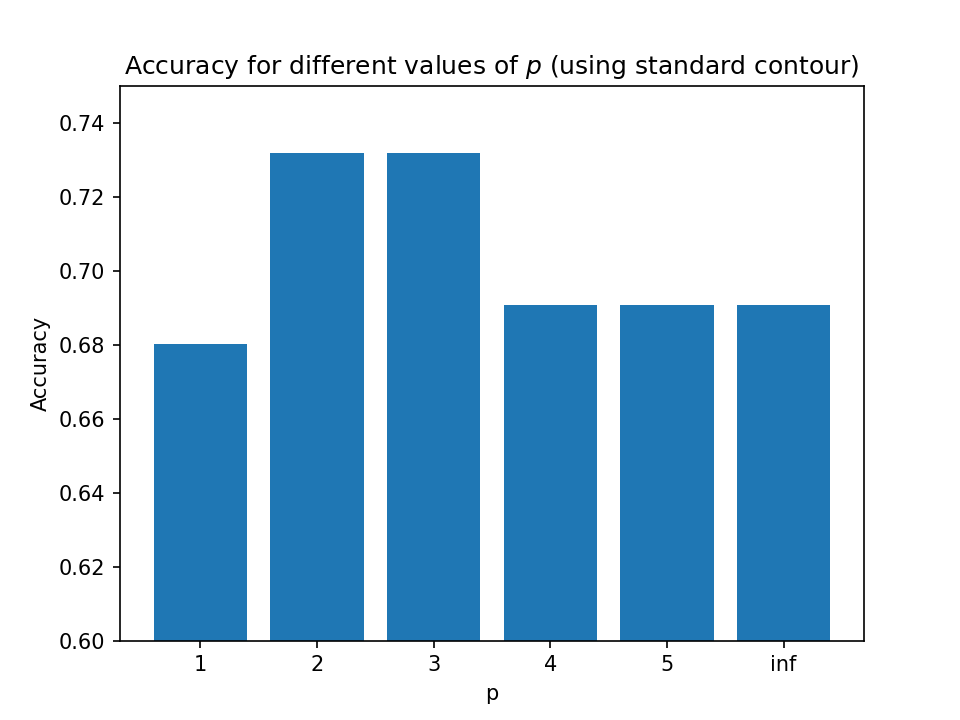}
         \caption{Accuracy on the brain artery problem using distances between stable ranks and KNN, for standard contour and different values of $p$.}
        \label{fig:acc_different_p}
\end{figure}

We therefore next turn to the problem of learning the contour of the stable ranks as well. We use the metric learning method described in Section \ref{sect:metriclearning}. Using 
leave-one-out cross-validation (LOOCV), for each training fold we learn the metric that optimally separates training samples from the two classes by minimizing the loss defined in (\ref{eq:metriclearningloss}). We then classify using KNN in the obtained distance space.

For metric learning, the contours are parametrized by densities which are unnormalized Gaussian mixtures with two components. The loss function is implemented in PyTorch \cite{paszke2019pytorch}. After a random initialization of the parameters, projected gradient descent (to respect the constraints $p \geq 1, \lambda_i, \sigma_i > 0$) 
with momentum is used to achieve a lower loss. An example of an optimization on a training fold over $25000$ iterations is shown in Figure \ref{fig:brain_parameters_prog}. The average and standard deviation of the optimized parameters, over the LOOCV folds can instead be found in Table \ref{table:converged_values}.
\begin{figure}
     \centering
         \includegraphics[width=1\textwidth]{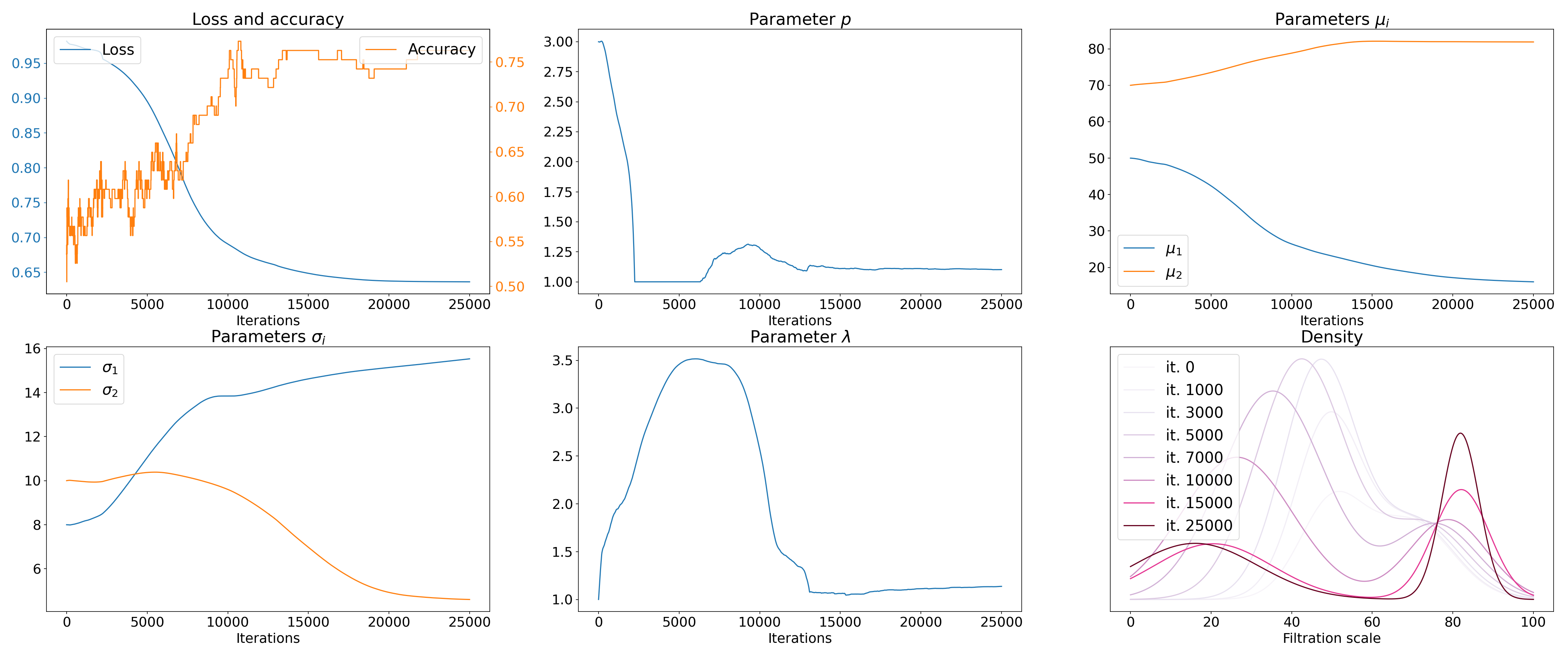}
        \caption{Results for one example run of the metric learning optimization for Wasserstein stable ranks (see Section \ref{sect:metriclearning}) over $25000$ iterations. \textbf{Top Left:} Progression of the loss and the KNN training fold accuracy over the iterations. \textbf{Top Middle, Top Right, Bottom Left, Bottom Middle}: Progression of the parameters in $\theta=(\mu_1,\mu_2,\sigma_1,\sigma_2,\lambda_2, p)$ parametrizing Wasserstein stable ranks: $p$, mean $\mu_i$, standard deviation $\sigma_i$ and $\lambda_2$ respectively over the iterations. \textbf{Bottom Right}: Density at different iterations.}
        \label{fig:brain_parameters_prog}
\end{figure}

\begin{table}[ht]
\centering
\label{tab:parameter_estimates}
\begin{tabular}{ccccccc}
\toprule
$p$ & $\mu_1$ & $\mu_2$ & $\sigma_1$ & $\sigma_2$ & $\lambda$ \\ 
\midrule
$1.10 \pm 0.04$ & $16.45 \pm 0.75$ & $81.88 \pm 0.27$ & $15.51 \pm 0.45$ & $4.55 \pm 0.24$ & $1.15 \pm 0.05$ \\
\bottomrule
\end{tabular}
\caption{Converged values for the parameters of the metric learning problem (mean and standard deviation over the LOOCV folds).}
\label{table:converged_values}
\end{table}

The metric learning is effective in finding distances that improve the classification performance: running the optimization problem not only decreases the loss  but also increases the corresponding classification accuracy (as is seen in Figure \ref{fig:brain_parameters_prog} in the top left plot), reaching $76.3 \%$ with the parameter values summarized in Table \ref{table:converged_values}. 

This is an improvement compared to the accuracies obtained at random initialization (between $44.3\%$ and $71.1\%$ for 10 random initializations of the parameters in Table \ref{table:converged_values}), showing the benefit of learning, but also compared to the results obtained when only varying $p$ and considering the standard contour in Figure \ref{fig:acc_different_p}. It is thus when we learn both $p$ and the contour that the best loss and corresponding classification accuracy is achieved.

\begin{figure}[h!]
    \centering
    \includegraphics[width=0.48\textwidth]{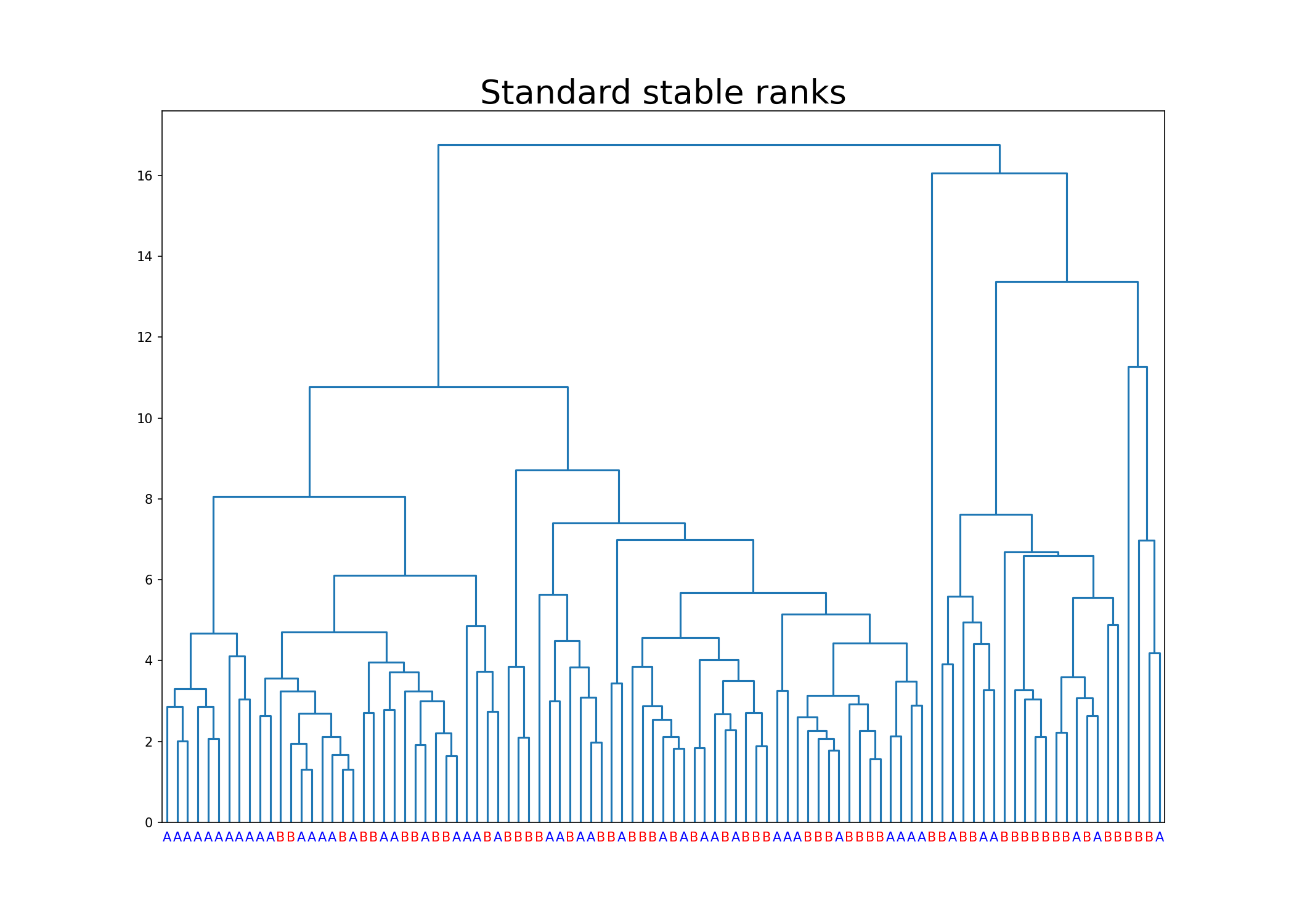}
    \includegraphics[width=0.48\textwidth]{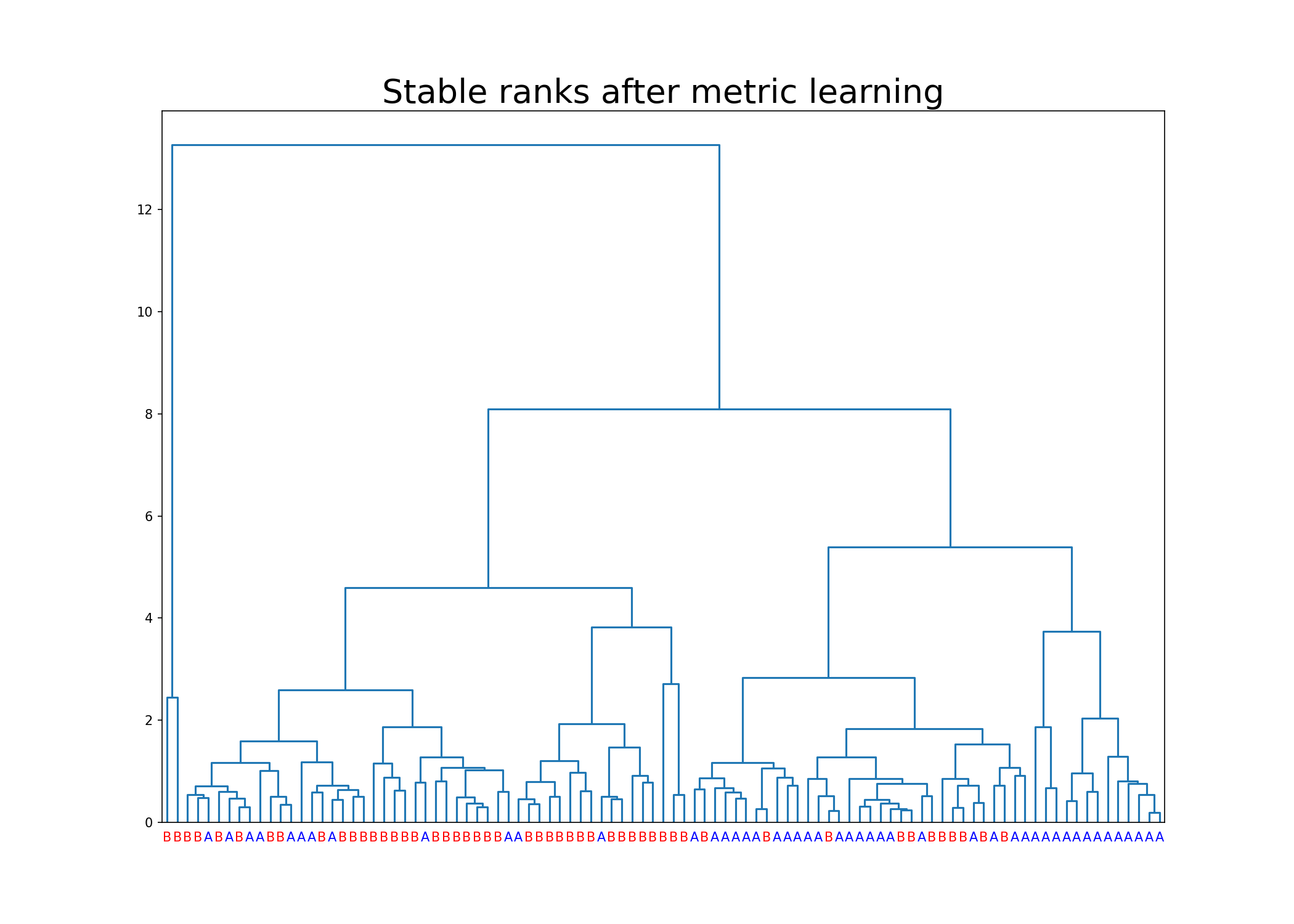}
    \caption{Hierarchical clustering on the standard stable ranks (left) and the optimized stable ranks resulting from the metric learning problem (right) with respect to the interleaving distance. The leaves (stable ranks in the dataset) are labeled and colored according to their class (A is $age \geq 45$, B is $age < 45$).}
    \label{fig:brain_clust}
\end{figure}

 In Figure \ref{fig:brain_clust} we illustrate the effect of the metric learning by plotting the hierarchical clustering (with average linkage) corresponding to the standard stable ranks (i.e., with $p=\infty$ and standard contour) and to the optimized stable ranks. We see that the optimized stable ranks (with the exception of two outliers) group into two clusters: one with a majority of class A and the other with a majority of class B, while the pattern for standard stable ranks is less clear.

The optimal parameters found with the metric learning method are of interest because they allow to construct a distance space in which machine learning methods can be carried out, but they are also interpretable: they contain information about which features of the dataset are important to distinguish the two classes.

\begin{figure}
     \centering
         \includegraphics[width=0.8\textwidth]{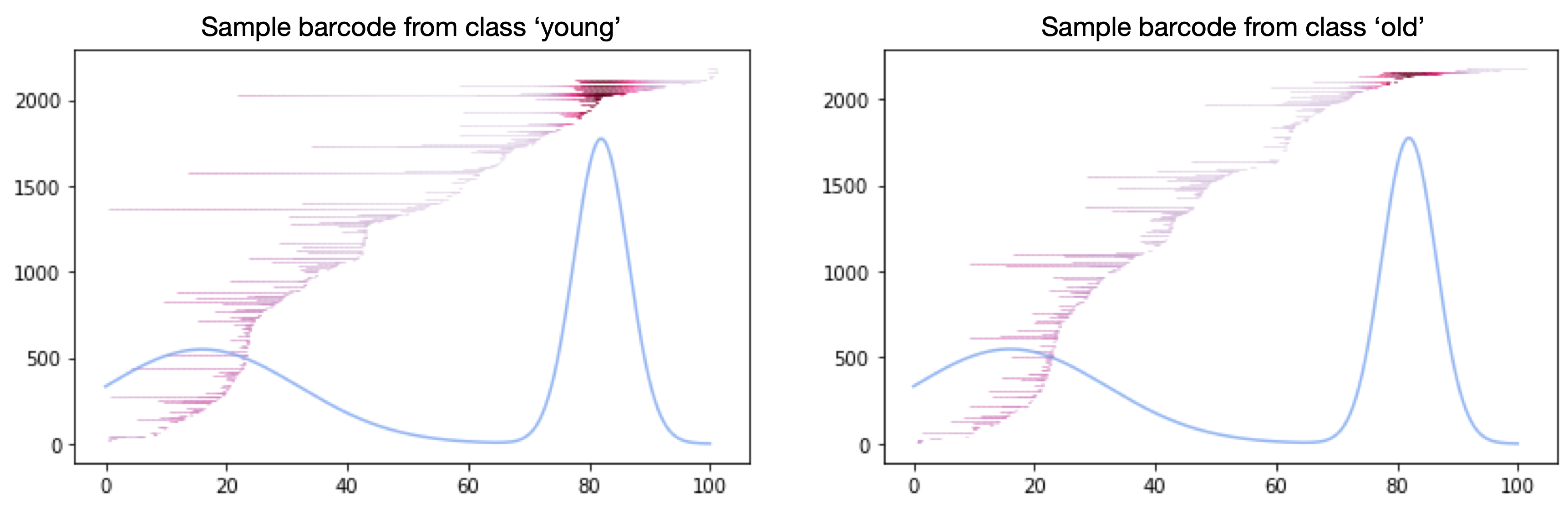}
        \caption{Sample barcodes from the two classes with superposed learned density. Bars are colored according to the density. 
        }
        \label{fig:brain_barcodes_density}
\end{figure}

This is illustrated in Figure \ref{fig:brain_barcodes_density} where two sample barcodes – one from each class – are displayed with the optimal density superposed and the bars colored according to the density. From the insight that some parts of the filtration scale are more important in distinguishing younger from older subjects, one may pursue the analysis by looking for characteristics of bars in that region of the barcode. One can also take the analysis a step further by looking at the object from which the filtered simplicial complex was created. In our case, since the filtration scale corresponds to the height (z-coordinate) in the 3D-embedding of the brain artery tree, one may for example investigate whether differences in brain artery between subjects of different ages in this particular region carries a biological meaning.


\section*{Declarations}

\textbf{Conflict of interest.}
On behalf of all authors, the corresponding author states that there is no conflict of interest.


\bibliographystyle{alpha}
\bibliography{BiblioDatabase}


\bigskip
\bigskip
{\footnotesize
\noindent
\textsc{Department of Mathematics, KTH, S-10044 Stockholm, Sweden}\\
\texttt{ \{jensag,guidolin,isaacren,scola\}@kth.se}\\
Corresponding author: Martina Scolamiero (\texttt{scola@kth.se})
}

\end{document}